\documentclass[10pt,reqno]{amsart}

\usepackage{amsthm,amsmath,amstext,amssymb,amscd,euscript, mathrsfs, dsfont,multicol,times,enumerate,subfig,sidecap}
\usepackage{enumerate}
\usepackage{amsmath, amssymb, amsthm}
\usepackage{mathrsfs}
\usepackage{esint}
\usepackage{xcolor}
\usepackage{mathtools}
\usepackage{bm}
\usepackage{bbm}
\usepackage{todonotes}

\usepackage[colorlinks=true, pdfstartview=FitV, linkcolor=blue, citecolor=red, urlcolor=blue]{hyperref} 

\newtheorem{thm}{Theorem}[section]
\newtheorem{corollary}[thm]{Corollary}
\newtheorem{lem}[thm]{Lemma}
\newtheorem{prop}[thm]{Proposition}
\theoremstyle{definition}
\newtheorem{defn}[thm]{Definition}

\newtheorem{rem}[thm]{Remark}

\newcommand\bE{\mathbb{E}}

\newcommand\bN{\mathbb{N}}

\newcommand\bP{\mathbb{P}}
\newcommand\bQ{\mathbb{Q}}
\newcommand\bR{\mathbb{R}}

\newcommand\bZ{\mathbb{Z}}

\newcommand\cF{\mathcal{F}}

\newcommand\cI{\mathcal{I}}

\newcommand\cL{\mathcal{L}}

\newcommand\cM{\mathcal{M}}

\newcommand\cO{\mathcal{O}}

\newcommand{\dd}{\,\mathrm{d}}

\newcommand{\p}{\partial}

\newcommand{\mysection}[1]{\section{#1}
\setcounter{equation}{0}}

\def\XXint#1#2#3{{\setbox0=\hbox{$#1{#2#3}{\int}$ }
\vcenter{\hbox{$#2#3$ }}\kern-.58\wd0}}

\makeatletter
\def\@tocline#1#2#3#4#5#6#7{\relax
  \ifnum #1>\c@tocdepth 
  \else
    \par \addpenalty\@secpenalty\addvspace{#2}%
    \begingroup \hyphenpenalty\@M
    \@ifempty{#4}{%
      \@tempdima\csname r@tocindent\number#1\endcsname\relax
    }{%
      \@tempdima#4\relax
    }%
    \parindent\z@ \leftskip#3\relax \advance\leftskip\@tempdima\relax
    \rightskip\@pnumwidth plus4em \parfillskip-\@pnumwidth
    #5\leavevmode\hskip-\@tempdima
      \ifcase #1
       \or\or \hskip 1em \or \hskip 2em \else \hskip 3em \fi%
      #6\nobreak\relax
    \dotfill\hbox to\@pnumwidth{\@tocpagenum{#7}}\par
    \nobreak
    \endgroup
  \fi}
\makeatother

\begin{document}
\title[Trace theorem to Volterra-type equations with local or non-local derivatives]
{On the trace theorem to Volterra-type equations with local or non-local derivatives}

\author[J.-H. Choi]{Jae-Hwan Choi}
\address[J.-H. Choi]{Department of Mathematical Sciences, Korea Advanced Institute of Science and Technology, 291 Daehak-ro, Yuseong-gu, Daejeon 34141, Republic of Korea}
\email{jaehwanchoi@kaist.ac.kr}

\author[J. B. Lee]{Jin Bong Lee}
\address[J. B. Lee]{Research Institute of Mathematics, Seoul National University, Gwanak-ro 1, Gwanak-gu, Seoul 08826, Republic of Korea}
\email{jinblee@snu.ac.kr}

\author[J. Seo]{Jinsol Seo}
\address[J. Seo]{School of Mathematics, Korea Institute for Advanced Study, 85 Hoegiro Dongdaemun-gu, Seoul 02455, Republic of Korea}
\email{seo9401@kias.re.kr}

\author[K. Woo]{Kwan Woo}
\address[K. Woo]{Department of Mathematical Sciences, Ulsan National Institute of Science and Technology, Unist-gil 50, Eonyang-eup, Ulju-gun, Ulsan 44919, Republic of Korea}
\email{gonow@unist.ac.kr}
\email{}

\thanks{}



\maketitle
\begin{abstract}
This paper considers traces at the initial time for solutions of evolution equations with local or non-local derivatives in vector-valued $L_p$ spaces with $A_p$ weights.
To achieve this, we begin by introducing a generalized real interpolation method.
Within the framework of generalized interpolation theory, we make use of stochastic process theory and two-weight Hardy's inequality to derive our trace and extension theorems.
Our results encompass findings applicable to time-fractional equations with broad temporal weight functions.
\end{abstract}
\tableofcontents

\mysection{Introduction and Main results}

In this article, we are interested in the trace and extension theorem for the following evolution equations in vector-valued $L_p$ spaces with $A_p$ weights:
\begin{align}
\label{eq0121_01}
    \partial_t u(t) = f(t), \quad u(0) = u_0,\\
    \label{eq0121_02}
    \partial_t^\kappa u(t) =  f(t), \quad u(0) = u_0.
\end{align}
Here, $\partial_t^{\kappa}$ stands for a (generalized) time-fractional derivative with a kernel $\kappa$ given by
$$
	\partial_t^\kappa u(t):=\partial_t \left( \int_0^t \kappa(t-s)(u(s) - u(0))~\mathrm{d}s \right).
$$
Note that $\partial_t^\kappa$ becomes the Caputo fractional derivative $\partial_t^{\alpha}$ for $\kappa(t) =t^{-\alpha}/\Gamma(1-\alpha)$ with $\alpha\in(0,1)$, where $\Gamma$ is the Gamma function.

Our trace theorems also hold for the following Volterra-type equations:
\begin{equation}
\label{23.08.14.13.29}
    \int_0^t(u(s)-u_0)\mathrm{d}s=\int_0^t\varkappa(t-s)f(s)\mathrm{d}s.
\end{equation}
Under our assumptions on kernel $\kappa$, the evolution equation \eqref{eq0121_02} can be rewritten as the Volterra-type equation \eqref{23.08.14.13.29},
where $\varkappa$ is a kernel that is uniquely determined by $\kappa$.
We discuss the equivalence at the end of Section \ref{volterra}.

In the studies of trace theorems, interpolation theory is mainly used in the literature.
That is, one expects to obtain the following inequality:
\begin{align}\label{trace_powerweight}
    \|u_0\|_{(X_0, X_1)_{\theta, p}} \lesssim \|u\|_{L_p(\bR_+, t^{\gamma_1}\,\mathrm{d}t; X_0)} + \|f\|_{L_q(\bR_+, t^{\gamma_2}\,\mathrm{d}t; X_1)}.
\end{align}
Note that $(X_0, X_1)_{\theta,p}$ is an interpolation space whose norm is given for a functional $K$ by
\begin{align}\label{inter_Knorm}
    \|u_0\|_{(X_0, X_1)_{\theta, p}} := \left( \int_0^\infty t^{-\theta p} K\left( t, u_0; X_0, X_1 \right)^p \,\frac{\mathrm{d}t}{t} \right)^{1/p},\quad \theta\in(0,1),\,\,p\in[1,\infty).
\end{align}
Precise information for $K$ and other interpolation methods are given in Section~\ref{sec_inter}.
The idea of trace inequalities in the literature is that one can express the functional $K$ as certain quantities of $\|u\|_{X_0}$ and $\|f\|_{X_1}$,
and the remained term $t^{-\theta p}$ contributes to a time weight in the right-hand side of \eqref{trace_powerweight}.
For detailed argument, we recommend \cite{Veraar23,KimWoo23,triebel} and references therein.

From the perspective of the classical approaches, to obtain \eqref{trace_powerweight} in terms of general temporal weights, 
it is natural to consider a generalization of $(X_0, X_1)_{\theta, p}$-interpolation in the sense that we put $\phi(t^{-1})$ rather than just $t^{-\theta}$. One of the novelties of this paper is that we suggest a function class, $\mathcal{I}_o(a,b)$, introduced in \cite{Gus_Pee1977} for such $\phi$ as well as construct general interpolation space $(X_0, X_1)_{\phi, p}$, whose norm is given by
\begin{align}\label{inter_Knorm_phi}
    \|u_0\|_{(X_0, X_1)_{\phi, p}} := \left( \int_0^\infty \phi \left( t^{-1} \right)^p K\left( t, u_0; X_0, X_1 \right)^p \,\frac{\mathrm{d}t}{t} \right)^{1/p},\quad \phi\in \mathcal{I}_o(0,1).
\end{align}
The definition of $\mathcal{I}_o(a,b)$ and its properties are given in the first part of Section~\ref{sec_prob}, and we explain the interpolation in Section~\ref{sec_inter}.
For short introduction of $\mathcal{I}_o(a,b)$, we note that $\phi \in \mathcal{I}_o(a,b)$  if and only if there exists $\varepsilon>0$ such that
$$
\lambda^{a+\varepsilon}\lesssim \frac{\phi(\lambda t)}{\phi(t)}\lesssim \lambda^{b-\varepsilon},\quad \forall t>0,\,\,\lambda\geq1.
$$
The advantage of the general interpolation is that we can handle not only trace inequalities for time local equations \eqref{eq0121_01} but also time non-local equations \eqref{eq0121_02}, including time-fractional derivatives for broader weight classes than those in the literature.
To the best of our knowledge, our result is new even for the local derivative case, in the sense that trace and extension theorems hold with Muckenhoupt's $A_p$-weight classes.

Before introducing the main results of this paper, we give a brief overview of the literature related to initial value problems in (weighted) Sobolev spaces with or without non-local derivatives.
The initial value problem of parabolic equations in $L_p$-Sobolev spaces with local time derivatives has been studied for a long time, and trace (and extension) theorem with or without power-type temporal weights can be found in, for example, Weidemaier \cite{Weidmaier95}, Pr\"uss \cite{Pruss02}, Lindemulder and Veraar \cite{LindeVeraar20} and references therein.
In the case of the parabolic equations with non-local derivatives, Zacher \cite{Zacher05, Zacher06} obtained an unweighted $L_p$-theory for the Volterra-type equations, including the time-fractional heat equations.
In particular, the kernel $\varkappa$ of \eqref{23.08.14.13.29} considered in \cite{Zacher05, Zacher06} belongs to a certain class, and its Laplace transform $\cL[\varkappa]$ satisfies some growth conditions near zero and infinity, in which case the solution space for \eqref{23.08.14.13.29} is given as the vector-valued Bessel potential spaces $H_p^{\alpha}$ with $\alpha > 1/p$.
See \cite{Pruss1993} for a detailed description of kernel $\kappa$ of $\partial_t^{\kappa}$ and solution spaces used in \cite{Zacher05, Zacher06}.
In Meyries and Schnaubelt \cite{Meyries12}, Meyries and Veraar \cite{MeyriesVeraar14}, Agresti, Lindemulder, and Veraar \cite{Veraar23}, the initial trace results in \cite{Zacher05, Zacher06} were extended to the case of vector-valued Triebel-Lizorkin spaces containing power-type temporal weight: 
$$
\text{$F_{p, q}^{\alpha}(\bR_+, w_{\gamma}; X_0) \cap L_p(\bR_+, w_{\gamma}; X_1)$, $w_{\gamma}(t) = |t|^{\gamma}$},
$$
where $\gamma \in (-1, p-1)$ and $\alpha > (1+\gamma)/p$.
While the aforementioned results are basically based on the semigroup approach, Dong and Kim \cite{DongKim21}, Dong and Liu \cite{DongLiu22}, and Kim and Woo \cite{KimWoo23} use different approaches to solve initial value problems for time-fractional parabolic equations and derive initial trace estimates for solutions.
Particularly, the authors in \cite{KimWoo23} construct new solution spaces suitable for discussing the initial value problem of time-fractional evolution equations and investigate the initial behavior of the solution using an appropriate representation via standard mollification.
The trace estimates in \cite{KimWoo23} are partially consistent with Theorem \ref{trace_non_local} in this paper. (See Remark \ref{rem_frac}.)
For a certain type of evolution equation with local derivatives, Choi, Kim, and Lee \cite{ChoiKimLee23} prove an extension theorem with general temporal weights by solving an initial value problem.
We emphasize that our main results allow one to consider $A_p$ weights in time variable and generalized time-fractional derivatives, which contain power-type temporal weights and standard time-fractional derivatives, respectively.

\subsection{Main results}
The main results of this paper are trace and extension theorems that identify the optimal function space for the initial conditions of equations in Sobolev spaces with $A_p$ weights.
Notably, the initial data spaces are closely related to the temporal weights and the properties of local or non-local time derivatives.

We introduce trace theorems with local and non-local derivatives, respectively.
Let $(X_0, X_1)$ be an interpolation couple and
$X_0+X_1$ be a Banach space with the norm 
$$
\|a\|_{X_0+X_1}:=\inf\{\|a_0\|_{X_0}+\|a_1\|_{X_1}: a = a_0 + a_1,\, a_0\in X_0,\,a_1\in X_1\}.
$$
\begin{thm}[Trace theorem with local derivative; half line]
    \label{trace_local}
    Let $p\in(1,\infty)$ and $w\in A_p$. Suppose that $u\in L_p(\bR_+,w\,\mathrm{d}t;X_0)$, $f\in L_p(\bR_+,w\,\mathrm{d}t;X_1)$, $u_0 \in X_0 + X_1$, and
    \begin{equation}
    \label{local_rep}
    u(t) = u_0 + \int_0^tf(s)\mathrm{d}s
    \end{equation}
    for $ t\in\bR_+$. Then $u_0 \in (X_0,X_1)_{W^{1/p},p}$ and
    \begin{align}
    \label{trace 1}
        \|u_0\|_{(X_0,X_1)_{W^{1/p},p}}\lesssim_{p, [w]_{A_p}} \|u\|_{L_p(\bR_+,w\,\mathrm{d}t;X_0)}+\|f\|_{L_p(\bR_+,w\,\mathrm{d}t;X_1)},
    \end{align}
    where $W(t) := \int_0^t w(s)~\mathrm{d}s$.
\end{thm}

For \eqref{eq0121_02}, assume that $\kappa:\bR_+ \to \bR_+$ is a right-continuous decreasing function with $\kappa\in\cI_o(-1,0)$.
Note that $\mathcal{I}_o(-1,0)$ contains the kernel $t^{-\alpha}/\Gamma(1-\alpha)$ generating the Caputo fractional derivative of order $\alpha\in(0,1)$.
Let us denote
\begin{equation}
\label{23.08.11.17.26}
\kappa^\ast(t):=\kappa^{-1}\big(1/t\big).
\end{equation}
Here, $\kappa^{-1}$ is a generalized inverse of $\kappa$ defined in  \eqref{2309190655}.
Then the following theorem holds:
\begin{thm}[Trace theorem with non-local derivative; half line]
    \label{trace_non_local}
    Let $p\in(1,\infty)$, $w\in A_p$ and $\kappa \in \cI_o(-1, 0)$. Suppose that $W \circ \kappa^{*} \in \cI_o(0, p)$, $u\in L_p(\bR_+,w\,\mathrm{d}t;X_0)$, $f\in L_p(\bR_+,w\,\mathrm{d}t;X_1)$, $u_0 \in X_0+X_1$, and
    \begin{equation}
    \label{eq0128_01}
    \int_0^t \kappa(t-s) \left( u(s) - u_0 \right)\,\mathrm{d}s = \int_0^t f(s)\,\mathrm{d}s,
    \end{equation}
    for $t \in \bR_+$.
    Then $u_0 \in (X_0,X_1)_{(W\circ\kappa^{*})^{1/p},p}$ and
    \begin{align}
    \label{trace 2}
        \|u_0\|_{(X_0,X_1)_{(W\circ\kappa^{*})^{1/p},p}} \lesssim_{p, [w]_{A_p}, \kappa} \|u\|_{L_p(\bR_+,w\,\mathrm{d}t;X_0)}+\|f\|_{L_p(\bR_+,w\,\mathrm{d}t;X_1)}.
    \end{align}
\end{thm}

\begin{rem}
If $\kappa$ is a strictly decreasing continuous function, then the assumption $\kappa\in\cI_o(0,1)$ and $W\circ \kappa^\ast\in\cI_o(0,p)$ can be replaced by the following: there exist constants $0<a_1<a_2$ and $0<b_1<b_2<p$ such that for any $t>0$ and $\lambda \geq 1$,
$$
\lambda^{a_1}\lesssim \left(\frac{\kappa(t)}{\kappa(\lambda t)}\right)^{a_2}\lesssim \frac{W(\lambda t)}{W(t)}\lesssim \left(\frac{\kappa(t)}{\kappa(\lambda t)}\right)^{b_2}\lesssim \lambda^{b_1}.
$$
For this, see Lemma \ref{22.10.26.15.52}-$(vi)$.
\end{rem}

For the Volterra-type equations \eqref{23.08.14.13.29}, it is shown in Section~\ref{volterra} that $u$ and $f$ of \eqref{23.08.14.13.29} satisfy \eqref{eq0128_01}.
Thus \eqref{trace 2} also holds for $u$ and $f$ of \eqref{23.08.14.13.29}.
The same is true for Theorem~\ref{extension_non_local}.

\begin{rem}
In Theorems \ref{trace_local} and \ref{trace_non_local}, we assume that $u$ (and $f$) is defined on $\bR_+$, but if $X_0$ is continuously embedded into $X_1$, \textit{i.e.}, $X_0 \subset X_1$, we also obtain trace estimates of $u$ defined on a finite time interval $(0, T)$. 
See Section \ref{23.08.09.13.27} for this.
\end{rem}

\hfill

Another novelty is that we construct $u$ and $f$ satisfying \eqref{eq0128_01} by means of probability theory, which is motivated by spectral theory.
Regarding the case of \eqref{eq0121_01}, it is natural to consider its fundamental solution as one-parameter semigroups.
Then together with generalized interpolation theory given in Section~\ref{sec_inter}, we have the following result:
\begin{thm}[Extension theorem with local derivative]
    \label{extension_local}
    Let $p\in(1,\infty)$, $w\in A_p$, and $a\in (X_0,X_1)_{W^{1/p},p}$. 
    Then there exist $u\in L_p(\bR_+,w\,\mathrm{d}t;X_0)$ and $f\in L_p(\bR_+,w\,\mathrm{d}t;X_1)$
    such that $u(0) = a$ and
    $$
    u(t)=a+\int_0^tf(s)\, \mathrm{d}s, 
    $$
    for $t \in \bR_+$. Furthermore, 
    $$
    \|u\|_{L_p(\bR_+,w\,\mathrm{d}t;X_0)}+\|f\|_{L_p(\bR_+,w\,\mathrm{d}t;X_1)}\lesssim_{p, [w]_{A_p}} \|a\|_{(X_0,X_1)_{W^{1/p},p}}.
    $$
\end{thm}

Even if we have a non-local kernel $\kappa$ as in \eqref{eq0121_02}, we can accomplish a similar result.
It is because the kernel $\kappa$ corresponds to a Bernstein function $\phi$, and $\phi$ yields a subordinator $S = (S_t)_{t\geq0}$ whose Laplace exponent is $\phi$.
Then one can establish a solution to \eqref{eq0121_02} by making use of the subordinator $S$.
This procedure is given in the second part of Section~\ref{sec_prob}, and particularly in Proposition~\ref{22.11.07.13.23}.
Finally, we have the following theorem:
\begin{thm}[Extension theorem with non-local derivative]
    \label{extension_non_local}
    Let $p\in(1,\infty)$ and $w\in A_p$. Suppose that $\kappa \in \cI_o(-1, 0)$, $W\circ\kappa^{*}\in \cI_o(0,p)$, and $a\in (X_0,X_1)_{(W\circ\kappa^{*})^{1/p},p}$.
    Then there exist $u\in L_p(\bR_+,w\,\mathrm{d}t;X_0)$ and $f\in L_p(\bR_+,w\,\mathrm{d}t;X_1)$
    such that
    \begin{equation}
    \label{eq0129_01}
    \int_0^t \kappa \left( t-s \right) \left( u(s) - a \right) \,\mathrm{d}s = \int_0^t f(s)\,\mathrm{d}s
    \end{equation}
    for $t \in \bR_+$ and $u(0) = a$. Furthermore,
\begin{equation}
	\label{eq0922_01}
    \|u\|_{L_p(\bR_+,w\,\mathrm{d}t;X_0)}+\|f\|_{L_p(\bR_+,w\,\mathrm{d}t;X_1)} \lesssim_{p, [w]_{A_p}, \kappa,  W \circ \kappa^{*}} \|a\|_{(X_0,X_1)_{(W\circ\kappa^{*})^{1/p},p}}
  \end{equation}
\end{thm}

Note that \eqref{eq0922_01} in Theorem \ref{extension_non_local} depends on $W\circ\kappa^{*}$ while \eqref{trace 2} in Theorem~\ref{trace_non_local} does not.
The difference arises because equivalence of $\|a\|_{(X_0,X_1)_{(W\circ\kappa^{*})^{1/p},p}^K}$ and $\|a\|_{(X_0,X_1)_{(W\circ\kappa^{*})^{1/p},p}^J}$ depends on $W\circ\kappa^{*}$, and we use $K$ and $J$ methods to prove Theorems~\ref{trace_non_local} and \ref{extension_non_local}, respectively.
In other words, the trace and extension theorems with non-local derivatives essentially depend on $p$, $[w]_{A_p}$ and $\kappa$.

\begin{rem}\label{rem_frac}
In this paper, while the results on the local derivative (Theorem~\ref{trace_local}, Corollary \ref{2307061138} and Theorem \ref{extension_local}) hold for any $A_p$ weights, the results on the non-local derivatives (Theorems~\ref{trace_non_local}, \ref{2307111235} and \ref{extension_non_local}) hold for $w\in A_p$ with $W \circ \kappa^{*} \in \cI_o(0,p)$.
In particular, if we consider the Caputo fractional derivative $\partial_t^\alpha$ in \eqref{eq0121_02}, then Theorems~\ref{trace_non_local}, \ref{2307111235} and \ref{extension_non_local} hold for $w(t) = |t|^\gamma$ with $-1<\gamma<p-1$ and $\gamma+1 < p\alpha$, which is a special case of main results in \cite{Veraar23} and \cite{KimWoo23}.

Indeed, let $\alpha \in (0, 1)$, $q \in (1, \infty)$, and $w(t) = |t|^{\gamma}$ where $\gamma \in (-1, q-1)$ and $1 + \gamma < q\alpha$. Note that $w \in A_q(\bR)$. 
If we take $\kappa(t) = \big( \Gamma ( 1-\alpha ) \big)^{-1}t^{-\alpha}$ in the definition of $I^{\kappa}$ (see \eqref{eq_non-local}), for sufficiently smooth function $u = u(t, x)$ defined on $[0, \infty) \times \Omega$,
$$
I^{\kappa}u (t, x) = \int_0^t \kappa(t-s) \big(u(s, x) - u(0, x)\big)\,\mathrm{d}s = I^{1-\alpha} \big( u ( \cdot, x ) - u(0, x)\big)(t),
$$
where $I^{1-\alpha}$ stands for the classical fractional integral,  \textit{i.e.}, the Riemann-Liouville fractional integral.
Also let $X_0 = H_p^2$ and $X_1 = L_p$, $1 < p < \infty$, where $H_p^2$ is a usual Sobolev space for the second order PDEs.
By the choice of $\kappa$, it is clear that $\kappa \in \cI_o(-1, 0)$ and 
$$
\kappa^\ast(t)=\big(\Gamma(1-\alpha)\big)^{1/\alpha}t^{1/\alpha}.
$$
Moreover, since $W(t) \simeq t^{1+\gamma}$, we also have $(W \circ \kappa^\ast)(\lambda) \simeq \lambda^{(1+\gamma)/\alpha} \in \cI_o(0, q)$ by the choice of $\alpha$, $q$, and $\gamma$. Then by Theorem \ref{trace_non_local}, 
$$
u(0, \cdot) \in (X_0, X_1)_{(W\circ\kappa^\ast)^{1/q}, q} = (H_p^2, L_p)_{(W\circ\kappa^\ast)^{1/q}, q}
$$
where $(W\circ\kappa^\ast)^{1/q}(\lambda) = \lambda^{(1+\gamma)/q\alpha}$.
On the other hand, by Proposition \ref{exs} (\textit{ii}), we have that $(H_p^2, L_p)_{(W\circ\kappa^\ast)^{1/q}, q} = B_{p,q}^{(W\circ\kappa^\ast)^{1/q}(2,0)}$.
    That is, we have
\begin{align*}
        \|f\|_{B_{p,q}^{(W\circ\kappa^\ast)^{1/q}(0,2)}} \,&= \|S_0 f\|_p + \left(\sum_{j\geq1} 2^{2qj}\big(W\circ\kappa^\ast\big)(2^{-2j}) \| \Delta_j f\|_{L_p}^q \right)^{\frac{1}{q}}\\
&\simeq \|S_0 f\|_p + \left(\sum_{j\geq1} 2^{2(1 - \frac{1+\gamma}{q\alpha})qj} \| \Delta_j f\|_{L_p}^q \right)^{\frac{1}{q}}
= \|f\|_{B_{p,q}^{2-\frac{2(\gamma+1)}{q\alpha}}}.
\end{align*} 
Thus, $u(0, \cdot) \in B_{p,q}^{2-\frac{2(\gamma+1)}{q\alpha}}$ and this coincides with \cite[Theorem 3.11]{KimWoo23} for $k=0$ in there. The extension part can be dealt with similar way.

We also remark that if we take $X_0$ and $X_1$ as weighted Sobolev spaces, the generalized interpolation space $(X_0, X_1)$ becomes a weighted Besov space with variable smoothness, which is characterized by means of difference operators.
For more detail, see \cite{btl}.

\end{rem}

\section*{Notations}
Throughout this paper, we will use the following standard notations.

$\bullet$ $\bR$, $\bQ$, and $\bZ$ are the set of all real numbers, rational numbers, and integers, respectively.

$\bullet$ $\bR_+ = (0, \infty)$, and $\bQ_+ = \bQ \cap \bR_+$.

$\bullet$ $a \wedge b = \min (a, b)$ and $a \vee b = \max  (a, b)$ for $a, b \in \bR$.

$\bullet$ For a set $\{a, b, \ldots \}$ and $X, Y \in \bR$, we say $X \lesssim_{a, b, \ldots} Y$, if $X \le N Y$ holds for some $N = N(a, b, \ldots) > 0$. Sometimes we omit $a, b, \ldots$ and just say $X \lesssim Y$ if the dependency on $a, b, \ldots$ is clear from the context.

$\bullet$ For $p \in [1, \infty]$, $p' := \frac{p}{p-1}$ is the H\"older conjugate ($\frac{1}{0} := \infty$ and $\frac{\infty}{\infty} := 1$).

$\bullet$ $\mathbbm{1}_A$ is a characteristic function on a set $A$.

$\bullet$ For a positive function $f$ defined on $\bR_+$, we use the standard big-$O$ notation $O\big(f(\lambda)\big)$ and little-$o$ notation $o\big(f(\lambda)\big)$.

\mysection{Preliminaries}\label{sec_prob}

\subsection{Functions of type $\cI(a, b)$ and $\cI_o(a,b)$ and their properties}

We introduce a class $\cI(a,b)$ and $\cI_o(a,b)$. 
In particular, functions in $\cI_o(0,1)$ are essential in a generalized interpolation introduced in Section~\ref{sec_inter}.
\begin{defn}
For a function $\phi:\bR_+\rightarrow \bR_+$, we define 
$$
s_\phi(\lambda) = \sup_{t>0}\frac{\phi(\lambda t)}{\phi(t)}\,,
$$
so that $s_{\phi}:\bR_+\rightarrow (0,\infty]$.
Observe that $s_\phi$ is \textit{submultiplicative}, \textit{i.e}., $s_\phi(\lambda\tau) \leq s_\phi(\lambda) s_\phi(\tau)$ for $\lambda, \tau>0$.
Together with the definition of $s_\phi(\lambda)$ we introduce the following class of functions. For $a, b \in \mathbb{R}$,
	\begin{align}\label{ftn class}
		\mathcal{I}(a, b) := \{\phi : \text{$s_\phi(\lambda) = O(\lambda^a)$ as $\lambda \to 0$, and $s_\phi(\lambda) = O(\lambda^b)$ as $\lambda \to \infty$} \}.
	\end{align}
We also define $\mathcal{I}_o(a,b)$, where $o$ denotes that we replace $O$-notation with $o$-notation in \eqref{ftn class}. Clearly, $\mathcal{I}_o(a,b)\subsetneq \mathcal{I}(a,b)$.
\end{defn}
\begin{rem}
\label{rmk_bdd}
Due to $1 = s_{\phi}(1)\leq s_{\phi}(\lambda)\,s_{\phi}(\lambda^{-1})$, we obtain that $\mathcal{I}(a,b)=\emptyset$ for $a> b$, and $\mathcal{I}_o(a,b)=\emptyset$ for $a\geq b$.
In addition, for $a<b$, $\phi\in \mathcal{I}_o(a,b)$ if and only if there exists a bounded function $K_{\phi}:\bR_+\rightarrow \bR_+$ such that
$$
s_{\phi}(\lambda)\leq \left(\lambda^a+\lambda^b\right)K_{\phi}(\lambda)\quad \textrm{and}\quad \lim_{\lambda\rightarrow 0}K_{\phi}(\lambda)=\lim_{\lambda\rightarrow \infty}K_{\phi}(\lambda)=0.
$$
See, for example, \cite[Proposition 1.1]{Gus_Pee1977}.
\end{rem}

The following lemma provides useful information about the classes $\cI(a, b)$ and $\cI_o(a, b)$:
\begin{lem}
\label{22.10.26.15.52}
    Let $a,b\in\bR$ and $\phi\in \mathcal{I}(a,b)$.
    \begin{enumerate}[(i)]
        \item For $\alpha\in\bR$, $t^\alpha \phi(t) \in \mathcal{I}(a+\alpha, b+\alpha)$.
        \item For $\alpha\geq0$, $\phi(t^\alpha), \phi(t)^\alpha \in \mathcal{I}(\alpha a, \alpha b)$.
        \item For $\alpha \leq 0$, $\phi(t^\alpha), \phi(t)^\alpha \in \mathcal{I}(\alpha b, \alpha a)$.
        \item Let $a,b>0$ and $\phi$ be strictly increasing, continuous, and
        $$
        \lim_{\lambda\uparrow\infty}\phi(\lambda)=\infty.
        $$
        Then its inverse $\phi^{-1}$ is of class $\mathcal{I}(\frac{1}{b}, \frac{1}{a})$.
        \item For $a<c$, $b>d$, $\mathcal{I}(c,d)\subset \mathcal{I}_{o}(a,b)$ and
        \begin{align}\label{ineq-221014 13300}
		\bigcup_{c>a,\, d<b} \mathcal{I}(c, d) = \mathcal{I}_o(a, b).
        \end{align}
        \item $\phi\in\mathcal{I}_{o}(a,b)$ if and only if there exists $\varepsilon >0$ such that
        \begin{align}\label{w scaling}
		\lambda^{a+\varepsilon} \lesssim \frac{\phi(\lambda t)}{\phi(t)} \lesssim \lambda^{b - \varepsilon}\quad \forall \, t\geq 0, \, \lambda\geq 1.
	    \end{align}
        \item If $p\geq0$ and $\phi\in\mathcal{I}_{o}(0,p)$, then
        \begin{align}
        \label{22.10.27.13.31}
		\int_0^\infty \left( 1 \wedge \frac{x}{t}\right)^p \phi(t) \frac{\mathrm{d}t}{t} \lesssim \phi(x).
        \end{align}
        \item If $p\geq0$ and $\phi\in\mathcal{I}_o(-p,0)$, then
        \begin{align*}
		\int_0^\infty \left( 1 \wedge \frac{t}{x}\right)^p \phi(t) \frac{\mathrm{d}t}{t} \lesssim \phi(x),
        \end{align*}
        
        \item For $\phi \in \mathcal{I}_o(a, b)$, the assertions $(i)$ -- $(iii)$ also hold for $\mathcal{I}_{o}(a,b)$ instead of $ \mathcal{I}(a,b)$.
    \end{enumerate}
\end{lem}

\begin{proof}
    $(i)$--$(iii)$ and $(ix)$ follows from direct computations. 
    
    For $(iv)$, first, we observe that there exists a constant $C>0$ such that
    $$
    \phi(\lambda t)\leq C\lambda^a\phi(t)
    $$
    for $t \in \bR_+$ and $\lambda \in (0,1]$.
    On the other hand, for $x\in\bR_+$, by the assumption for $\phi$, there exists $t\in\bR_+$ such that $C\phi(t)=x$. Then, again by the assumption for $\phi$, we have $\lambda\phi^{-1}(x/C)\leq\phi^{-1}(\lambda^ax)$, and hence
    $$
 \frac{\phi^{-1}(x/(C\lambda^a))}{\phi^{-1}(x)}\leq \frac{1}{\lambda}
    $$
    for all $x \in \bR_+$ and $\lambda\in(0,1]$.
    This certainly implies that $s_{\phi^{-1}}(\lambda)=O(\lambda^{1/a})$ as $\lambda \to \infty$.
    Similarly, we also have $s_{\phi^{-1}}(\lambda)=O(\lambda^{1/b})$ as $\lambda\to0$.
    
    $(v)$ Clearly, $\mathcal{I}(c,d)\subset\mathcal{I}_o(a,b)$ for $a< c$ and $b > d$, and thus
    $$
    \bigcup_{a<c,b>d}\mathcal{I}(c,d) \subset \mathcal{I}_o(a,b).
    $$
Therefore, it suffices to show that ``$\supset$" holds instead of ``$=$" in \eqref{ineq-221014 13300}. Let $\phi \in \mathcal{I}_o(a,b)$ and take a constant $N>0$ such that $s_{\phi}(\lambda)\leq \frac{1}{2}\lambda^a$ if $\lambda\leq 2^{-N}$, and $s_{\phi}(\lambda) \leq \frac{1}{2}\lambda^b$ if $\lambda\geq 2^{N}$.
Then for $\lambda\leq 2^{-N}$, there is $k\in\mathbb{N}$ such that $2^{-(k+1)N}< \lambda \leq 2^{-kN}$, which yields
	\begin{align}\label{ineq-221014 1359 3}
		s_{\phi} \left( \lambda \right) \leq \left(s_{\phi} \left( 2^{-N} \right) \right)^{k-1} s_{\phi} \left( 2^{N \left( k-1 \right)}\lambda \right) \leq 2^{-k} \lambda^a \leq 2 \lambda^{a + \frac{1}{N}}
	\end{align}
by the submultiplicativity.	
Similarly, for $\lambda \geq 2^{N}$ there is $l\in \mathbb{N}$ such that $2^{lN} \leq \lambda < 2^{(l+1)N}$, which yields
	\begin{align}\label{ineq-221014 1359 4}
	s_{\phi} \left( \lambda \right) \leq \left(s_{\phi} \left( 2^N \right)\right)^{l-1} s_{\phi}\left( 2^{-N\left( l-1 \right)} \lambda \right) \leq 2^{-l} \lambda^{b} \leq 2 \lambda^{b - \frac{1}{N}}.
	\end{align}
By \eqref{ineq-221014 1359 3} and \eqref{ineq-221014 1359 4}, $\phi \in \mathcal{I}(a + \frac{1}{N}, b - \frac{1}{N})$. This proves \eqref{ineq-221014 13300}.

$(vi)$ This is a direct result of $(v)$ and its proof with Remark \ref{rmk_bdd}.

$(vii)$ By $(v)$ and $(vi)$, there exists $\varepsilon>0$ such that $\phi\in\mathcal{I}(\varepsilon,p-\varepsilon)$ and
$$
\frac{\phi(t)}{\phi(x)}\lesssim\left(\frac{t}{x}\right)^{\varepsilon}\mathbbm{1}_{t\leq x}+\left(\frac{t}{x}\right)^{p-\varepsilon}\mathbbm{1}_{t>x}. 
$$
Therefore,
\begin{equation}
\label{22.10.27.13.30}
\begin{aligned}
    \int_0^{\infty}\left(1\wedge\frac{x}{t}\right)^p\frac{\phi(t)}{\phi(x)}\frac{\mathrm{d}t}{t}&=\int_0^x\cdots\,\,+\int_x^{\infty}\cdots\\
    &\lesssim \int_0^x\left(\frac{t}{x}\right)^{\varepsilon}\frac{\mathrm{d}t}{t}+\int_x^{\infty}\left(\frac{x}{t}\right)^p\left(\frac{t}{x}\right)^{p-\varepsilon}\frac{\mathrm{d}t}{t}\lesssim 1.
\end{aligned}    
\end{equation}

$(viii)$ Since $\phi\in\mathcal{I}_{o}(-p,0)$, by $(i)$, $t^p\phi(t)\in\mathcal{I}_{o}(0,p)$. Using $(vii)$,
\begin{align*}
    \int_0^{\infty}\left(1\wedge\frac{t}{x}\right)^p\phi(t)\,\frac{\mathrm{d}t}{t}=x^{-p}\int_0^{\infty}\left(1\wedge\frac{x}{t}\right)^pt^p\phi(t)\,\frac{\mathrm{d}t}{t}\lesssim \phi(x).
\end{align*}
The lemma is proved.
\end{proof}

When a kernel $\kappa$ of $\partial_t^\kappa$ is of class $\cI_o(-1,0)$, we have the following proposition:
\begin{prop}
\label{22.11.07.11.09}
Let $\kappa$ be a positive function defined on $\bR_+$.
\begin{enumerate}[(i)]
\item If $\kappa\in\mathcal{I}_o(-1,0)$, then for $\lambda \in \bR_+$,
\begin{align}\label{221103637}
\cL[\kappa](\lambda) := \int_0^{\infty} \mathrm{e}^{-\lambda t} \kappa(t) \,\mathrm{d}t \simeq \frac{\kappa(\lambda^{-1})}{\lambda },
\end{align}
where the equivalence ($\simeq$) depends only on $K_\phi$ in Remark~\ref{rmk_bdd}.
Moreover, $\cL[\kappa]\in\mathcal{I}_o(-1,0)$ (by Lemma \ref{22.10.26.15.52}-$(ix)$).

\item If $\kappa$ is a decreasing function on $\bR_+$, and $\cL[\kappa]\in\cI_o(-1,0)$, then \eqref{221103637} holds, and the equivalence in \eqref{221103637} depends only on $K_{\cL[\kappa]}$.
Moreover, $\kappa\in\cI_o(-1,0)$ (by Lemma \ref{22.10.26.15.52}-$(ix)$).
\end{enumerate}
\end{prop}
\begin{proof}
$(i)$ First note that $\cL[\kappa](\lambda)$ is well-defined on $\bR_+$ due to $\kappa\in \cI_o(0,1)$. Then for $\lambda\in\bR_+$,
$$
\cL[\kappa](\lambda)=\int_0^{\lambda^{-1}}\mathrm{e}^{-\lambda t}\kappa(t)\mathrm{d}t+\int_{\lambda^{-1}}^{\infty}\mathrm{e}^{-\lambda t}\kappa(t)\mathrm{d}t=:I_0(\lambda)+I_1(\lambda).
$$
Since $\kappa\in\cI_o(-1,0)$, by Lemma \ref{22.10.26.15.52}-$(vi)$, there exists $\varepsilon>0$ such that
$$
(s^{-1} t)^{-1+\varepsilon}\kappa(s) \lesssim\kappa(t)\lesssim (s^{-1} t)^{-\varepsilon}\kappa(s)\quad\forall\,\, 0<s<t<\infty.
$$
Therefore, we have
\[
\kappa(\lambda^{-1})\int_0^{\lambda^{-1}}\mathrm{e}^{-\lambda t}(\lambda t)^{-1+\varepsilon}\mathrm{d}t \lesssim I_0(\lambda)\lesssim \kappa(\lambda^{-1})\int_0^{\lambda^{-1}}\mathrm{e}^{-\lambda t}(\lambda t)^{-\varepsilon}\mathrm{d}t,
\]
and
\[
\kappa(\lambda^{-1})\int_{\lambda^{-1}}^{\infty}\mathrm{e}^{-\lambda t}(\lambda t)^{-1+\varepsilon}\mathrm{d}t \lesssim I_1(\lambda)\lesssim \kappa(\lambda^{-1})\int_{\lambda^{-1}}^{\infty}\mathrm{e}^{-\lambda t}(\lambda t)^{-\varepsilon}\mathrm{d}t.
\]
This implies
$$
\cL[\kappa](\lambda)=I_0(\lambda)+I_1(\lambda)\simeq\frac{\kappa(\lambda^{-1})}{\lambda}.
$$

$(ii)$ We assume that $\cL[\kappa]\in\cI_o(-1,0)$. Observe that
$$
\cL[\kappa](\lambda^{-1})=\lambda\int_0^{\infty}\mathrm{e}^{-z}\kappa(\lambda z)\mathrm{d}z.
$$
Since $\kappa$ is a decreasing function, we see that
\begin{align}
\label{22.11.01.17.20}
    \cL[\kappa](\lambda^{-1})\geq \lambda\int_0^1\mathrm{e}^{-z}\kappa(\lambda z)\mathrm{d}z\geq \mathrm{e}^{-1}\lambda \kappa(\lambda).
\end{align}
On the other hand, by \eqref{22.11.01.17.20}, for any $\delta \in \bR_+$,
$$
\lambda \int_0^{\delta}\mathrm{e}^{-z} \kappa(\lambda z) \mathrm{d}z  \leq \mathrm{e} \int_0^{\delta} \frac{\mathrm{e}^{-z}}{z}\cL[\kappa](\lambda^{-1}z^{-1})\,\mathrm{d}z.
$$
For $0<\delta\leq 1$, by Lemma \ref{22.10.26.15.52}-$(vi)$ with the fact that $\cL[\kappa] \in \cI_o(-1, 0)$, we have 
\begin{align*}
    \int_0^{\delta} \frac{\mathrm{e}^{-z}}{z}\cL[\kappa](\lambda^{-1}z^{-1})\mathrm{d}z &= \cL[\kappa](\lambda^{-1}) \int_0^{\delta} \frac{\mathrm{e}^{-z}}{z} \frac{\cL[\kappa](\lambda^{-1}z^{-1})}{ \cL[\kappa](\lambda^{-1})} \, \mathrm{d}z\\
    &\lesssim \cL[\kappa](\lambda^{-1})  \int_0^{\delta} \mathrm{e}^{-z} z^{\varepsilon - 1} \, \mathrm{d}z,
\end{align*}
for some $\varepsilon > 0$.
Hence
$$
\lambda \int_0^{\delta}\mathrm{e}^{-z} \kappa(\lambda z) \mathrm{d}z \leq \frac{1}{2} \cL[\kappa](\lambda^{-1}),
$$
for sufficiently small $ \delta = \delta(K_{\cL[\kappa]}) \in (0, 1)$. Then we have
\begin{align*}
    \cL[\kappa](\lambda^{-1}) &= \lambda \int_0^{\delta} \mathrm{e}^{-z} \kappa (\lambda z )\,\mathrm{d}z + \lambda \int_{\delta}^{\infty} \mathrm{e}^{-z} \kappa (\lambda z )\,\mathrm{d}z \\
    &\leq \frac{1}{2} \cL[\kappa](\lambda^{-1}) +  \lambda \int_{\delta}^{\infty} \mathrm{e}^{-z} \kappa (\lambda z )\,\mathrm{d}z.
\end{align*}
Since $\kappa$ is decreasing, 
$$
\cL[\kappa](\lambda^{-1}) \le 2\lambda \int_{\delta}^{\infty} \mathrm{e}^{-z} \kappa (\lambda z )\,\mathrm{d}z \le 2\lambda \kappa(\delta \lambda),
$$
and then, 
$$
\cL[\kappa](\lambda^{-1}) \le \cL[\kappa](\delta\lambda^{-1}) \le 2\delta^{-1} \lambda \kappa(\lambda)
$$
for all $\lambda>0$, where the first inequality is due to $\delta\leq 1$. This certainly implies that $\cL[\kappa](\lambda^{-1}) \lesssim \lambda \kappa(\lambda)$. The proposition is proved.
\end{proof}

\subsection{Volterra-type equations associated with Bernstein functions}\label{volterra}

In this subsection, we give a connection between a kernel $\kappa$ of non-local derivative in the evolution equation \eqref{eq0121_02} and a Bernstein function in probability theory.
Moreover, we use properties of Bernstein functions to construct $u$ and $f$ satisfying \eqref{eq0128_01}.

Let $\kappa:\bR_+\rightarrow \bR_+$ be a right-continuous decreasing function with $\kappa(0+) = \lim_{t \downarrow 0} \kappa(t) = \infty$ and $\kappa(\infty)=0$.
Then there exists a unique non-negative measure $\mu$ on $\bR_+$ such that
\begin{align}\label{2302231052}
\kappa(s)-\kappa(t)=\mu \left((s,t]\right),
\end{align}
for $0<s<t<\infty$.
Note that
\begin{align}\label{2302231625}
\int_0^\infty(1\wedge  t)\mu(\mathrm{d}t)=\int_0^1\kappa(s)\mathrm{d}s\quad \textrm{and}\quad \mu(\bR_+)=\kappa(0+) = \infty.
\end{align}
We will consider an  evolution equation
\begin{equation}
\label{eq_non-local}
I^{\kappa}u(t):=\int_0^t\kappa(t-s)\left(u(s)-u(0)\right)\mathrm{d}s=\int_0^tf(s)\mathrm{d}s\,,
\end{equation}
or for simplicity, we just denote \eqref{eq_non-local} by $\partial_t^\kappa u=f$.

\begin{rem}
\label{rmk_0129_01}
If $\kappa\in\cI_o(-1,0)$, then by Lemma \ref{22.10.26.15.52}-($vi$), there exists $\varepsilon>0$ such that
$$
s^{-\varepsilon}\kappa(1)\lesssim \kappa(s)\lesssim s^{-1+\varepsilon}\kappa(1)
$$
for $s\leq 1$ and
$$
\kappa(r)\lesssim r^{-\varepsilon}\kappa(1)
$$
for $r\geq 1$ (see \eqref{w scaling}). Therefore, we have
$$
\int_0^1 \kappa(s)\, \mathrm{d}s<\infty, \quad \lim_{t\rightarrow 0}\kappa(t)=\infty, \quad \textrm{and} \quad \lim_{t\rightarrow \infty}\kappa(t)=0.
$$
\end{rem}

\begin{defn}
    An infinitely differentiable function $\phi:(0,\infty)\to[0,\infty)$ is called a Bernstein function if
    $$
    (-1)^nD^n\phi(\lambda)\leq0,\quad \forall n\in\bN,\,\lambda\in(0,\infty).
    $$
\end{defn}
It is well known (\textit{e.g.} \cite[Theorem 3.2]{schbern}) that Bernstein function $\phi$ has a unique representation
\begin{align}\label{230919757}
\phi(\lambda)=a+b\lambda+\int_{0}^{\infty}\left(1-\mathrm{e}^{-\lambda t}\right)\,\mu(\mathrm{d}t),
\end{align}
where $a,b\geq0$ and
$\mu$ is a nonnegative measure on $\bR_+$ satisfying
\begin{align}\label{2211041216}
\int_0^{\infty}\left(1\wedge t\right)\mu(\mathrm{d}t)<\infty.
\end{align}
In \eqref{230919757}, $a$, $b$, and $\mu$ are called the \textit{killing term}, \textit{drift}, and \textit{L\'evy measure} of $\phi$, respectively.
Furthermore, the triplet $(a,b,\mu)$, which determines a Bernstein function $\phi$ uniquely, is called the \textit{L\'evy triplet} of $\phi$.
In this article, we only consider Bernstein functions without killing term, \textit{i.e.} $\phi(0+)=0$.
Indeed, it is well-known the connection between Bernstein functions without killing term and subordinators. Before its description, we first give definitions of subordinators and convolution semigroup of probability measures.
\begin{defn}
\label{23.02.28.15.38}
Let $S = (S_t)_{t \ge 0}$ be a real-valued L\'evy process defined on a probability space $(\Omega, \cF, \bP)$. We say $S = (S_t)_{t \ge 0}$ is a \textit{subordinator} if $\bP(S_t\geq0,\,\forall t\geq0)=1$. A \textit{convolution semigroup of probability measures} on $[0,\infty)$ is a family of probability measures $\{\mu_t\}_{t\geq0}$ satisfying the following properties;
\begin{enumerate}[(i)]
    \item For all $t\geq0$, $\mu_t([0,\infty))=1$.
    \item For all $t,s\geq0$, $\mu_t\ast\mu_s=\mu_{t+s}$, where
    $$
    \mu_t\ast\mu_s(B):=\int_0^{\infty}\int_{0}^{\infty}\mathbbm{1}_{B}(t'+s')\mu_t(\mathrm{d}t')\mu_s(\mathrm{d}s').
    $$
    \item For every compactly supported continuous functions $f:[0,\infty)\to\bR$,
    $$
    \lim_{t\downarrow0}\int_{0}^{\infty}f(t')\mu_t(\mathrm{d}t')=\int_{0}^{\infty}f(t')\varepsilon_0(\mathrm{d}t'),
    $$
    where $\varepsilon_0$ is the centered Dirac measure on $\bR$.
\end{enumerate}
\end{defn}
For the definition and properties of L\'evy processes, see \textit{e.g.} \cite{sato1999levy}. Now we give a short description for the connection between Bernstein functions without killing term and subordinators.  
\cite[Theorem 5.2]{schbern} provides one-to-one correspondence between Bernstein functions and convolution semigroups of sub-probability measures in the sense that
$$
\int_{[0,\infty)}\mathrm{e}^{-\lambda r}\,\mu_t(\mathrm{d}r)=\mathrm{e}^{-t\phi(\lambda)},\quad \forall (t,\lambda)\in[0,\infty)\times\bR_+.
$$
Here, convolution semigroups of sub-probability measures implies Definition \ref{23.02.28.15.38} with $\mu_t([0,\infty))=1$ replaced by $\mu_t([0,\infty))\leq1$.
However, the proof of \cite[Theorem 5.2]{schbern} includes that there is one-to-one correspondence between Bernstein functions without killing term and convolution semigroups of probability measures.
Similarly, from \cite[Proposition 5.5]{schbern}, we also obtain the fact that every law of subordinators is a convolution semigroup of probability measures on $[0,\infty)$.
Combining these results, for a given subordinator $S=(S_t)_{t\geq0}$, there exists a Bernstein function $\phi$ such that $\phi(0+)=0$ and
\begin{equation}
\label{sub_laplace}
\left( \int_{\Omega} \mathrm{e}^{-\lambda S_t(\omega)} \, \bP(\mathrm{d}\omega) =: \right) \bE[\mathrm{e}^{-\lambda S_t}] =\mathrm{e}^{-t\phi(\lambda)},\quad \forall (t,\lambda)\in[0,\infty)\times\bR_+.
\end{equation}
For the converse, we observe that any convolution semigroups of probability measures on $[0,\infty)$ is infinitely divisible (see e.g. \cite[Definition 7.1]{sato1999levy}). Then by \cite[Theorem 7.10]{sato1999levy}, there exists a real-valued L\'evy process $S=(S_t)_{t\geq0}$ such that $(\mu_1)^t(\cdot)=\bP(S_t\in\cdot)$, where $(\mu_1)^t$ is a probability measure satisfying
$$
\int_{-\infty}^{\infty}\mathrm{e}^{iz\cdot r}(\mu_1)^t(\mathrm{d}r)=\left(\int_{-\infty}^{\infty}\mathrm{e}^{iz\cdot r}\mu_1(\mathrm{d}r)\right)^t,\quad\forall z\in\bR. 
$$
If $t$ is a positive rational number, then by Definition \ref{23.02.28.15.38} $(ii)$, $(\mu_1)^t=\mu_t$. Using the density argument and Definition \ref{23.02.28.15.38} $(iii)$, $(\mu_1)^t=\mu_t$ for all $t\geq0$. Hence $\mu_t(\cdot)=\bP(S_t\in\cdot)$. Since the support of $\mu_t$ is contained in $[0,\infty)$, the state space of $S$ is $[0,\infty)$. This certainly implies that $S$ is a subordinator. Therefore for a given Bernstein function $\phi$ without killing term, there exists a subordinator $S=(S_t)_{t\geq0}$ satisfying \eqref{sub_laplace}. Summarizing these facts,
\begin{align*}
    \text{Subordinators}&
    \rightleftharpoons \text{Convolution semigroups of probability measures on }[0,\infty) \\
    &\rightleftharpoons \text{Bernstein functions without killing term}.
\end{align*}
Since there is a one-to-one correspondence between Bernstein functions without killing term and subordinators, we say a subordinator $S=(S_t)_{t\geq0}$ is a \textit{subordinator with Laplace exponent} $\phi$ if \eqref{sub_laplace} holds.

The following proposition gives a relation between a kernel $\kappa$ and a Bernstein function $\phi$.
\begin{prop}
\label{prop0129_01}
    Let $\mu$ be a nonnegative measure on $\bR_+$ with \eqref{2211041216}.
    Then
    $$
    \kappa(x):=\int_0^{\infty}\mathbbm{1}_{(x,\infty)}(t)\mu(\mathrm{d}t)\in\mathcal{I}_o(-1,0)
    $$
    if and only if
    $$
    \phi(\lambda):=\int_0^{\infty}(1-\mathrm{e}^{-\lambda t})\mu(\mathrm{d}t)\in\mathcal{I}_o(0,1).
    $$
    Moreover if $\kappa \in \cI_o(-1,0)$ or $\phi\in\cI_o(0,1)$, then for all $\lambda \in \bR_+$, we have
    \begin{equation}
    \label{22.11.01.17.45}
        \phi(\lambda) \simeq \kappa(\lambda^{-1}).
    \end{equation}
\end{prop}
\begin{proof}
Observe that
\begin{align}
\label{22.11.02.14.24}
    \phi(\lambda)=\int_0^{\infty}\left(\int_0^t\lambda \mathrm{e}^{-\lambda s}\mathrm{d}s\right)\mu(\mathrm{d}t)=\lambda\int_0^{\infty}\left(\int_s^{\infty}\mu(\mathrm{d}t)\right)\mathrm{e}^{-\lambda s}\mathrm{d}s=\lambda \cL[\kappa](\lambda).
\end{align}
By Lemma \ref{22.10.26.15.52}-$(ix)$ and Proposition \ref{22.11.07.11.09}, $\kappa \in \cI_o(-1, 0)$ is equivalent to $\cL[\kappa] \in \cI_o(-1,0)$, and to $\phi\in\cI_o(0,1)$.
The equivalence \eqref{22.11.01.17.45} is also obtained from Proposition \ref{22.11.07.11.09} with the fact that $\kappa$ is a decreasing function. The proposition is proved.
\end{proof}

\begin{rem}
\label{23.03.02.15.47}
From \eqref{22.11.01.17.45}, we also obtain
\begin{align}\label{2309190652}
\phi^{-1}\simeq \frac{1}{\kappa^{-1}},
\end{align}
where $\phi^{-1}$ is the inverse of $\phi$ and $\kappa^{-1}$ is a generalized inverse of $\kappa$ defined by
\begin{align}\label{2309190655}
\kappa^{-1}(\lambda):=\inf\{s>0\,:\,\kappa(s)\leq\lambda\}
\end{align}
(for more detail on properties of the generalized inverse, see \cite[Proposition 1]{EmbrechtsHofert2013} with $T=-\kappa$ in there).
Indeed, due to \eqref{22.11.01.17.45}, there exists a constant $N\geq 1$ such that 
\begin{align*}
\{s>0\,:\,\phi\left(1/s\right)\leq N^{-1}\lambda\}\subset\{s>0\,:\,\kappa\left(s\right)\leq \lambda\}\subset \{s>0\,:\,\phi\left(1/s\right)\leq N\lambda\}\,,
\end{align*}
which implies that
$$
\frac{1}{\phi^{-1}(N\lambda)}\leq \kappa^{-1}(\lambda)\leq \frac{1}{\phi^{-1}(N^{-1}\lambda)}.
$$
Due to $\phi\in\cI_o(0,1)$ and \ref{22.10.26.15.52}-$(iv)$ and $(v)$, we have
$$
\phi^{-1}(N^{-1}\,\cdot\,)\simeq \phi^{-1}\simeq \phi^{-1}(N\,\cdot\,).
$$
This proves \eqref{2309190652}.
\end{rem}

The following lemma shows that the second assertion of \eqref{2302231625} is equivalent to a subordinator $S$ being strictly increasing.
\begin{lem}\label{2211041251}
Let $S=(S_t)_{t\geq0}$ be a subordinator with Laplace exponent
$$
\phi(\lambda):=\int_0^{\infty}(1-\mathrm{e}^{-\lambda t})\mu(\mathrm{d}t),
$$
i.e., \eqref{sub_laplace} holds for $(t, \lambda) \in [0, \infty) \times \bR_+$, where $\mu$ is a nonnegative measure defined on $\bR_+$ with \eqref{2211041216}. Then $\mu(\bR_+)=\infty$ if and only if $S$ is strictly increasing almost surely.
\end{lem}

\begin{proof}We only give a sketch of the proof.
By the definition of subordinator, $S$ is strictly increasing almost surely if and only if 
\begin{align}
\label{23.02.28.18.21}
    \bP\left( \cup_{ \substack{p, q \in \bQ_+ \\ p > q}} \{S_{p} - S_{q} = 0\} \right)= \sum_{ \substack{p, q \in \bQ_+ \\ p > q}}\bP\left( S_{p} - S_{q} = 0 \right) =\sum_{ \substack{p, q \in \bQ_+ \\ p > q}}\bP\left( S_{p-q} = 0 \right) =  0.
\end{align}
Also \eqref{23.02.28.18.21} is equivalent to $\bP(S_t = 0 ) = 0$ for any $t > 0$, which means that
$$
\lim_{\lambda \to \infty} \mathrm{e}^{-t\phi(\lambda)}=\lim_{\lambda \to \infty} \bE[\mathrm{e}^{-\lambda S_t}] =\bP(S_t=0)+\lim_{\lambda \to \infty} \bE[\mathrm{e}^{-\lambda S_t}\mathbbm{1}_{\{S_t>0\}}] 
$$ for any $t > 0$. Since $\lim_{\lambda \to \infty} \bE[\mathrm{e}^{-\lambda S_t}\mathbbm{1}_{\{S_t>0\}}]=0$ by the Lebesgue's dominated convergence theorem, $\bP(S_t=0)=0$ is equivalent to $\lim_{\lambda \to \infty} \mathrm{e}^{-t\phi(\lambda)}=0$.
This proves that $S$ is strictly increasing almost surely if and only if $\lim_{\lambda \to  \infty}\phi(\lambda)=\infty$, that is, $\mu(\bR_+) = \infty$.
\end{proof}

Below are properties of subordinators in the literature, which we use frequently and crucially in this paper.
\begin{prop}
\label{23.08.14.18.03}
Let $S=(S_t)_{t\geq0}$ be a subordinator with Laplace exponent $\phi(\lambda):=\int_0^\infty(1-\mathrm{e}^{-\lambda s})\mu(\mathrm{d}s)$, where $\mu$ satisfies \eqref{2302231052}.
Suppose that $\phi\in \cI_o(0,1)$, and put $K(t)=\int_0^t\kappa(s)\mathrm{d}s$.
	\begin{enumerate}[(i)]
		\item (\cite[Lemma 2.1]{chen2017time}) There exists a null set $\mathcal{N}\subset (0,\infty)$ (under the Lebesgue measure) such that 
		$$
		\bP(S_s\geq t)=\int_0^s \bE\left[\kappa(t-S_r)\mathbbm{1}_{\{t\geq S_r\}}\right]\mathrm{d} r
		$$
		for all $s>0$ and $t\in(0,\infty)\setminus\mathcal{N}$.
	    In particular, 
		\begin{align}\label{22.11.03.10.57}
		\int_0^{l} \bP(S_s\geq t)\mathrm{d}t=\int_0^{s}\bE\left[K(l-S_r)\mathbbm{1}_{\{l\geq S_r\}}\right]\mathrm{d}r
		\end{align}
		
		\item (\cite[(2.5)]{chen2017time}) For any $t>0$,
		\begin{align}\label{22.11.03.10.51}
		\int_0^t\kappa(t-r)\bP(S_s>r)\mathrm{d} r=K(t)-\bE\left[K(t-S_s)\mathbbm{1}_{\{t\geq S_s\}}\right].
		\end{align}
		
		\item (\cite[Lemma 3.1]{chen2018heat}) Let $\phi\in \cI_{o}(0,1)$. Then there exists $N\geq 1$ such that
		\begin{align}\label{2302231059}
		\lambda \phi'(\lambda)\leq \phi(\lambda)\leq N\lambda\phi'(\lambda).
		\end{align}
		In particular, $\phi'\in\mathcal{I}_{o}(-1,0)$.

		\item(\cite[Proposition 3.3.(i)]{chen2018heat})  Suppose that $\phi\in\cI_o(0,1)$.
		Then there are constant $c_1,\,c_2>0$ such that for all $r,t\geq 0$,
		\begin{align}\label{23022310471}
		\bP\Big(S_r\geq t\big(1+\mathrm{e}^{r\phi(1/t)}\big)\Big)\leq c_1r\phi(1/t)
		\end{align}
		and
		\begin{align}\label{23022310472}
		\bP(S_r\geq t)\geq 1-\mathrm{e}^{-c_2r\phi(1/t)}.
		\end{align}
		In particular, for each $L$, there exist constants $c_{1,L}, c_{2,L}>0$ such that for all $r\phi(1/t)\leq L$,
		\begin{align}\label{23022310473}
		c_{1,L}r\phi(1/t)\leq \bP(S_r\geq t)\leq c_{2,L}r\phi(1/t).
		\end{align}
	\end{enumerate}
The constants $c_1, c_2$, $c_{1, L}$ and $c_{2, L}$ in \eqref{23022310471}--\eqref{23022310473} depends only on $K_\phi$ in Remark~\ref{rmk_bdd}.
\end{prop}

The following lemma helps us to prove Theorem~\ref{extension_non_local} by constructing appropriate $u$ and $f$ satisfying \eqref{eq0121_02}.
\begin{prop}
\label{22.11.07.13.23}
Let $\kappa:\bR_+\rightarrow \bR_+$ be a right-continuous decreasing function  with $\kappa(0+) = \infty$, $\kappa(\infty)=0$, and $\kappa\in \cI_o(-1,0)$. We denote $\mu$ a nonnegative measure on $\bR_+$ defined by $\kappa$ such that $\mu$ satisfies $\kappa(s) - \kappa(t) = \mu\left( \left[s, t\right) \right)$ for $0 < s < t < \infty$,
and 
$$
\Theta(t,\lambda):=\lambda\int_0^{\infty}\mathrm{e}^{-\lambda r}\bP(S_r\geq t)\,\mathrm{d}r,
$$
where $S = (S_t)_{t \ge 0}$ is a subordinator with Laplace exponent
$$
\phi(\lambda):=\int_0^{\infty}(1-\mathrm{e}^{-\lambda t})\mu(\mathrm{d}t).
$$
\begin{enumerate}[(i)]
    \item For $t, \lambda \in \bR_+$, the function $\Theta$ satisfies
    $$
     \int_0^t \kappa(t-s) \big( \Theta ( s, \lambda ) - \Theta ( 0, \lambda ) \big)\,\mathrm{d}s  =:  I^{\kappa} \Theta(t, \lambda) = -\int_0^t \lambda \Theta(s, \lambda)\,\mathrm{d}s,
    $$
    that is, 
    $$
    \p_t^\kappa\Theta(t,\lambda)=-\lambda\Theta(t,\lambda),
    $$
    where $\kappa(x)=\int_0^{\infty}\mathbbm{1}_{(x,\infty)}\mu(\mathrm{d}t)$.
    \item For $t,\lambda \in \bR_+$,
    $$
    \Theta(t,\lambda) \simeq 1\wedge\frac{\phi(t^{-1})}{\lambda}\,,
    $$
where the equivalence depends only on $K_\phi$ in Remark \ref{rmk_bdd}.
\end{enumerate}
\end{prop}
\begin{proof}
First, note that for fixed $t \in \bR_+$, $\bP(S_r \ge t)$ is (Borel) measurable with respect to $r \in \bR_+$ by the definition of subordinator, so $\Theta(t, \lambda)$ is well-defined.

$(i)$ It is clear that $\kappa$ and $\mu$ satisfy $\kappa(x) = \int_0^{\infty}\mathbbm{1}_{(x,\infty)}\mu(\mathrm{d}t)$ and $\Theta(0,\lambda)=1$ for all $\lambda\in\bR_+$. Then by Fubini's theorem,
\begin{align*}
    I^\kappa\big(\Theta(\cdot,\lambda)\big)(t)&=\int_0^t\kappa(t-s)\big(\Theta(s,\lambda)-1\big)\,\mathrm{d}s\\
    &=\lambda\int_0^{\infty}\mathrm{e}^{-\lambda r}\left(\int_0^t\kappa(t-s)\bP(S_r\geq s)\mathrm{d}s\right)\,\mathrm{d}r-K(t),
\end{align*}
where $K(t)=\int_0^t\kappa(s)\mathrm{d}s$.
From \eqref{22.11.03.10.51}, it is easily seen that
$$
I^\kappa\big(\Theta(\cdot,\lambda)\big)(t)=-\lambda\int_0^{\infty}\mathrm{e}^{-\lambda r}\bE[K(t-S_r)\mathbbm{1}_{\{t\geq S_r\}}]\mathrm{d}r.
$$
On the other hand, \eqref{22.11.03.10.57} and Fubini's theorem implies that
$$
\int_0^t\Theta(s,\lambda)\mathrm{d}s=\lambda\int_0^{\infty}\mathrm{e}^{-\lambda r}\left(\int_0^t\bP(S_r\geq s)\mathrm{d}s\right)\mathrm{d}r=\int_0^{\infty}\mathrm{e}^{-\lambda r}\bE[K(t-S_r)\mathbbm{1}_{\{t\geq S_r\}}]\mathrm{d}r.
$$
Therefore, we prove $\p_t^\kappa\Theta(t,\lambda)=-\lambda\Theta(t,\lambda)$ for $t,\lambda \in \bR_+$.

$(ii)$ Due to Proposition \ref{prop0129_01}, $\phi\in \cI_o(0,1)$. Then by \eqref{23022310472}, there exists a constant $c_2 = c_2(K_{\phi}) = c_2(\kappa)>0$ such that $\bP(S_r \ge t) \ge 1 - \mathrm{e}^{-c_2r\phi(t^{-1})}$ and then
$$
\Theta(t,\lambda)\geq \lambda\int_0^{\infty}\mathrm{e}^{-\lambda r}\left(1-\mathrm{e}^{-c_2 r\phi(t^{-1})}\right)\mathrm{d}r=1-\frac{\lambda}{\lambda+c_2 \phi(t^{-1})}.
$$
Obviously, one can take $c_2 < 1$ in the above estimate.
Then since $1\wedge a \simeq a/(1+a)$ for $a>0$ and
$$
\left( 1-\frac{\lambda}{\lambda+c_2\phi(t^{-1})} \right)^{-1} = \frac{\lambda}{c_2 \phi(t^{-1})} + 1 \lesssim_{c_2} \frac{\lambda}{ \phi(t^{-1})} + 1,
$$
we obtain $\Theta(t,\lambda) \gtrsim_{c_2} 1\wedge\frac{\phi(t^{-1})}{\lambda} $.

Now, we prove the opposite inequality.
Divide $\Theta$ into two parts;
$$
\Theta(t,\lambda)=\lambda\int_0^{\infty}\mathrm{e}^{-\lambda r}\bP(S_r\geq t)\mathrm{d}r = \int_0^{2/\phi(t^{-1})}\cdots+\int_{2/\phi(t^{-1})}^{\infty}\cdots=:I_1(t,\lambda)+I_2(t,\lambda).
$$
By taking $L = 2$ in \eqref{23022310473}, there exists a constant $c_{2, L} = c_{2, L}(K_{\phi}) = c_{2, L}(\kappa) > 0$ such that $\bP(S_r \ge t) \le c_{2, L} r\phi(t^{-1})$ for any $r$ and $t$ satisfying $r\phi(t^{-1}) \le 2 (= L)$. Then,
$$
I_1(t,\lambda) \leq c_{2, L} \phi(t^{-1})\lambda \int_0^{2/\phi(t^{-1})}r\mathrm{e}^{-\lambda r}\,\mathrm{d}r.
$$
By direct calculation, one can check that
$$
\lambda \int_0^{2/\phi(t^{-1})}r\mathrm{e}^{-\lambda r}\,\mathrm{d}r \le \frac{1}{\lambda}\wedge\frac{2}{\phi(t^{-1})}, 
$$
and hence $I_1(t,\lambda) \lesssim_{c_{2, L}} 1 \wedge \frac{\phi(t^{-1})}{\lambda}$.
For $I_2$, since $\bP(S_r\geq t)\leq 1$ for all $r,t\in\bR_+$,
$$
I_2(t,\lambda)\leq \lambda \int_{2/\phi(t^{-1})}^{\infty}\mathrm{e}^{-\lambda r}\mathrm{d}r=\mathrm{e}^{-2\lambda/\phi(t^{-1})}\lesssim 1\wedge\frac{\phi(t^{-1})}{\lambda},
$$
where we use the fact that $\mathrm{e}^{-t} \lesssim 1 \wedge t^{-1}$ for $t \in \bR_+$. We obtain $\Theta(t,\lambda) \lesssim_{c_{2, L}} 1\wedge\frac{\phi(t^{-1})}{\lambda} $.
The proposition is proved.
\end{proof}

We end this section by showing the equivalence of two equations:
\begin{align}
\label{23.08.15.18.24}
    \int_0^t\kappa(t-s)(u(s)-u_0)\mathrm{d}s=\int_0^tf(s)\mathrm{d}s
    \Longleftrightarrow 
    \int_0^t(u(s)-u_0)\mathrm{d}s=\int_0^t\varkappa(t-s)f(s)\mathrm{d}s,
\end{align}
where $\varkappa(t):=\int_0^{\infty}\bP(S_r\leq t)\,\mathrm{d}r$ and $S_r$ is a subordinator introduced in Proposition~\ref{23.08.14.18.03}.
By Proposition \ref{23.08.14.18.03},
\begin{equation}
\label{23.08.15.18.21}
    \int_0^t\kappa(t-s)\varkappa(\mathrm{d}s)=\int_0^{\infty}\bE[\kappa(t-S_{r})\mathbbm{1}_{t\geq S_{r}}]\mathrm{d}r=\lim_{r\to{\infty}}\bP(S_{r}\geq t)=1,
\end{equation}
and
\begin{equation}
\label{23.08.15.18.22}
    \begin{aligned}
    \int_0^t\kappa(t-s)\varkappa(s)\mathrm{d}s&=\int_0^{\infty}\int_0^t\kappa(t-s)\bP(S_r\leq s)\mathrm{d}s\mathrm{d}r\\
    &=\int_{0}^{\infty}\bE[K(t-S_r)\mathbbm{1}_{\{t\geq S_r\}}]\mathrm{d}r=\int_0^{t}\lim_{r\to\infty}\bP(S_r\geq l)\mathrm{d}l=t,
\end{aligned}
\end{equation}
where $K(t)=\int_0^t\kappa(s)\mathrm{d}s$.
Taking the integration to the left equation in \eqref{23.08.15.18.24} with respect to $\varkappa(\mathrm{d}s)$, and making use of Fubini's theorem and \eqref{23.08.15.18.21}, we have the right equation in \eqref{23.08.15.18.24}:
\begin{align*}
\int_0^t(u(l)-u_0)\mathrm{d}l&=\int_0^t\left(\int_0^{t-l}\kappa(t-l-s)\varkappa(\mathrm{d}s)\right)(u(l)-u_0)\mathrm{d}l\\
&=\int_0^t\int_0^{t-s}\kappa(t-s-l)(u(l)-u_0)\mathrm{d}l\,\varkappa(\mathrm{d}s)\\
&=\int_0^{t}\int_0^{t-s}f(l)\mathrm{d}l\,\varkappa(\mathrm{d}s)=\int_0^t\varkappa(t-l)f(l)
\mathrm{d}l.
\end{align*}
Similarly, taking the non-local derivative $\partial_t^{\kappa}$ to the right equation in \eqref{23.08.15.18.24}, and making use of Fubini's theorem and \eqref{23.08.15.18.22}, we have the left equation in \eqref{23.08.15.18.24}:
\begin{align*}
    \int_0^t\kappa(t-l)(u(l)-u_0)\mathrm{d}l&=\partial_t\left(\int_0^tK(t-l)(u(l)-u_0)\mathrm{d}l\right)\\
    &=\partial_t^{\kappa}\left(\int_0^t(u(s)-u_0)\mathrm{d}s\right)\\&=\partial_t^{\kappa}\left(\int_0^t\varkappa(t-s)f(s)\mathrm{d}s\right)\\
    &=\partial_t\left(\int_0^t\kappa(t-s)\left(\int_0^s
    \varkappa(s-l)f(l)\mathrm{d}l\right)\mathrm{d}s\right)\\
    &=\partial_t\left(\int_0^t(t-l)f(l)\mathrm{d}l\right)=\int_0^tf(l)\mathrm{d}l.
\end{align*}

\mysection{Generalized real interpolation}\label{sec_inter}

In this section, we present a generalized real interpolation in the sense that we put $\phi(t^{-1})$ rather than $t^{-\theta}$ in \eqref{inter_Knorm}.
The function $\phi$ is chosen to be of class $\cI_o(0,1)$ introduced in the first part of Section~\ref{sec_prob}.
Among various interpolation methods, $K$ and $J$ methods are our main focus.
\begin{defn}
Given two Banach spaces $A_0$ and $A_1$, we say that an ordered pair $(A_0, A_1)$ is an interpolation couple if both $A_0$ and $A_1$ are continuously embedded in the same Hausdorff topological vector space $Z$.
\end{defn}
It can be easily checked that the two subspaces of $Z$
\begin{equation*}
    \begin{gathered}
    A_0\cap A_1:=\{a\in Z: a\in A_0,\,a\in A_1\},\\
    A_1+A_1:=\{a\in Z:a=a_0+a_1,\,a_0\in A_0,\,a_1\in A_1\}
    \end{gathered}
\end{equation*}
are Banach spaces with the norm
\begin{equation*}
\begin{gathered}
\|a\|_{A_0\cap A_1}:=\max(\|a\|_{A_0},\|a\|_{A_1}),\\
\|a\|_{A_0+A_1}:=\inf\{\|a_0\|_{A_0}+\|a_1\|_{A_1}:a=a_0+a_1,\,a_0\in A_0,\,a_1\in A_1\}.
\end{gathered}
\end{equation*}

Through this section, we assume that $(A_0,A_1)$ is an interpolation couple, $\phi\in\cI_o(0,1)$, and $p\in[1,\infty]$.

For $t>0$ we define
        \begin{align*}
                    & K(t, a; A_0, A_1) := \inf \{ \|a_0\|_{A_0} + t\|a_1\|_{A_1}:a=a_0+a_1,a_0\in A_0,a_1\in A_1\},\\
                    & J(t,a;A_0,A_1):=\max_{a\in A_0\cap A_1}(\|a\|_{A_0},t\|a\|_{A_1}).
        \end{align*}
Note that $K$ and $J$ are merely the functional used in $K$-method and $J$-method of the classical interpolation theory, respectively.
For measurable functions $F:\bR_+\rightarrow [0,\infty]$, the functional $\Phi^{\phi}_{p}(F)$ is defined by
$$
\Phi^{\phi}_{p}(F):=\begin{cases}
    \left(\int_0^\infty \left(\phi(t^{-1})F(t) \right)^p \frac{\mathrm{d}t}{t}\right)^{1/p}\quad &\textrm{if}\quad p\in[1,\infty),\\
    \sup_{t>0}\phi(t^{-1})F(t)\quad &\textrm{if}\quad p=\infty.
    \end{cases}
$$
The space $(A_0, A_1)_{\phi, p}^K$ is defined by
\begin{equation*}
    \begin{gathered}
   (A_0, A_1)_{\phi, p}^K := \{a\in A_0 + A_1 : \|a\|_{(A_0, A_1)_{\phi, p}^K}:=\Phi^{\phi}_{p}(K(\cdot,a;A_0,A_1)) <\infty \}.
    \end{gathered}
\end{equation*}
The space $(A_0,A_1)_{\phi,p}^J$ is defined by
\begin{equation*}
    \begin{gathered}
    (A_0,A_1)_{\phi,p}^J:=\{a\in A_0+A_1:\|a\|_{(A_1,A_1)_{\phi,p}^J}:=\inf_{v}\Phi^{\phi}_p(J(\cdot,v(\cdot);A_0,A_1))<\infty\},
    \end{gathered}
\end{equation*}
where the infimum is over all measurable functions $v:\bR_+\to A_0\cap A_1$ satisfying $$
a=\int_0^{\infty}v(t)\frac{\mathrm{d}t}{t}.
$$
We provide concrete examples of $(A_0, A_1)_{\phi, p}^K$ and $(A_0, A_1)_{\phi, p}^J$ in Appendix \ref{23.08.09.13.53}.

By $\ell_p$, we denote the set of all real sequences $a = (a_j)_{j \in \bZ}$ satisfying $\|a\|_{\ell_p}<\infty$, where 
$$
\| a \|_{\ell_p}:=\begin{cases}
     \left(\sum_{j \in \bZ} |a_j|^p \right)^{1/p} \quad &\textrm{for} \quad p\in[1,\infty),\\
     \sup_{j \in \bZ} |a_j| \quad &\textrm{for} \quad p=\infty.
    \end{cases}
$$

The following proposition suggests sequential expression for interpolation norms: 
\begin{prop}
\label{22.10.27.16.40}\,

    \begin{enumerate}[(i)]
        \item For $a\in (A_0,A_1)_{\phi,p}^K$,
        $$
        \|a\|_{(A_0,A_1)_{\phi,p}^K} \simeq_{\phi} \|\alpha^{\phi}_K(a)\|_{\ell_p},
        $$
        where $\alpha^{\phi}_K(a):=\big(\alpha^{\phi}_{K,j}(a)\big)_{j\in\bZ}:=\big(\phi(2^{-j})K(2^j,a;A_0,A_1)\big)_{j\in\bZ}$.
        \item For $a\in(A_0,A_1)_{\phi,p}^J$,
        $$
        \|a\|_{(A_0,A_1)_{\phi,p}^J} \simeq_{\phi} \inf_{u}\|\alpha^{\phi}_J(u)\|_{\ell_p},
        $$
        where $\alpha^{\phi}_{J}(u):=(\alpha^{\phi}_{J,j})_{j\in\bZ}:=\big(\phi(2^{-j})J(2^j,u_j;A_0,A_1)\big)_{j\in\bZ}$ and the infimum is over all sequences $(u_j)_{j\in\bZ}\subset A_0 \cap A_1$ satisfying 
        $$
        a=\sum_{j\in\bZ}u_j\,\,\textrm{(convergence in }A_0+A_1).
        $$
    \end{enumerate}
\end{prop}

\begin{proof}
$(i)$ It can be easily checked that for $t\in[2^j,2^{j+1}]$, $j \in \bZ$, 
$$
K(2^{j},a;A_0,A_1)\leq K(t,a;A_0,A_1)\leq 2K(2^{j},a;A_0,A_1).
$$
Since $\phi \in \cI_o(0, 1)$, by Lemma \ref{22.10.26.15.52}-$(vi)$, we have
$$
1 \le (2^{-j}t)^{\varepsilon} \lesssim \frac{\phi(2^{-j})}{\phi(t^{-1})} \lesssim (2^{-j}t)^{1-\varepsilon} \le 2^{1-\varepsilon}\quad \textrm{for all}\,\,t\in[2^j,2^{j+1}]\,,
$$
\textit{i.e.},
\begin{equation}
\label{eq0110_04}
\phi(t^{-1}) \simeq\phi(2^{-j})\quad \textrm{for all}\,\,t\in[2^j,2^{j+1}].
\end{equation}
This certainly implies that
$$
\|a\|_{(A_0,A_1)_{\phi,p}^K}\simeq_{\phi}\|\alpha^{\phi}_K(a)\|_{\ell_p}.
$$

$(ii)$ We only prove for $p \in [1, \infty)$ since the proof of the case $p=\infty$ is similar. Since $a\in (A_0,A_1)_{\phi,p}^J$, there exists a measurable function $u:\bR_+\to A_0\cap A_1$ such that $a=\int_0^{\infty}u(t)\frac{\mathrm{d}t}{t}$. Put
$$
u_j:=\int_{2^j}^{2^{j+1}}u(t)\frac{\mathrm{d}t}{t}\in A_0\cap A_1,
$$
then $a=\sum_{j\in\bZ}u_j$. By the definition of $J$-functional and Minkowski's inequality,
$$
\alpha_{J,j}^{\phi} \le \phi(2^{-j}) \int_{2^j}^{2^{j+1}} \max (\| u\|_{A_0}, t\|u\|_{A_1})\,\frac{\mathrm{d}t}{t} \lesssim \int_{2^j}^{2^{j+1}}\phi(t^{-1})J(t,u(t);A_0,A_1)\frac{\mathrm{d}t}{t},
$$
where we use \eqref{eq0110_04} for $t \in [2^j, 2^{j+1}]$.
This implies that $\|\alpha^{\phi}_J(u)\|_{\ell_p}\lesssim_{\phi} \Phi^{\phi}_p\big(J(\cdot,u;A_0,A_1)\big)$, and then,
$$
 \inf_{u}\|\alpha^{\phi}_J(u)\|_{\ell_p}\lesssim_{\phi}\|a\|_{(A_0,A_1)_{\phi,p}^J}.
$$
For the converse, assume that there exists $u=(u_j)_{j\in\bZ }\subseteq A_0\cap A_1$ such that
$$
\alpha_J^{\phi}(u)\in\ell_p,\quad a=\sum_{j\in\bZ}u_j.
$$
Set
$$
v(t):=\frac{1}{\log 2}\sum_{j\in\bZ}u_j\mathbbm{1}_{(2^j,2^{j+1}]}(t)\in A_0\cap A_1,
$$
then $a=\int_0^{\infty}v(t)\frac{\mathrm{d}t}{t}$. Clearly, $v$ is (Bochner) measurable. Therefore, by \eqref{eq0110_04} for $t \in [2^j, 2^{j+1}]$ again, we have
\begin{align*}
\|a\|_{(A_0,A_1)_{\phi,p}^J}^p \le \Phi^{\phi}_p\big(J(\cdot,v;A_0,A_1)\big)^p&= \sum_{j\in\bZ}\int_{2^j}^{2^{j+1}}\big(\phi(t^{-1})J(t,v(t);A_0,A_1)\big)^p\frac{\mathrm{d}t}{t}\\
&\lesssim \sum_{j\in\bZ}(\phi(2^{-j})J(2^j,u_j;A_0,A_1))^p=\|\alpha_J^{\phi}(u)\|_{\ell_p}^p,
\end{align*}
and taking the infimum in $u$, we have the result. The proposition is proved.
\end{proof}

\begin{lem}
\label{lem0120_01}
For $a\in (A_0,A_1)_{\phi,p}^J$, there exists an infinitely differentiable function $v:\bR_+\to A_0\cap A_1$ such that $a=\int_0^{\infty}v(t)\frac{\mathrm{d}t}{t}$.
\end{lem}
\begin{proof}
Let $\varphi$ be a nonnegative infinitely differentiable function with compact support in $(0,\infty)$,
\begin{equation}
\label{molifier}
    \int_0^{\infty}\varphi(s^{-1})\frac{\mathrm{d}s}{s}=1.
\end{equation}
Since $a\in(A_0,A_1)_{\phi,p}^J$, there exists measurable $u:\bR_+\to A_0\cap A_1$ such that $a=\int_0^{\infty}u(t)\frac{\mathrm{d}t}{t}$. Put
$$
v(t):=\int_0^{\infty}\varphi \left( \frac{t}{s}  \right)u \left( s \right)\frac{\mathrm{d}s}{s}.
$$
Clearly, $v$ is an infinitely differentiable function from $\bR_+$ to $A_0\cap A_1$. By Fubini's theorem (\textit{e.g.} \cite[Proposition 1.2.7.]{hytonen2016analysis}) and \eqref{molifier},
$$
a=\int_0^{\infty}v(t)\frac{\mathrm{d}t}{t}.
$$
The lemma is proved.
\end{proof}

\begin{lem}
\label{22.10.27.16.32}
For any $a\in A_0+A_1$, 
$$
K(t,a;A_0,A_1)\lesssim_{\phi,p} \frac{\|a\|_{(A_0,A_1)_{\phi,p}^K}}{\phi(t^{-1})}
$$
\end{lem}

\begin{proof}
The case $p = \infty$ is trivial, so we only prove the case $p \in [1, \infty)$.
By the definition of $K$-functional, for $t, s > 0$,
$$
K(t,a;A_0,A_1)\leq \max(1,t/s)K(s,a:A_0,A_1).
$$
This also implies that
$$
\min(1,s/t)K(t,a;A_0,A_1)\leq K(s,a;A_0,A_1).
$$
Taking $\Phi^{\phi}_p$ both sides with respect to $s$,
$$
\Phi^{\phi}_p(\min(1,\cdot/t))K(t,a,;A_0,A_1)\leq \|a\|_{(A_0,A_1)_{\phi,p}^K}.
$$
Therefore it suffices to prove that $\Phi^{\phi}_p(\min(1,\cdot/t)) \gtrsim \phi(t^{-1})$, and this can be verified by Lemma \ref{22.10.26.15.52}-$(vi)$ since
$$
\frac{\Phi^{\phi}_p(\min(1,\cdot/t))}{ \phi(t^{-1})} = \left( \int_0^{\infty} \left| \frac{\phi(s^{-1})}{\phi(t^{-1})} \left(1 \wedge \frac{s}{t} \right) \right|^p\,\frac{\mathrm{d}s}{s} \right)^{1/p}.
$$
The lemma is proved.
\end{proof}

We present two lemmas to show equivalence between $K$ and $J$ methods (Theorem~\ref{thm_equiv}) and further properties of our generalized interpolation (Proposition~\ref{prop_230223}).

\begin{thm}[Equivalence theorem]\label{thm_equiv}
$(A_0,A_1)_{\phi,p}^K$ and $(A_0,A_1)_{\phi,p}^J$ have same elements.
In addition, for any $a\in (A_0,A_1)_{\phi,p}^K\big(=(A_0,A_1)_{\phi,p}^J\big)$,
$$
\|a\|_{(A_0,A_1)_{\phi,p}^K}\simeq \|a\|_{(A_0,A_1)_{\phi,p}^J}.
$$
\end{thm}
\begin{proof}
The proof is comparable to the proof of classical real interpolation. However, for the sake of completeness, we present proof.
Let $a\in(A_0,A_1)_{\phi,p}^J$. Then by Lemma \ref{lem0120_01}, there exists  an infinitely differentiable $v:\bR_+\to A_0\cap A_1$ such that $a=\int_0^{\infty}v(t)\frac{\mathrm{d}t}{t}$. Then by \cite[Lemma 3.2.1.]{1976bergh},
\begin{align*}
    K(t,a;A_0,A_1)&\leq \int_0^{\infty}K(t,v(s);A_0,A_1)\frac{\mathrm{d}s}{s}\leq \int_0^{\infty}\min(1,t/s)J(s,v(s);A_0,A_1)\frac{\mathrm{d}s}{s}\\
    &=\int_0^{\infty}\min(1,s^{-1})J(ts,v(ts);A_0,A_1)\frac{\mathrm{d}s}{s}.
\end{align*}
Since $\phi\in \mathcal{I}_o(0,1)$, by Lemma \ref{22.10.26.15.52}-$(vi)$, there exists $\varepsilon>0$ such that
\begin{equation}
\label{22.10.26.18.14}
    \frac{\phi(t^{-1}s)}{\phi(t^{-1})}\lesssim s^{\varepsilon}\mathbbm{1}_{s\leq1}+s^{1-\varepsilon}\mathbbm{1}_{s>1}.
\end{equation}
Then by \eqref{22.10.26.18.14}, for $p \in [1, \infty)$,
\begin{align*}
    \|a\|_{(A_0,A_1)_{\phi,p}^K}&\leq \int_{0}^{\infty}\left(\int_0^{\infty}\big(\phi(t^{-1})J(ts,v(ts);A_0,A_1)\big)^p\frac{\mathrm{d}t}{t}\right)^{1/p}\min(1,s^{-1})\frac{\mathrm{d}s}{s}\\
    &= \int_{0}^{\infty}\left(\int_0^{\infty}\big(\phi(t^{-1}s)J(t,v(t);A_0,A_1)\big)^p\frac{\mathrm{d}t}{t}\right)^{1/p}\min(1,s^{-1})\frac{\mathrm{d}s}{s}\\
    &\lesssim \Phi^{\phi}_p(v)\int_{0}^{\infty}(s^{\varepsilon}\mathbbm{1}_{s\leq1}+s^{1-\varepsilon}\mathbbm{1}_{s>1})\min(1,s^{-1})\frac{\mathrm{d}s}{s}\lesssim\Phi^{\phi}_p(v).
\end{align*}
Taking the infimum in $v$, we have
$$
(A_0,A_1)_{\phi,p}^J\subseteq (A_0,A_1)_{\phi,p}^K.
$$
The case $p = \infty$ is almost the same.

For the converse, note that by Lemma \ref{22.10.26.15.52}-($vi$) and Lemma \ref{22.10.27.16.32},
$$
\lim_{t \to \infty}\min(1,t^{-1})K(t,a;A_0,A_1) = \lim_{t\downarrow0}\min(1,t^{-1})K(t,a;A_0,A_1)=0,
$$
for all $a\in (A_0,A_1)_{\phi,p}^K$. Since the assumption in \cite[Lemma 3.3.2.]{1976bergh} is satisfied, by this lemma, there exists $u=(u_j)_{j\in\bZ}\subseteq A_0\cap A_1$ such that
$$
a=\sum_{j\in\bZ}u_j,\quad J(2^j,u_j;A_0,A_1)\leq 4K(2^j,a;A_0,A_1).
$$
Then by Proposition \ref{22.10.27.16.40}, we have $(A_0,A_1)_{\phi,p}^K\subseteq (A_0,A_1)_{\phi,p}^J$.
The theorem is proved.
\end{proof}
Note that the inequality \eqref{22.10.26.18.14} actually depends on $\phi$ itself.
Thus, if we put $\phi = W\circ\psi$, then we have
$$
	\|a\|_{(A_0,A_1)_{W\circ\psi,p}^K} \simeq_{W\circ\psi} \|a\|_{(A_0,A_1)_{W\circ\psi,p}^J}
$$

Due to the equivalence theorem, we now denote
\begin{equation}
\label{eq0122_01}
(A_0, A_1)_{\phi, p}:=(A_0,A_1)_{\phi,p}^K \big(=(A_0,A_1)_{\phi,p}^J\big)\quad\textrm{and}\quad \|a\|_{(A_0,A_1)_{\phi,p}}:=\|a\|_{(A_0,A_1)^K_{\phi,p}}.
\end{equation}
The space $(A_0, A_1)_{\phi, p}$ satisfies the following properties.
\begin{prop}\label{prop_230223}\,

    \begin{enumerate}[(i)]
        \item For $1\leq p_0\leq p_1\leq\infty$, we have
        $$
        A_0\cap A_1 \subset  (A_0,A_1)_{\phi,p_0} \subset (A_0,A_1)_{\phi,p_1} \subset A_0+A_1,
        $$
that is, 
        $$
        \|a\|_{A_0+A_1}\lesssim \|a\|_{(A_0,A_1)_{\phi,p_1}}\lesssim\|a\|_{(A_0,A_1)_{\phi,p_0}}\lesssim\|a\|_{A_0\cap A_1}.
        $$
        
        \item Endowed with the norm $\|\cdot\|_{(A_0,A_1)_{\phi,p}}$ the space $(A_0,A_1)_{\phi,p}$ is a Banach space.
        \item For $\tilde{\phi}(t):=t\phi(t^{-1})$, $(A_0,A_1)_{\phi,p}=(A_1,A_0)_{\tilde{\phi},p}$.
        \item Suppose that that $A_0$ be continuously embedded into $A_1$.
	If $\phi,\,\psi\in\cI_o(0,1)$ satisfies that
	\begin{align}\label{230705516}
	N_1^{-1}\phi(s)\leq \psi(s)\leq N_1\phi(s) \quad \forall\,\, 0<s\leq c\,,
	\end{align}
        for some fixed $c>0$, then $(A_0,A_1)_{\phi,p}=(A_0,A_1)_{\psi,p}$.
	Moreover, for any $a\in (A_0,A_1)_{\phi,p}\,(=(A_0,A_1)_{\psi,p})$,
	\begin{align}\label{230705651}
	\|a\|_{(A_0,A_1)_{\phi,p}}\simeq \|a\|_{(A_0,A_1)_{\psi,p}}.
	\end{align}

 \item If $a\in A_0\cap A_1$ and $a\neq 0$, then
 $$
 \|a\|_{(A_0,A_1)_{\phi,p}}\lesssim \|a\|_{A_0} \, \phi\left(\|a\|_{A_1} / \|a\|_{A_0}\right).
 $$
    \end{enumerate}
\end{prop}
\begin{proof}
$(i)$ The first and third embeddings follow from
$$
\min(1,t)\|a\|_{A_0+A_1}\leq K(t,a;A_0,A_1)\leq \min(1,t)\|a\|_{A_0\cap A_1}
$$
and Lemma \ref{22.10.26.15.52}-$(vi)$.
For the second embedding, by Lemma \ref{22.10.27.16.32}, ($ \frac{\infty}{\infty} := 1$)
\begin{align*}
    \|a\|_{(A_0,A_1)_{\phi,p_1}}&\leq \Phi_{p_0}^{\phi}(K(\cdot,a;A_0,A_1))^{p_0/p_1}\Phi_{\infty}^{\phi}(K(\cdot,a;A_0,A_1))^{1-p_0/p_1}\\
    &\lesssim \|a\|_{(A_0,A_1)_{\phi,p_0}}^{p_0/p_1}\|a\|_{(A_0,A_1)_{\phi,p_0}}^{1-p_0/p_1}=\|a\|_{(A_0,A_1)_{\phi,p_0}}. 
\end{align*}

$(ii)$ The case $p=\infty$ is clear, thus we prove completeness only for $p\in[1,\infty)$.
Let $\{a_n\}_{n=1}^{\infty}$ be a Cauchy sequence in $(A_0,A_1)_{\phi,p}$.
By $(i)$, $\{a_n\}_{n=1}^{\infty}$ is also a Cauchy sequence in $A_0+A_1$, hence there exists $\lim_{n\rightarrow \infty}a_n=: a$ in $A_0+A_1$.
For given $\varepsilon>0$, there is $N>0$ such that if $n,m\geq N$, then $\|a_n-a_m\|_{(A_0,A_1)_{\phi,p}}<\varepsilon$.
Therefore for $0<r\leq R<\infty$, we have
\begin{align*}
    &\left(\int_r^R\big(\phi(t^{-1})K(t,a-a_n;A_0,A_1)\big)^p\frac{\mathrm{d}t}{t}\right)^{1/p}\\
    &<\varepsilon+\left(\int_r^R\big(\phi(t^{-1})K(t,a-a_m;A_0,A_1)\big)^p\frac{\mathrm{d}t}{t}\right)^{1/p}.
\end{align*}
Using $K(t,a-a_m;A_0,A_1)\leq \max(1,t)\|a-a_m\|_{A_0+A_1}$ and letting $m\rightarrow \infty$, we have
$$
\left(\int_r^R\big(\phi(t^{-1})K(t,a-a_n;A_0,A_1)\big)^p\frac{\mathrm{d}t}{t}\right)^{1/p}<\varepsilon.
$$
Since $r,\,R$ are arbitrary constant, by letting $r\rightarrow0,\,R\rightarrow\infty$, we obtain that $\|a_n-a\|_{(A_0,A_1)_{\phi,p}}\leq \varepsilon$.
Therefore we conclude that $a\in (A_0,A_1)_{\phi,p}$ and $a_n$ converges to $a$ in $(A_0,A_1)_{\phi,p}$.

$(iii)$ By Lemma \ref{22.10.26.15.52} $(ix)$, $\tilde{\phi}\in \mathcal{I}_o(0,1)$.
Since
$$
K(t,a; A_1,A_0)=tK(t^{-1},a;A_0,A_1),
$$
we have
\begin{align*}
\|a\|_{(A_1,A_0)_{\tilde{\phi},p}}^p&=\int_0^{\infty}\left(\tilde{\phi}(t^{-1})K(t,a;A_1,A_0)\right)^p\frac{\mathrm{d}t}{t}\\
&=\int_0^{\infty}\left(\phi(t)K(t^{-1},a;A_0,A_1)\right)^p\frac{\mathrm{d}t}{t}=\|a\|_{(A_0,A_1)_{\phi,p}}^p.
\end{align*}

$(iv)$ Due to the definition of the interpolation spaces, we only need to prove that for any $a\in A_0+A_1$, \eqref{230705651} holds.
Let us fix $a\in A_0+A_1\,(\subseteq A_1)$.
Since \eqref{230705516} implies that
	$$
	\int_{c^{-1}}^{\infty}\big|\phi(t^{-1})K(t,a;A_0,A_1)\big|^p\,\frac{\mathrm{d}t}{t}\simeq \int_{c^{-1}}^{\infty}\big|\psi(t^{-1})K(t,a;A_0,A_1)\big|^p\,\frac{\mathrm{d}t}{t}\,,
	$$
	it is sufficient to prove that 
	\begin{align}\label{230705534}
	\int_0^{c^{-1}}\big|\phi(t^{-1})K(t,a;A_0,A_1)\big|^p\,\frac{\mathrm{d}t}{t}&\lesssim\int_0^{\infty}\big|\psi(t^{-1})K(t,a;A_0,A_1)\big|^p\,\frac{\mathrm{d}t}{t},\\
 \int_0^{c^{-1}}\big|\psi(t^{-1})K(t,a;A_0,A_1)\big|^p\,\frac{\mathrm{d}t}{t}&\lesssim\int_0^{\infty}\big|\phi(t^{-1})K(t,a;A_0,A_1)\big|^p\,\frac{\mathrm{d}t}{t}.
	\end{align}
	Due to the similarity, we only prove \eqref{230705534}.
	Since $a\in A_1$, we have
	$$
	K(t,a;A_0,A_1)\leq t\|a\|_{A_1}\quad\forall t>0\,,
	$$
	which implies
	\begin{align}\label{230705559}
	\int_0^{c^{-1}}\big|\phi(t^{-1})K(t,a;A_0,A_1)\big|^p\,\frac{\mathrm{d}t}{t}\lesssim \|a\|_{A_1}
	\end{align}
	(recall that $\phi\in\cI_o(0,1)$).
	Since $A_0$ is continuously embedded in $A_1$, we obtain that for any $a_0\in A_0\,(\subset A_1)$ and $a_1\in A_1$ with $a=a_0+a_1$,
	\begin{align*}
	\|a\|_{A_1}\leq \|a_0\|_{A_1}+\|a_1\|_{A_1}\lesssim\|a_0\|_{A_0}+\|a_1\|_{A_1}\lesssim \|a_0\|_{A_0}+t\|a_1\|_{A_1}\quad \forall\,\,t\geq c^{-1}.
	\end{align*}
	Therefore
	$$
	\|a\|_{A_1}\lesssim K(a,t;A_0,A_1)\quad \forall t\geq c^{-1}\,,
	$$
	which implies
	\begin{align}\label{230705558}
	\|a\|_{A_1}\lesssim \int_{c^{-1}}^{c^{-2}}\big|\psi(t^{-1})K(t,a;A_0,A_1)\big|^p\,\frac{\mathrm{d}t}{t}\leq \int_{c^{-1}}^{\infty}\big|\psi(t^{-1})K(t,a;A_0,A_1)\big|^p\,\frac{\mathrm{d}t}{t}
	\end{align}
	(note that $\psi(t^{-1})\simeq \psi(1)$ for all $1\leq t\leq 2$).
	By combining \eqref{230705559} and \eqref{230705558}, \eqref{230705534} is obtained.

 $(v)$ Put $\alpha:=\|a\|_{A_0}/\|a\|_{A_1}$.
 Then 
 \begin{align}\label{230920957}
 K(t,a;A_0,A_1)\leq \min(\|a\|_0,t\|a\|_{A_1})=\|a\|_{A_0}\min\Big(1,\frac{t}{\alpha}\Big)\,,
 \end{align}
 and due to $\phi\in\cI_o(0,1)$, 
 \begin{align}\label{230920958}
 \phi(t^{-1})\lesssim \phi(\alpha^{-1})\max \left( \left( \frac{\alpha}{t} \right)^{\varepsilon}, \left(\frac{\alpha}{t}\right)^{1-\varepsilon}\right)
 \end{align}
 for some $\varepsilon\in(0,1/2)$.
 \eqref{230920957} and \eqref{230920957} imply that
 $$
 \phi(t^{-1})K(t,a;A_0,A_1)\lesssim \|a\|_{A_0}\phi(\alpha^{-1})\min\Big(\big(\frac{\alpha}{t}\big)^\varepsilon,\big(\frac{t}{\alpha}\big)^\varepsilon\Big).
 $$
 Therefore, we have
 \begin{align*}
 \|a\|_{(A_0,A_1)_{\phi,p}}:=& \left( \int_0^{\infty} \left|\phi \left( t^{-1} \right) K \left( t, a; A_0, A_1\right)\right|^p \,\frac{\mathrm{d}t}{t} \right)^{1/p} \\
 \lesssim &\|a\|_{A_0}\phi(\alpha^{-1}) \left( \int_0^{\infty} \left| \min \left( \frac{\alpha}{t}, \frac{t}{\alpha}  \right)   \right|^{p\varepsilon} \,\frac{\mathrm{d}t}{t} \right)^{1/p}\\
 \lesssim &\|a\|_{A_0}\phi\left(\|a\|_{A_1} / \|a\|_{A_0}\right).
 \end{align*}
The lemma is proved.
\end{proof}

To compute interpolation norms of examples in Appendix, it might be helpful if one has a stability theorem such as Theorem 1.10.2 in \cite{triebel}.
We first give a definition for the stability theorem.
Note that for $\theta \in (0, 1)$ and $p \in [1, \infty]$, if we take $\phi(t) = t^{\theta} \in \cI_o(0, 1)$, then $(A_0, A_1)_{\phi, p}$ coincides with the classical real interpolation space $(A_0, A_1)_{\theta, p}$.
\begin{defn}

We say a Banach space $E$ belongs to the class $J(\theta, A_0, A_1)$ if 
$$
 (A_0 \cap A_1 \subset\,) \,\,\,  (A_0, A_1)_{\theta, 1} \subset E \subset A_0 + A_1,
$$
and $E$ belongs to the class $K(\theta, A_0, A_1)$ if 
$$
    A_0 \cap A_1 \subset E \subset (A_0, A_1)_{\theta, \infty}\,\,\, (\,\subset A_0 + A_1).
$$
\end{defn}

From the above notions, we have the following lemma:
\begin{lem}[\cite{triebel}, Lemma 1.10.1]\label{lem 221112 1609}
    Let $E$ be a Banach space such that $A_0\cap A_1 \subset E \subset A_0 + A_1$ and $\theta \in(0,1)$.
    \begin{enumerate}[(i)]
        \item The following statements are equivalent.
            \begin{enumerate}[(a)]
                \item $E$ is of class $J(\theta)$.
                \item There exists a positive constant $c$ such that for all $a\in A_0\cap A_1$, 
                $$
                \|a\|_E \leq c \|a\|_{A_0}^{1-\theta} \|a\|_{A_1}^\theta.
                $$
                \item There exists a positive constant $c$ such that for all $a\in A_0\cap A_1$ and $t\in (0,\infty)$, $$
                \|a\|_E \leq c t^{-\theta}J(t,a;A_0, A_1).
                $$
            \end{enumerate}
        \item $E$ is of class $K(\theta)$ if and only if there exists a positive constant $c$  such that for all $a\in E$ and $t\in (0,\infty)$, $K(t,a;A_0, A_1)  \leq c t^{\theta}\|a\|_E$.
    \end{enumerate}
\end{lem}
Using Lemma \ref{lem 221112 1609}, we state a $\phi$-analogue of the stability theorem (Theorem 1.10.2 in \cite{triebel}), and it will be used to compute examples in Appendix~.
\begin{thm}[Stability theorem]\label{thm stability}
    Let $\theta_0,\,\theta_1\in [0,1]$ with $\theta_0\neq \theta_1$, $E_j \in K(\theta_j)\cap J(\theta_j)$, $j=0,1$, and $p\in[1,\infty]$.
    Then 
    $$
        (E_0, E_1)_{\phi, p} = (A_0, A_1)_{\psi, p},\quad \psi(s) = s^{\theta_0}\phi(s^{\theta_1 - \theta_0}).
    $$
\end{thm}
\begin{proof}
 Due to Proposition~\ref{prop_230223}.(iii), it suffices to prove for the case $0\leq \theta_0<\theta_1\leq 1$.
 Note that $\psi\in \cI_o(\theta_0,\theta_1)$, and observe that
 $$
 \phi(s)=s^{-\frac{\theta_0}{\theta_1 - \theta_0}} \psi(s^{\frac{1}{\theta_1 - \theta_0}}),
 $$
    We borrow the argument in \cite[Theorem 1.10.2]{triebel} and only deal with the case $p \in [1, \infty)$ since proof for the case $p=\infty$ is much simpler.

        We first prove $(E_0, E_1)_{\phi,p} \subset (A_0, A_1)_{\psi,p}$.
        For $a\in  (E_0, E_1)_{\phi,p}$, take arbitrary $e_0 \in E_0$ and $e_1\in E_1$ satisfying $a = e_0 + e_1$.
        By Lemma \ref{lem 221112 1609}-($ii$),
        \begin{align*}
            K(t,a; A_0, A_1) 
            &\leq K(t, e_0; A_0, A_1) + K(t, e_1; A_0, A_1)\\
            &\lesssim t^{\theta_0}\|e_0\|_{E_0} + t^{\theta_1}\|e_1\|_{E_1}\\
            &= t^{\theta_0}\left( \|e_0\|_{E_0}+ t^{\theta_1 -\theta_0}\|e_1\|_{E_1}\right).
        \end{align*}
        Taking the infimum for $e_0,\,e_1$ gives
        $$
            K(t,a; A_0, A_1) \lesssim t^{\theta_0} K(t^{\theta_1 - \theta_0}, a; E_0, E_1).
        $$
        Then it follows that
        \begin{align*}
            \|a\|_{(A_0, A_1)_{\psi,p}}
            &= \left\|\psi(t^{-1}) K(t, a; A_0, A_1)\right\|_{L_p\left(\bR_+,\frac{\mathrm{d}t}{t}\right)}\\
            &\lesssim \left\|\psi(t^{-1}) t^{\theta_0 } K\big(t^{\theta_1 - \theta_0}, a; E_0, E_1\big)\right\|_{L_p\left(\bR_+,\frac{\mathrm{d}t}{t}\right)}\\
            &\simeq \left\|\psi(t^{-\frac{1}{\theta_1 - \theta_0}}) t^{\frac{\theta_0 }{\theta_1 - \theta_0}} K(t, a; E_0, E_1)\right\|_{L_p\left(\bR_+,\frac{\mathrm{d}t}{t}\right)}\\
            &= \|a\|_{(E_0, E_1)_{\phi,p}}.
        \end{align*}
        Hence we have $(E_0, E_1)_{\phi,p} \subset (A_0, A_1)_{\psi,p}$.

        We now prove $(E_0, E_1)_{\phi,p} \supset (A_0, A_1)_{\psi,p}$.
        For a fixed $a \in (A_0, A_1)_{\psi, p}$, take $u:\bR_+\rightarrow A_0\cap A_1$ such that
        \begin{align}\label{230706742}
        a = \int_0^\infty u(t^{\theta_1 - \theta_0}) \frac{\mathrm{d}t}{t},\quad \big\| \psi(t^{-1}) J(t, u(t^{\theta_1 - \theta_0}); A_0, A_1) \big\|_{L_p(\mathbb{R}_+, \frac{\mathrm{d}t}{t})}\leq 2\|a\|_{(A_0,A_1)_{\psi,p}^J}.
        \end{align}
        Since $A_0\cap A_1\subset E_0\cap E_1$, the integral in \eqref{230706742} holds in $E_0 \cap E_1$.
        By Lemma \ref{lem 221112 1609}-($c$), we obtain that for any $\tau \in (0,\infty)$
        \begin{align*}
            J(t, u(t); E_0, E_1) 
            &= \max\left( \|u(t)\|_{E_0}, t\|u(t)\|_{E_1}\right)\\
            &\lesssim \max\left( \tau^{-\theta_0}J(\tau, u(t); A_0, A_1), t \tau^{-\theta_1}J(\tau, u(t); A_0, A_1)\right).
        \end{align*}
        Take $\tau^{\theta_1 - \theta_0} =t$ so that $J(t, u(t); E_0, E_1) \lesssim t^{-\frac{\theta_0}{\theta_1 - \theta_0}} J(t^{\frac{1}{\theta_1 - \theta_0}}, u(t); A_0, A_1)$.
        Finally, it follows that
        \begin{align*}
            \|a\|_{(E_0,E_1)_{\phi,p}^J}\,&\leq \left\|\phi(t^{-1}) J(t, u(t); E_0, E_1)\right\|_{L_p\left(\bR_+,\frac{\mathrm{d}t}{t}\right)}\\
            &\lesssim \left\|t^{\frac{\theta_0}{\theta_1 - \theta_0}} \psi(t^{-\frac{1}{\theta_1 - \theta_0}}) t^{-\frac{\theta_0}{\theta_1 - \theta_0}} J(t^{\frac{1}{\theta_1 - \theta_0}}, u(t); A_0, A_1)\right\|_{L_p\left(\bR_+,\frac{\mathrm{d}t}{t}\right)}\\
            &\simeq \left\| \psi(t^{-1}) J\big(t,  u(t^{\theta_1-\theta_0}); A_0, A_1\big) \right\|_{L_p\left(\bR_+,\frac{\mathrm{d}t}{t}\right)}\leq 2\|a\|_{(A_0,A_1)_{\psi,p}^J}
        \end{align*}
        Hence we have $(A_0, A_1)_{\psi,p} \subset (E_0, E_1)_{\phi,p}$.
    The theorem is proved.
\end{proof}

\mysection{Proof of Main Results}

We begin with definitions and lemmas, which are crucially used in our proof.
Proofs of Theorems~\ref{trace_local}, \ref{trace_non_local}, \ref{extension_local} and \ref{extension_non_local} are given in Subsections~\ref{ssec_local} and \ref{ssec_nlocal}.
We first prove the results with local derivative (Theorems~\ref{trace_local} and \ref{extension_local}) in Section~\ref{ssec_local}.
The non-local results for half line case (Theorems~\ref{trace_non_local} and \ref{extension_non_local}) are proved in Section~\ref{ssec_nlocal}.

\begin{defn}
    Let $p \in [1, \infty]$, $v : \bR \to [0, \infty)$ be a measurable function, and $X$ be a Banach space. For a $X$-valued measurable function $f = f(t)$, we say $f \in L_p(\cO, v\, \mathrm{d}t;X)$ if
$$
\| f\|_{L_p(\cO, v\, \mathrm{d}t;X)}:=\begin{cases}
     \left( \int_{\cO} \|f(t)\|_{X}^p v(t) \,\mathrm{d}t \right)^{1/p} < \infty \quad &\textrm{for}\quad p\in[1,\infty),\\
    \sup_{t\in\cO} \|f(t)\|_{X} < \infty \quad &\textrm{for} \quad p=\infty.
    \end{cases}
    $$
Here $\cO$ is an open subset of $\bR_+$.
If $X = \bR$, we simply set $L_p(\cO, v\, \mathrm{d}t;X) = L_p(\cO, v\, \mathrm{d}t)$. 
\end{defn}

\begin{defn}\label{230920806}
Let $p \in (1, \infty)$. For a nonnegative measurable function $w: \bR \to [0, \infty)$, we say $w \in A_p(\bR)(=A_p)$ (Muckenhoupt $A_p$-weight) if 
$$
[w]_{A_p(\bR)} = [w]_{A_p} := \sup_{[a,b]\subset \bR} \left( \frac{1}{b-a} \int_a^b w(t) \,\mathrm{d}t \right)\left( \frac{1}{b-a} \int_a^b w(t)^{-\frac{1}{p-1}} \,\mathrm{d}t \right)^{p-1} < \infty.
$$
For $p=1$, we define
$$
[w]_{A_1(\bR)} = [w]_{A_1} := \sup_{[a,b]\subset \bR} \left[\left( \frac{1}{b-a} \int_a^b w \left( t \right) \,\mathrm{d}t \right) \sup_{t\in[a,b]}\left( w \left( t \right)^{-1}\right)\right] < \infty.
$$
For given $w \in A_p$, we denote $W(t):=\int_0^tw(s)\, \mathrm{d}s$.
\end{defn}

While this paper focuses on evolution equations in $\bR_+$, for convenience, we consider Muckenhoupt $A_p$-weight on $\bR$.
It is worth noting that if a function $w_0$, defined on $[0,\infty)$, is a Muckenhoupt $A\big([0,\infty)\big)$-weight (as defined by Definition \ref{230920806} with $\bR$ replaced by $[0,\infty)$), then the even extension of $w_0$ is a Muckenhoupt $A_p(\bR)$-weight.

\begin{rem}
\label{22.10.28.18.21}
Let $p \in (1,\infty)$ and $w\in A_p$. Then we claim that $W^{1/p}\in\mathcal{I}_o(0,1)$.
It suffices to prove that there exists $\varepsilon = \varepsilon(p, [w]_{A_p}) > 0$ such that
\begin{equation}
\label{eq0126_01}
\lambda^{\varepsilon} \lesssim_{p, [w]_{A_p}} \frac{W(\lambda t)}{W(t)} \lesssim_{p, [w]_{A_p}} \lambda^{p - \varepsilon}
\end{equation}
for $t \in \bR_+$ and $\lambda \ge 1$.
By \cite[Corollary 7.2.8.]{grafakos2014classical}, there exists $\delta = \delta(p, [w]_{A_p})\in(0,1)$ such that for $\lambda\in[0,1]$,
\begin{equation}
\label{22.10.28.13.57}
    s_{W}(\lambda):=\sup_{t>0}\frac{W(\lambda t)}{W(t)}\leq C\lambda^{\delta},
\end{equation}
where $C=C(p,[w]_{A_p}) > 0$. 
On the other hand, by \cite[Corollary 7.2.6.]{grafakos2014classical}, there exists $\delta_0 = \delta_0(p, [w]_{A_p}) \in (0, 1)$ such that $p - \delta_0 \ge 1$ and $w\in A_{p-\delta_0}$. In the virtue of \cite[Proposition 7.1.5.-(8)]{grafakos2014classical} (by taking $f = \mathbbm{1}_{(0, t)}$ and $Q = (0, \lambda t)$ in there), for $\lambda\in[1,\infty)$,
\begin{equation}
\label{22.10.28.13.58}
    W(\lambda t)
    \leq \lambda^{p-\delta_0} [w]_{A_{p-\delta_0}} W(t) 
    \lesssim_{p, [w]_{A_p}} \lambda^{p-\delta_0}W(t). 
\end{equation}
By combining \eqref{22.10.28.13.57}, \eqref{22.10.28.13.58}, and by taking $\varepsilon = \delta \wedge \delta_0$, 
we obtain  \eqref{eq0126_01}, thus $W^{1/p} \in \cI_o(0, 1)$ is also.
\end{rem}

\begin{defn}\label{23.08.08.16.02}
For $h:\bR\to \bR$, the (uncentered) Hardy-Littlewood maximal operator $\cM$ is defined by
$$
\cM h(t):=\sup_{[a,b]\ni t}\frac{1}{b-a}\int_{a}^b|h(s)|\mathrm{d}s.
$$
\end{defn}
It is well-known that for $p\in(1,\infty)$, $\cM : L_p(\bR, w\,\mathrm{d}t)\to L_p(\bR, w\,\mathrm{d}t)$ is bounded if and only if $w\in A_p$.
Moreover,
$$
\|\cM\|_{L_p(\bR, w\,\mathrm{d}t)\to L_p(\bR, w\,\mathrm{d}t)}\lesssim_{p} [w]_{A_p}^{\frac{1}{p-1}}
$$
For more detail, see \cite[Chapter 7]{grafakos2014classical}.

\begin{lem}[\cite{muckenhoupt1972hardy}, Theorems 1 and 2]
\label{two weight hardy}
Let $p\in[1,\infty]$ and
$$
F_0(\lambda):=\int_0^{\lambda}f(s)\mathrm{d}s,\quad F_{\infty}(\lambda):=\int_{\lambda}^{\infty}f(s)\mathrm{d}s. 
$$
Then
\begin{enumerate}[(i)]
    \item there exists a constant $C_0 > 0$ such that
    $$
    \|F_0U\|_{L_p(\bR_+)}\leq C_0\|fV\|_{L_p(\bR_+)}
    $$
    if and only if
    \begin{align}
    \label{two weight 1}
        B_0:= \sup_{r\in\bR_+}\|U\|_{L_p((r,\infty))}\|1/V\|_{L_{p'}((0,r))}<\infty.
    \end{align}
    Moreover, if \eqref{two weight 1} is true, then one can take $C_0$ such that $C_0\lesssim_p B_0$.
    \item there exists a constant $C_{\infty} > 0$ such that
    $$
    \|F_{\infty}U\|_{L_p(\bR_+)}\leq C_{\infty}\|fV\|_{L_p(\bR_+)}
    $$
    if and only if
    \begin{align}
    \label{two weight 2}
        B_{\infty}:=\sup_{r\in\bR_+}\|U\|_{L_p((0,r))}\|1/V\|_{L_{p'}((r,\infty))}<\infty.
    \end{align}
    Moreover, if \eqref{two weight 2} is true, one can take $C_{\infty}$ such that $C_{\infty}\lesssim_p B_{\infty}$.
\end{enumerate}
\end{lem}

\begin{lem}
\label{lem0114_01}
For any measurable function $h:\bR_+\rightarrow [0,\infty]$ and $s>0$,
$$
g(s) := \int_0^{\infty}te^{-t}h(st)\mathrm{d}t \lesssim \cM(h\mathbbm{1}_{\bR_+})(s).
$$
Here $\cM$ is the (uncentered) Hardy-Littlewood maximal operator defined in Definition \ref{23.08.08.16.02}, and the inequality is independent of $h$ and $s$.
\end{lem}

\begin{proof}
Observe that $g(s)= \int_0^1 \ldots + \int_1^{\infty} \ldots$ and
$$
\int_0^{1} te^{-t}h(st)\, \mathrm{d}t = \frac{1}{s} \int_0^{s} \frac{t}{s}e^{-t/s}h(t)\, \mathrm{d}t \lesssim  \frac{1}{2s} \int_0^{2s} h(t)\,\mathrm{d}t \le \cM(h\mathbbm{1}_{\bR_+})(s).
$$
On the other hand, 
$$
\int_1^{\infty} te^{-t}h(st)\, \mathrm{d}t \lesssim \int_1^{\infty} e^{-t/2} h(st)\,\mathrm{d}t  = 2 \int_1^{\infty} \left( \int_t^{\infty} e^{-x/2}\,\mathrm{d}x\right) h(st)\,\mathrm{d}t
$$
$$
 \simeq \int_1^{\infty} \left( \int_1^x h(st) \,\mathrm{d}t \right) e^{-x/2} \,\mathrm{d}x = \int_1^{\infty} \left( \int_s^{xs} h(t)  \,\mathrm{d}t \right) \frac{e^{-x/2}}{s} \,\mathrm{d}x 
 $$
 $$
 \le \int_1^{\infty} \frac{xs}{s}\left( \frac{1}{xs}\int_{0}^{xs} h(t)  \,\mathrm{d}t \right) e^{-x/2} \,\mathrm{d}x \le  \cM(h\mathbbm{1}_{\bR_+})(s) \int_1^{\infty} xe^{-x/2}\,\mathrm{d}x \simeq \cM(h\mathbbm{1}_{\bR_+})(s).
$$
The lemma is proved.
\end{proof}

In the following subsections, we identify the (optimal) initial trace spaces for evolutionary equations such as \eqref{eq0121_01} and \eqref{eq0121_02}. More precisely, for $u \in L_p(\bR_+, w\, \mathrm{d}t;X_0)$ such that $\partial_t u$ or $\partial_t^{\kappa} u \in L_p(\bR_+, w\, \mathrm{d}t; X_1)$ in a proper sense, we prove which space $u(0)$ belongs to in a trace sense by using the tools from Sections \ref{sec_prob} and \ref{sec_inter}.

\subsection{Proof of Theorems~\ref{trace_local} and \ref{extension_local}: local derivative case}\label{ssec_local}

In this subsection, we prove the trace theorem (Theorem~\ref{trace_local}) for functions $u \in L_p(\bR_+, w\,\mathrm{d}t; X_0)$ with the following expression
$$
u(t) = u_0 + \int_0^t f(s)\,\mathrm{d}s
$$
for some $u_0 \in X_0 + X_1$ and $f \in L_p(\bR_+, w\,\mathrm{d}t; X_1)$. 
We also prove the corresponding extension theorem (Theorem \ref{extension_local}). 
Note that in the proofs of Theorems \ref{trace_local} and \ref{extension_local}, we do not need the results from Section \ref{volterra}.

\subsubsection{Proof of Theorem~\ref{trace_local}}

Observe that taking the Laplace transform in \eqref{local_rep} yields
\begin{align}\label{23.02.20.18.16}
	u_0 = \lambda\cL[u](\lambda) - \cL[f](\lambda).
\end{align}
By \eqref{23.02.20.18.16} and \eqref{eq0122_01}, we have
\begin{align}
\label{22.10.30.17.36}
    K(\lambda,u_0;X_0,X_1)\leq \lambda\cL[\|u\|_{X_0}](\lambda)+\lambda\cL[\|f\|_{X_1}](\lambda).
\end{align}
Then it follows that
\begin{equation}\label{23.02.20.18.21}
\begin{aligned}
	\|u_0\|_{(X_0,X_1)_{W^{1/p},p}}^p
	&= \int_0^\infty W(\lambda^{-1}) K(\lambda, u_0 ; X_0, X_1)^p~\frac{\mathrm{d}\lambda}{\lambda}\\
	&\lesssim \int_0^\infty W(\lambda^{-1}) |\lambda \cL[\|u\|_{X_0}](\lambda)|^p~\frac{\mathrm{d}\lambda}{\lambda}\\
	&\quad+ \int_0^\infty W(\lambda^{-1}) |\lambda \cL[\|f\|_{X_1}](\lambda)|^p~\frac{\mathrm{d}\lambda}{\lambda}.
\end{aligned}
\end{equation}
Due to \eqref{23.02.20.18.21}, it suffices for \eqref{trace 1} to show that for nonnegative $h \in L_p(\bR_+,w\,\mathrm{d}t)$,
\begin{align}
\label{22.10.30.17.46}
    \left(\int_{0}^{\infty}W(\lambda^{-1})|\lambda\cL[h](\lambda)|^p\frac{\mathrm{d}\lambda}{\lambda}\right)^{1/p}\lesssim \|h\|_{L_p(\bR_+,w\,\mathrm{d}t)},
\end{align}
where $\cL[h](\lambda) = \int_0^{\infty}e^{-\lambda t}h(t) \,\mathrm{d}t$.
Note that \eqref{22.10.30.17.46} also verifies a well-definedness of \eqref{22.10.30.17.36} in $X_0+X_1$.

To prove \eqref{22.10.30.17.46}, we use Lemma~\ref{lem0114_01} and assume the following lemma for a moment:
\begin{lem}
    \label{trace lem1}
    Let $p\in(1,\infty)$ and $w\in A_p$. Then for $f \in L_p(\bR, w\,\mathrm{d}t)$, we have
    \begin{align*}
        \int_0^{\infty}|F(t)|^p\frac{W(t)}{t^{p+1}}\, \mathrm{d}t \lesssim_{p, [w]_{A_p}} \int_0^{\infty}|f(t)|^pw(t)\, \mathrm{d}t,
    \end{align*}
    where $F(t):=\int_0^tf(s)\, \mathrm{d}s$.
\end{lem}
By changing variables $\lambda \to \lambda^{-1}$,
$$
\int_{0}^{\infty}W(\lambda^{-1})|\lambda\cL[h](\lambda)|^p\frac{\mathrm{d}\lambda}{\lambda}
=\int_0^{\infty}|\cL[h](\lambda^{-1})|^p\frac{W(\lambda)}{\lambda^{p+1}}\mathrm{d}\lambda.
$$
Since
$$
\cL[h](\lambda^{-1})=\int_0^{\lambda}g(s)\mathrm{d}s,\quad 
g(s):=\int_0^{\infty}te^{-t}h(st)\mathrm{d}t,
$$
by Lemma \ref{lem0114_01}, we have
$$
\int_{0}^{\infty}W(\lambda^{-1})|\lambda\cL[h](\lambda)|^p\frac{\mathrm{d}\lambda}{\lambda}
\lesssim \int_0^{\infty}\left|\int_0^{\lambda}\cM (h\mathbbm{1}_{\bR_+})(s)\mathrm{d}s\right|^p\frac{W(\lambda)}{\lambda^{p+1}}\mathrm{d}s.
$$
On the other hand, by Lemma \ref{trace lem1} and $L_p$ estimates for Hardy-Littlewood maximal function with $A_p$-weights,
\begin{align*}
\int_0^{\infty}\left|\int_0^{\lambda}\cM (h\mathbbm{1}_{\bR_+})(s)\mathrm{d}s\right|^p\frac{W(\lambda)}{\lambda^{p+1}}\mathrm{d}\lambda 
&\lesssim_{p, [w]_{A_p}} \int_0^{\infty}|\cM (h\mathbbm{1}_{\bR_+})(s)|^pw(s)\mathrm{d}s\\
&\lesssim_{p, [w]_{A_p}} \|h\|_{L_p(\bR_+,w\,\mathrm{d}t)}^p.
\end{align*}
This certainly implies \eqref{22.10.30.17.46}.

We prove Lemma~\ref{trace lem1} and complete the proof of Theorem~\ref{trace_local}
\begin{proof}[Proof of Lemma~\ref{trace lem1}]
With the help of Theorem \ref{two weight hardy}-($i$), it is enough to check
$$
\sup_{r>0}\left(\int_{r}^{\infty}\frac{W(t)}{t^{p+1}}\mathrm{d}t\right)^{1/p}\left(\int_0^rw(t)^{-\frac{1}{p-1}}\mathrm{d}t\right)^{1/p'}<\infty.
$$
By Lemma \ref{22.10.26.15.52}-($vi$) and \eqref{eq0126_01} of Lemma \ref{22.10.28.18.21},
$$
\int_r^{\infty}\frac{W(t)}{t^{p+1}}\mathrm{d}t\lesssim_{p, [w]_{A_p}} r^{-p} W(r) = r^{-p} \int_0^r w(s) \, \mathrm{d}s.
$$
Therefore,
$$
\left(\int_{r}^{\infty}\frac{W(t)}{t^{p+1}}\mathrm{d}t\right)^{1/p} \left(\int_0^rw(t)^{-\frac{1}{p-1}}\mathrm{d}t\right)^{1/p'} \lesssim \left( \frac{1}{r} \int_0^r w(s) \, \mathrm{d}s \right)^{1/p} \left( \frac{1}{r} \int_0^rw(t)^{-\frac{1}{p-1}}\mathrm{d}t\right)^{1/p'},
$$
and then,
$$
\sup_{r>0} \left(\int_{r}^{\infty}\frac{W(t)}{t^{p+1}}\mathrm{d}t\right)^{1/p}\left(\int_0^rw(t)^{-\frac{1}{p-1}}\mathrm{d}t\right)^{1/p'} \lesssim  [w]_{A_p}^{1/p}.
$$
The lemma is proved.
\end{proof}

\subsubsection{Proof of Theorem~\ref{extension_local}}

Since 
$$
a \in (X_0,X_1)_{W^{1/p},p} = (X_0,X_1)_{W^{1/p},p}^J,
$$
there exists a measurable function $v:\bR_+\to X_0\cap X_1$ such that
$$
a=\int_0^{\infty}v(\lambda)\frac{\mathrm{d}\lambda}{\lambda} \quad \textrm{and} \quad \Phi_{p}^{W^{1/p}} \left(J\left( \cdot,v\left( \cdot \right);X_0,X_1 \right) \right)\leq2\|a\|_{(X_0,X_1)_{W^{1/p},p}^J} \big(\simeq \|a\|_{(X_0,X_1)_{W^{1/p},p}}\big).
$$
Let
$$
u(t):=\int_0^{\infty}\mathrm{e}^{-t\lambda}v(\lambda)\frac{\mathrm{d}\lambda}{\lambda},\quad f(t):=-\int_0^{\infty}\lambda \mathrm{e}^{-t\lambda}v(\lambda)\frac{\mathrm{d}\lambda}{\lambda}
$$
for $t \in [0, \infty)$ and $t \in \bR_+$, respectively. 
Then $u(t) - \int_0^t f(s)\mathrm{d}s = a$ for all $t\geq0$, since 
$$
\int_0^t f(s)~\mathrm{d}s = -\int_0^\infty (1- \mathrm{e}^{-t\lambda}) v(\lambda) \frac{\mathrm{d}\lambda}{\lambda}.
$$
It is clear that $u(0) = a$ and
\begin{align}\label{23.02.22.20.40}
\|u(t)\|_{X_0}+\|f(t)\|_{X_1}\lesssim \int_0^{\infty}\mathrm{e}^{-t\lambda}J(\lambda,v(\lambda);X_0,X_1)\frac{\mathrm{d}\lambda}{\lambda}
\end{align}
for $t \in \bR_+$. 
Note that the right-hand side of \eqref{23.02.22.20.40} is a form of $\int_0^{\infty}\mathrm{e}^{-t\lambda}f(\lambda)\frac{\mathrm{d}\lambda}{\lambda}$.
For this type of quantity, we assume the following lemma:
\begin{lem}
\label{extension lem1}
Let $p\in(1,\infty)$ and $w\in A_p$. Then for $f \in L_{p}(\bR_+, W(t^{-1})\frac{\mathrm{d}t}{t})$, we have
\begin{align*}
    \int_{0}^{\infty}|F(t)|^pw(t)\mathrm{d}t \lesssim_{p, [w]_{A_p}} \int_0^{\infty}|f(t)|^pW(t^{-1})\frac{\mathrm{d}t}{t},
\end{align*}
where $F(t) :=\int_0^{\infty}\mathrm{e}^{-t\lambda}f(\lambda)\frac{\mathrm{d}\lambda}{\lambda}$.
\end{lem}
Then by Lemma \ref{extension lem1} with $J(\lambda, v(\lambda); X_0, X_1)$ in place of $f(\lambda)$,
\begin{align*}
    &\|u\|_{L_p(\bR_+,w\,\mathrm{d}t;X_0)}+\|f\|_{L_p(\bR_+,w\,\mathrm{d}t;X_1)}\\
    &\lesssim \left\|\int_0^{\infty} \exp(-\cdot\lambda)J(\lambda,v(\lambda);X_0,X_1)\frac{\mathrm{d}\lambda}{\lambda}\right\|_{L_p(\bR_+,w\,\mathrm{d}t)}\\
    &\lesssim_{p, [w]_{A_p}} \Phi_p^{W^{1/p}}\left(J\left( \cdot,v\left( \cdot \right);X_0,X_1 \right) \right) \lesssim \|a\|_{(X_0,X_1)_{W^{1/p},p}}.
\end{align*}
Thus $u$ and $f$ are well-defined in $L_p(\bR_+,w\,\mathrm{d}t;X_0)$ and $L_p(\bR_+,w\,\mathrm{d}t;X_1)$, respectively. Also, by Fubini's theorem (\textit{e.g.} \cite[Proposition 1.2.7.]{hytonen2016analysis}), we have
$$
u(t)=a+\int_0^tf(s)\, \mathrm{d}s,
$$
for $t \in \bR_+$.

We prove Lemma~\ref{extension lem1} and complete the proof of Theorem~\ref{extension_local}.
\begin{proof}[Proof of Lemma~\ref{extension lem1}]
Without loss of generality, we assume that $f$ is nonnegative. First we divide $F$ into two parts;
$$
F(t)=\int_0^t\mathrm{e}^{-t / \lambda}f(\lambda^{-1})\frac{\mathrm{d}\lambda}{\lambda}+\int_{t}^{\infty}\mathrm{e}^{-t / \lambda}f(\lambda^{-1})\frac{\mathrm{d}\lambda}{\lambda}=:F_1(t)+F_2(t).
$$
Since $e^{-x} \leq 1 \wedge x^{-1}$ for all $x\in\bR_+$, we have
$$
F_1(t)\leq \frac{1}{t}\int_0^t f(\lambda^{-1})\mathrm{d}\lambda,\quad F_2(t)\leq \int_t^{\infty}f(\lambda^{-1})\frac{\mathrm{d}\lambda}{\lambda}
$$
\textbf{Estimate of $F_1$}: We claim that
\begin{align}
\label{f1 est}
    \int_{0}^{\infty}|F_1(t)|^pw(t)\mathrm{d}t\leq C\int_0^{\infty}|f(t)|^pW(t^{-1})\frac{\mathrm{d}t}{t}
\end{align}
for some $C = C(p, [w]_{A_p}) > 0$. To obtain \eqref{f1 est}, it is enough to verify that \eqref{two weight 1} holds for
$$
U(t):=\frac{w(t)^{1/p}}{t},\quad \textrm{and} \quad V(t):=\frac{W(t)^{1/p}}{t^{1/p}}
$$
with the help of Lemma \ref{two weight hardy}-($i$). (Take $t F_1(t)$ and $f(t^{-1})$ in place of $F_0$ and $f$, respectively.)
Since $t^{-p}=\int_0^{t^{-1}}ps^{p-1}\, \mathrm{d}s$, we have
\begin{align*}
    \int_r^{\infty}|U(t)|^p\mathrm{d}t=p\int_0^{r^{-1}}\left(\int_{r}^{s^{-1}}w(t)\mathrm{d}t\right)s^{p-1}\mathrm{d}s\leq p\int_0^{r^{-1}}W(s^{-1})s^{p-1}\mathrm{d}s.
\end{align*}
Note that by Lemma \ref{22.10.26.15.52}-($vi$) and  \eqref{eq0126_01} of Lemma \ref{22.10.28.18.21}, for $s \in (0, r^{-1})$ we have
$$
\frac{W(s^{-1})}{W(r)} = \frac{W(s^{-1}r^{-1}r)}{W(r)} \lesssim_{p, [w]_{A_p}} (s^{-1}r^{-1})^{p-\varepsilon},
$$
for some $\varepsilon = \varepsilon(p, [w]_{A_p}) >0$ and then, $\int_0^{r^{-1}}W(s^{-1})s^{p-1}\mathrm{d}s\lesssim_{\varepsilon,p} r^{-p}W(r)$. This implies that
\begin{align}
\label{22.10.30.16.25}
    \|U\|_{L_p((r,\infty))}\lesssim_{p, [w]_{A_p}}r^{-1}W(r)^{1/p}.
\end{align}
For $0<t<r$, by Lemma \ref{22.10.26.15.52}-($vi$) and \eqref{eq0126_01} of Lemma \ref{22.10.28.18.21} again,
$$
\left| \frac{1}{V(t)} \right|^p \lesssim_{p, [w]_{A_p}} W(r)^{-1}r^{p-\varepsilon}t^{1+\varepsilon-p},
$$
and then,
\begin{align}
\label{22.10.30.17.10}
    \|1/V\|_{L_{p'}((0,r))}\lesssim_{p, [w]_{A_p}} rW(r)^{-1/p}.
\end{align}
Combining \eqref{22.10.30.16.25} and \eqref{22.10.30.17.10}, we have \eqref{two weight 1}.

\textbf{Estimate of $F_2$}:
Now we prove that
\begin{align}
\label{f2 est}
    \int_{0}^{\infty}|F_2(t)|^pw(t)\mathrm{d}t\leq C\int_0^{\infty}|f(t)|^pW(t^{-1})\frac{\mathrm{d}t}{t}
\end{align}
for some $C = C(p, [w]_{A_p}) > 0$.
If we prove \eqref{two weight 2} for
$$
U(t):=w(t)^{1/p},\quad V(t):=W(t)^{1/p}t^{1-\frac{1}{p}},
$$
then by Theorem \ref{two weight hardy}-($ii$), we obtain \eqref{f2 est}.
Clearly,
\begin{equation}
\label{eq0117_01}
\|U\|_{L_p((0,r))}= W(r)^{1/p}.
\end{equation}
For $0<r<t$, by Lemma \ref{22.10.26.15.52}-($vi$) and \eqref{eq0126_01} of Lemma \ref{22.10.28.18.21} again,
$$
\left| \frac{1}{V(t)} \right|^p \lesssim_{p, [w]_{A_p}} W(r)^{-1}r^{\varepsilon}t^{1-\varepsilon-p},
$$
and then,
\begin{equation}
\label{eq0117_02}
\|1/V\|_{L_{p'}((r,\infty))}\lesssim_{p, [w]_{A_p}} W(r)^{-1/p}.
\end{equation}
Combining \eqref{eq0117_01} and \eqref{eq0117_02}, we obtain \eqref{two weight 2}.
The lemma is proved.
\end{proof}

\subsection{Proof of Theorems~\ref{trace_non_local} and \ref{extension_non_local}: non-local derivative case}\label{ssec_nlocal}
In this subsection, we prove the trace theorem (Theorems~\ref{trace_non_local}) for functions $u$, roughly speaking, involving a non-local derivative. 
We also provide the corresponding extension theorem (Theorem \ref{extension_non_local}).

We briefly recall the essential tools from Section \ref{volterra}.

$\bullet$ $\kappa:\bR_+\rightarrow \bR_+$ is a right-continuous decreasing function with $\kappa(0+) = \infty$, $\kappa(\infty)=0$, and $\int_0^1 \kappa(s) \, \mathrm{d}s < \infty$. For the remainder of this section, we fix a $\kappa$ that satisfies the above assumptions.  Also, see Remark \ref{rmk_0129_01} for a sufficient condition for the above assumptions.

$\bullet$ There is a nonnegative measure $\mu$ defined on $\bR_+$ such that $\kappa(x)=\int_0^{\infty}\mathbbm{1}_{(x,\infty)}\mu(\mathrm{d}t)$. By the choice of $\kappa$, this $\mu$ satisfies \eqref{2211041216}.

$\bullet$ Since $\mu$ satisfies \eqref{2211041216}, $\phi(\lambda) =\int_0^{\infty}\left(1-e^{-\lambda t}\right)\mu(\mathrm{d}t)$ is a Bernstein function.

$\bullet$ There is a subordinator $S=(S_t)_{t\geq0}$ defined on an probability space $(\Omega, \cF, \bP)$ with Laplace exponent $\phi$, \textit{i.e.},   $\bE\big[\mathrm{e}^{-\lambda S_t}\big]=\mathrm{e}^{-t\phi(\lambda)}$ for any $t,\,\lambda\geq 0$.

$\bullet$ For $(t, \lambda) \in [0, \infty) \times \bR_+$, $\Theta(t,\lambda) := \lambda\int_0^{\infty}\mathrm{e}^{-\lambda r}\bP(S_r\geq t)\mathrm{d}r$ satisfies $\Theta(0,\lambda)=1$ and $\partial_t^{\kappa}\Theta(t,\lambda)=\lambda \Theta(t,\lambda)$.

It is given in Proposition~\ref{prop0129_01} that for $\kappa$ mentioned above, 
if $\kappa \in \cI_o(-1, 0)$, then $\phi(\lambda):=\lambda \cL[\kappa](\lambda)$ is a Bernstein function and $\phi(\lambda)\simeq \kappa(\lambda^{-1})$.
For this $\phi$, we denote
\begin{align}
\label{23.03.02.15.43}
    \psi(t) := \frac{1}{\phi^{-1}(t^{-1})},
\end{align}
where $\phi^{-1}$, the inverse function of $\phi$, is well-defined since Bernstein functions are continuous and strictly increasing.
It is provided in Remark \ref{23.03.02.15.47} that $\psi\simeq \kappa^{*}$, where $\kappa^{*}$ denotes the function defined as \eqref{23.08.11.17.26}. 
Due to $W(t):=\int_0^t w(s)\dd s \in \cI_o(0,p)$, we have $W\circ \psi\simeq W\circ \kappa^{*}$,
so that
$$
(X_0,X_1)_{(W\circ\kappa^{*})^{1/p},p}=(X_0,X_1)_{(W\circ\psi)^{1/p},p}.
$$
Therefore, we prove Theorems~\ref{trace_non_local} and \ref{extension_non_local} for $(X_0,X_1)_{(W\circ\psi)^{1/p},p}$ instead of $(X_0,X_1)_{(W\circ\kappa^{*})^{1/p},p}$.

We first prove a non-local counterpart of Theorem \ref{trace_local}.

\subsubsection{Proof of Theorem~\ref{trace_non_local}}

Note that the space $(X_0,X_1)_{(W\circ\psi)^{1/p},p}$ is well-defined due to the assumption $(W\circ \psi)^{1/p} \in \cI_o(0, 1)$.

Similar to the proof of Theorem~\ref{trace_local}, we take $\cL$ to the both sides of \eqref{eq0128_01} and use Fubini's theorem (\textit{e.g.} \cite[Proposition 1.2.7.]{hytonen2016analysis}) to obtain
$$
\cL[\kappa](\lambda) \left( \cL[u](\lambda) - \frac{u_0}{\lambda}  \right) = \frac{\phi(\lambda)}{\lambda}\left( \cL[u](\lambda) - \frac{u_0}{\lambda}  \right) = \frac{1}{\lambda}\cL[f](\lambda),\quad \lambda\in\bR_+,
$$
where we use \eqref{22.11.02.14.24} for the first equality.
That is, we have
\begin{align}\label{nlocal_rep1}
u_0 = \lambda \cL[u](\lambda) + \frac{\lambda \cL[f](\lambda)}{\phi(\lambda)},\quad \lambda \in \bR_+,
\end{align}
or, equivalently, (Put $\lambda = \phi^{-1}(\tau)(=1/\psi(\tau^{-1}))$.)
\begin{align}\label{nlocal_rep2}
u_0=\frac{\cL[u](1/\psi(\tau^{-1}))}{\psi(\tau^{-1})}+\frac{\cL[f](1/\psi(\tau^{-1}))}{\tau\psi(\tau^{-1})},\quad \tau \in \bR_+.
\end{align}
Then by the definition of $K$-functional, 
\begin{align}
\label{22.11.08.14.20}
    K(\tau,u_0;X_0,X_1)\leq \frac{\cL[\|u\|_{X_0}+\|f\|_{X_1}](1/\psi(\tau^{-1}))}{\psi(\tau^{-1})},
\end{align}
and it follows from \eqref{22.11.08.14.20} that
\begin{equation}\label{23.02.20.20.30}
\begin{aligned}
	\|u_0\|_{(X_0,X_1)_{(W\circ\psi)^{1/p},p}}^p
	&= \int_0^\infty (W\circ\psi)(\tau^{-1})  \left|K(\tau, u_0 ; X_0, X_1)\right|^p~\frac{\mathrm{d}\tau}{\tau}\\
	&\lesssim \int_0^\infty (W\circ\psi)(\tau^{-1}) \left|\frac{\cL[\|u\|_{X_0}+\|f\|_{X_1}](1/\psi(\tau^{-1}))}{\psi(\tau^{-1})} \right|^p~\frac{\mathrm{d}\tau}{\tau}
\end{aligned}
\end{equation}
Thus it suffices for \eqref{trace 2} to show that for nonnegative $h$,
\begin{align}
\label{22.11.08.14.24}
    \left(\int_{0}^{\infty}\left| (W\circ\psi)(\tau^{-1}) \frac{\cL[h](1/\psi(\tau^{-1}))}{\psi(\tau^{-1})}\right|^p\frac{\mathrm{d}\tau}{\tau}\right)^{1/p}\lesssim_{p, [w]_{A_p}} \|h\|_{L_p(\bR_+,w\,\mathrm{d}t)},
\end{align}
where $\cL[h](\lambda) = \int_0^{\infty}\mathrm{e}^{-\lambda t}h(t) \,\mathrm{d}t$.
Note that \eqref{nlocal_rep1} and \eqref{nlocal_rep2} are well-defined in $X_0+X_1$ due to \eqref{22.11.08.14.24}.

By changing variable $\tau \to \phi(\lambda^{-1})$ with the help of \eqref{2302231059}, we have
$$
\int_{0}^{\infty}\left| (W\circ\psi)(\tau^{-1}) \frac{\cL[h](1/\psi(\tau^{-1}))}{\psi(\tau^{-1})}\right|^p \, \frac{\mathrm{d}\tau}{\tau} \simeq_{\kappa} \int_0^{\infty}|\cL[h](\lambda^{-1})|^p\frac{W(\lambda)}{\lambda^{p+1}} \, \mathrm{d}\lambda.
$$
We also have
$$
\cL[h](\lambda^{-1}) = \int_0^{\lambda} \left( \int_0^{\infty}t\mathrm{e}^{-t}h(st) \, \mathrm{d}t \right) \mathrm{d}s := \int_0^{\lambda}g(s)\,\mathrm{d}s 
$$
and $g(s)\lesssim \cM(h\mathbbm{1}_{\bR_+})(s)$ for $s \in \bR_+$.
Then it follows by Lemma \ref{lem0114_01} that
$$
\int_0^{\infty}|\cL[h](\lambda^{-1})|^p\frac{W(\lambda)}{\lambda^{p+1}}\mathrm{d}\lambda\lesssim \int_0^{\infty}\left|\int_0^{\lambda}\cM (h\mathbbm{1}_{\bR_+})(s)\mathrm{d}s\right|^p\frac{W(\lambda)}{\lambda^{p+1}}\mathrm{d}s.
$$
Also, by Lemma \ref{trace lem1} and weighted $L_p$ estimates for the Hardy-Littlewood maximal functions, 
\begin{align*}
\int_0^{\infty}\left|\int_0^{\lambda}\cM (h\mathbbm{1}_{\bR_+})(s)\mathrm{d}s\right|^p\frac{W(\lambda)}{\lambda^{p+1}}\mathrm{d}\lambda &\lesssim_{p, [w]_{A_p}} \int_0^{\infty}|\cM(h\mathbbm{1}_{\bR_+})(s)|^pw(s)\mathrm{d}s\\
&\lesssim_{p, [w]_{A_p}} \|h\|_{L_p(\bR_+,w\,\mathrm{d}t)}^p,
\end{align*}
and this implies \eqref{22.11.08.14.24}.
Hence the theorem is proved.

\subsubsection{Proof of Theorem~\ref{extension_non_local}}

As noted in the proof of Theorem \ref{trace_non_local}, the space $(X_0,X_1)_{(W\circ\psi)^{1/p},p}$ is well-defined. Since 
$$a \in (X_0,X_1)_{(W\circ\psi)^{1/p},p} = (X_0,X_1)_{(W\circ\psi)^{1/p},p}^J\,,
$$
there exists a measurable function $v:\bR_+\to X_0\cap X_1$ such that
$$
a=\int_0^{\infty}v(\lambda)\frac{\mathrm{d}\lambda}{\lambda} \quad \textrm{and} \quad \Phi_{p}^{(W\circ\psi)^{1/p}}\left(J\left( \cdot,v\left( \cdot \right);X_0,X_1 \right) \right) \leq2\|a\|_{(X_0,X_1)_{(W\circ\psi)^{1/p},p}}.
$$
Recall that $\Theta(t,\lambda) = \lambda\int_0^{\infty}\mathrm{e}^{-\lambda r}\bP(S_r\geq t)\mathrm{d}r$ for $(t, \lambda) \in [0, \infty) \times \bR_+$. Let
$$
u(t):=\int_0^{\infty}\Theta(t,\lambda)v(\lambda)\frac{\mathrm{d}\lambda}{\lambda},\quad f(t):=-\int_0^{\infty}\lambda \Theta(t,\lambda)v(\lambda)\frac{\mathrm{d}\lambda}{\lambda}
$$
for $t \in [0, \infty)$ and $t \in \bR_+$, respectively. Since $\Theta(0, \lambda) = 1$ for $\lambda \in \bR_+$, it is clear that $u(0) = a$. We also have
\begin{align}\label{23.02.22.20.56}
\|u(t)\|_{X_0}+\|f(t)\|_{X_1}\lesssim \int_0^{\infty}\Theta(t,\lambda)J(\lambda,v(\lambda);X_0,X_1)\frac{\mathrm{d}\lambda}{\lambda}.
\end{align}
Note that the right-hand side of \eqref{23.02.22.20.56} is a form of $\int_0^\infty \Theta(t,\lambda) f(\lambda)\frac{\mathrm{d}\lambda}{\lambda}$.
For this quantity, we assume the following lemma for a moment:
\begin{lem}
\label{extension lem2}
Let $p\in(1,\infty)$ and $w\in A_p$. Suppose that $\phi \in \cI_o(0, 1)$ and $W\circ\psi\in \cI_o(0,p)$. Then for $f \in L_p(\bR_+, (W \circ \psi)(t^{-1}) \frac{\mathrm{d}t}{t})$, we have
\begin{align*}
    \int_0^{\infty}|F(t)|^pw(t)\mathrm{d}t \lesssim_{N} \int_0^{\infty}|f(t)|^p(W\circ\psi)(t^{-1})\frac{\mathrm{d}t}{t},
\end{align*}
where $F(t)=\int_0^{\infty} \Theta(t,\lambda) f(\lambda) \,\frac{\mathrm{d}\lambda}{\lambda}$ and $N = N(p, [w]_{A_p}, W \circ \psi, \kappa)$.
\end{lem}
Then by Lemma \ref{extension lem2} with $J(\lambda,v(\lambda);X_0,X_1)$ in place of $f(\lambda)$,
\begin{align*}
    &\|u\|_{L_p(\bR_+,w\,\mathrm{d}t;X_0)}+\|f\|_{L_p(\bR_+,w\,\mathrm{d}t;X_1)}\\
    &\lesssim \left\|\int_0^{\infty}\Theta(\cdot,\lambda)J(\lambda,v(\lambda);X_0,X_1)\frac{\mathrm{d}\lambda}{\lambda}\right\|_{L_p(\bR_+,w\,\mathrm{d}t)}\\
    &\lesssim \Phi_p^{(W\circ\psi)^{1/p}}\left(J\left( \cdot,v\left( \cdot \right);X_0,X_1 \right) \right) \lesssim\|a\|_{(X_0,X_1)_{(W\circ\psi)^{1/p},p}}.
\end{align*}
Thus $u$ and $f$ are well-defined in $L_p(\bR_+,w\,\mathrm{d}t;X_0)$ and $L_p(\bR_+,w\,\mathrm{d}t;X_1)$, respectively.
Finally, by Fubini's theorem (\textit{e.g.} \cite[Proposition 1.2.7.]{hytonen2016analysis}) and Proposition \ref{22.11.07.13.23}-($ii$), it directly follows that \eqref{eq0129_01}.

We prove Lemma~\ref{extension lem2} and complete the proof of Theorem~\ref{extension_non_local}.
Note that it is a variant of Lemma \ref{extension lem1} in the sense that one can recover Lemma~\ref{extension lem1} by choosing $\phi(t) = \psi(t) = t$ and $\Theta(t, \lambda) = e^{-t\lambda}$ in Lemma~\ref{extension lem2}.

\begin{proof}[Proof of Lemma~\ref{extension lem2}]
Without loss of generality, we assume that $f$ is nonnegative. First we divide $F$ into two parts; As in the proof of Theorem \ref{trace_non_local},  by using \eqref{2302231059},
$$
F(t) \simeq_{\kappa} \int_0^{\infty}\Theta(t,\phi(\lambda^{-1}))(f\circ\phi)(\lambda^{-1})\frac{\mathrm{d}\lambda}{\lambda}= \int_0^{t}\cdots+\int_{t}^{\infty}\cdots=:F_1(t)+F_2(t).
$$
By Proposition \ref{22.11.07.13.23}, we know that $\Theta (t, \phi(\lambda^{-1})) \simeq 1 \wedge \frac{\phi(t^{-1})}{\phi(\lambda^{-1})}$ and then
$$
F_1(t)\lesssim_{\kappa} \phi(t^{-1})\int_0^{t} \frac{(f\circ\phi)(\lambda^{-1})}{\phi(\lambda^{-1})}\frac{\mathrm{d}\lambda}{\lambda},\quad F_2(t)\lesssim_{\kappa} \int_{t}^{\infty}(f\circ\phi)(\lambda^{-1})\frac{\mathrm{d}\lambda}{\lambda}.
$$
On the other hand, by changing variable $t \to \phi(t)$ with the help of \eqref{2302231059} again, we have
$$
\int_0^{\infty}|f(t)|^p(W\circ\psi)(t^{-1})\frac{\mathrm{d}t}{t} \simeq_{\kappa} \int_0^{\infty}|(f\circ\phi)(t)|^p\frac{W(t^{-1})}{t}\mathrm{d}t.
$$

\textbf{Estimate of $F_1$}: We claim that
\begin{align}
\label{f1 est2}
    \int_0^{\infty}|F_1(t)|^pw(t)\mathrm{d}t&\leq C\int_0^{\infty}|(f\circ\phi)(t)|^pW(t^{-1})\,\frac{\mathrm{d}t}{t} \nonumber\\
    &= C\int_0^{\infty}|(f\circ\phi)(t^{-1})|^pW(t)\, \frac{\mathrm{d}t}{t}
\end{align}
for some $C = C(p, W \circ \psi, [w]_{A_p}) > 0$, and by Theorem \ref{two weight hardy}-($i$), it is enough to obtain \eqref{two weight 1} for
$$
U(t):=\phi(t^{-1})w(t)^{1/p}\quad \textrm{and} \quad V(t):=W(t)^{1/p}\phi(t^{-1})t^{1-1/p}.
$$
Since $\phi(t^{-1})^{p}=p\int_0^{\phi(t^{-1})}s^{p-1}\,\mathrm{d}s$, we have
\begin{align*}
    \int_r^{\infty}|U(t)|^p~\mathrm{d}t
    &=p\int_0^{\phi(r^{-1})}\left(\int_{r}^{\psi(s^{-1})}w(t)\mathrm{d}t\right)s^{p-1}~\mathrm{d}s
    \\
    &\leq p\int_0^{\phi(r^{-1})}(W\circ\psi)(s^{-1})s^{p-1}~\mathrm{d}s.
\end{align*}
By the assumption that $W \circ \psi \in \cI_o(0, p)$, we have
$$
\frac{\left( W \circ \psi \right) \left( s^{-1} \right) }{\left( W \circ \psi \right) \left( \frac{1}{\phi \left( r^{-1} \right)} \right)}
 \lesssim_{W \circ \psi} \left(  \frac{\phi(r^{-1})}{s} \right)^{p - \varepsilon}
$$
for $r \in \bR_+$, $s \in (0, \phi(r^{-1}))$ and $\varepsilon = \varepsilon(W \circ \psi) > 0$. Then
$$
\int_0^{\phi(r^{-1})}(W\circ\psi)(s^{-1})s^{p-1}~\mathrm{d}s
\lesssim_{W \circ \psi} \phi(r^{-1})^{p}W(r)
$$
 and this implies
\begin{align}
\label{22.11.07.16.33}
    \|U\|_{L_p((r,\infty))}\lesssim_{p, W \circ \psi}\phi(r^{-1})W(r)^{1/p}.
\end{align}
Similarly, by the assumption that $W \circ \psi \in \cI_o(0, p)$, for $ 0 < t < r$,
\begin{align*}
    \frac{1}{V(t)}&=\frac{t^{\frac{1}{p}-1}}{ \left( \left( W \circ \psi \right) \left( \frac{1}{\phi(t^{-1})} \right) \right)^{1/p}\phi(t^{-1})} \\
    &\lesssim_{p, [w]_{A_p}} \left( \phi \left( r^{-1} \right) \right)^{-1+\frac{\varepsilon}{p}} W(r)^{-1/p}   \phi(t^{-1})^{-\frac{\varepsilon}{p}} t^{\frac{1}{p} - 1}
\end{align*}
where $\varepsilon = \varepsilon(p, [w]_{A_p}) > 0$. From this,
\begin{align}
\label{22.11.07.18.39}
    \|1/V\|_{L_{p'}((0,r))}\lesssim_{p, [w]_{A_p}} \phi(r^{-1})^{-1}W(r)^{-1/p}.
\end{align}
Combining \eqref{22.11.07.16.33} and \eqref{22.11.07.18.39}, we have \eqref{two weight 1}, and thus \eqref{f1 est2}.

\textbf{Estimate $F_2$}: Now we prove that
\begin{align}
\label{f2 est2}
    \int_0^{\infty}|F_2(t)|^pw(t)\mathrm{d}t\leq C\int_0^{\infty}|(f\circ\phi)(t)|^p\frac{W(t^{-1})}{t}\mathrm{d}t
\end{align}
for some $C = C(p, [w]_{A_p}) > 0$, and by Theorem \ref{two weight hardy}-($ii$),  it is enough to obtain \eqref{two weight 2} for
$$
U(t):=w(t)^{1/p}\quad \textrm{and} \quad V(t):=W(t)^{1/p}t^{1-\frac{1}{p}}.
$$
Clearly, $\|U\|_{L_p((0,r))}= W(r)^{1/p}$.
For $0<r<t$, by Lemmas  \ref{22.10.26.15.52}-($ix$) and \ref{22.10.28.18.21},
$$
\frac{1}{V(t)}\lesssim_{p, [w]_{A_p}}W(r)^{-1/p}r^{\varepsilon/p}t^{\frac{(1-\varepsilon)}{p}-1},
$$
where $\varepsilon = \varepsilon(p, [w]_{A_p}) > 0$, and then
$$
\|U\|_{L_p((0,r))} \|1/V\|_{L_{p'}((r,\infty))} \lesssim_{p, [w]_{A_p}} 1.
$$
Now we obtain \eqref{f2 est2} and also desired inequality.
The lemma is proved.
\end{proof}

\begin{rem}
By keep tracking the proof of Theorem \ref{trace_non_local}, one verify that the trace estimate \eqref{trace 2} depend only on $p$, $M$, $\nu_l$, $\nu_u$, and $K_{\kappa}$, where $[w]_{A_p} \le M$, $\nu_l \le \kappa(1) \le \nu_u$ and $K_{\kappa}$ ($\kappa$ in place of $\phi$ in Remark \ref{rmk_bdd}).
Similarly, the inequality \eqref{eq0922_01} in Theorem \ref{extension_non_local} depend only on $p$, $M$, $\nu_l$, $\nu_u$, $K_{\kappa}$, and $K_{W \circ \kappa^{*}}$.
\end{rem}

\mysection{Trace theorems for finite time interval}
\label{23.08.09.13.27}

In this section, we assume that $X_0$ is continuously embedded into $X_1$, \textit{i.e.}, $X_0\subset X_1$ so that there exists a constant $N_{X_0,X_1}>0$ such that
\begin{align}\label{2308171037}
\|a\|_{X_1}\leq N_{X_0,X_1}\|a\|_{X_0}\quad\forall\,\,a\in X_0.
\end{align}
A representative example of this situation is when $X_0=W_p^2(\bR^d)$ and $X_1=L_p(\bR^d)$ with $N_{X_0,X_1}=1$.

As a corollary of Theorem \ref{trace_local}, we obtain a version of Theorem~\ref{trace_local} for finite time intervals $(0,T)$, $T<\infty$, instead of $\bR_+$.

\begin{corollary}[Trace theorem with local derivative; finite interval]
\label{2307061138}
Let $T\in(0,\infty)$, $p\in(1,\infty)$ and $w\in A_p$.
Assume that $X_0$ is continuously embedded into $X_1$, and $u\in L_p((0,T),w\dd t;X_0)$, $f\in L_p((0,T),w\dd t;X_1)$ and $u_0\in X_0+X_1$ satisfy that
$$
u(t)=u_0+\int_0^tf(s)\dd s \quad\forall t\in(0,T).
$$
Then $u_0\in (X_0,X_1)_{W^{1/p},p}$ and
$$
\|u_0\|_{(X_0,X_1)_{W^{1/p},p}}\lesssim_{p,[w]_{A_p},N_{X_0,X_1},T} \|u\|_{L_p((0,T),w\dd t;X_0)}+\|f\|_{L_p((0,T),w\dd t;X_1)}\,,
$$
where $W(t):=\int_0^tw(s)\dd s$.
\end{corollary}
\begin{proof}
Take a nonnegative decreasing $\zeta\in C^{\infty}(\bR)$ such that
$\zeta(s)= 1$ on $s\leq \frac{T}{2}$ and $\zeta(s)=0$ for all $s\geq T$.
Put $\widetilde{u}:=\zeta u$ and $\widetilde{f}:=\zeta f+\zeta'u$.
Since $X_0$ is continuously embedded into $X_1$,
\begin{equation*}
    \begin{gathered}
        \|\widetilde{u}\|_{L_p(\bR_+,w\dd t;X_0)}\leq \|u\|_{L_p((0,T),w\dd t;X_0)},\\ \|\widetilde{f}\|_{L_p(\bR_+,w\dd t;X_1)} \leq \|f\|_{L_p((0,T),w\dd t;X_1)}+N(N_{X_0,X_1},T)\|u\|_{L_p((0,T),w\dd t;X_0)}.
    \end{gathered}
\end{equation*}
Moreover
$$
\widetilde{u}(t)=u_0+\int_0^t\widetilde{f}(s)\dd s \quad\forall t\in\bR_+.
$$
By applying Theorem \ref{trace_local} for $(u,f)$ replaced by $(\widetilde{u},\widetilde{f})$, the corollary is proved.
\end{proof}

Similar to Corollary \ref{2307061138}, in the situation that $X_0$ is continuously embedded into $X_1$, we can also consider a version of Theorem \ref{trace_non_local} for finite time intervals $(0,T)$ instead of $\mathbb{R}_+$. 
However, the non-local derivative $\partial_t^{\kappa}$ makes the proof more intricate than the straightforward computation seen in Corollary \ref{2307061138}.
In particular, for finite time interval $(0,T)$ cases, it is natural to consider a kernel $\kappa^\circ$ and a weight function $w^\circ$ defined on $(0, T)$.
It turns out that if $w^{\circ} \in A_p\left( \left( 0, T \right) \right)$ (\textit{i.e.} $w^{\circ}$ satisfies \eqref{eq230925_01}) and $\kappa^{\circ} : (0, T) \to \bR_+$ satisfy 
$$
    \lambda^{-1+\varepsilon}\lesssim \frac{\kappa^{\circ}(\lambda t)}{\kappa^{\circ}(t)}\lesssim \lambda^{-\varepsilon}\quad \forall\lambda\in[1,\infty),\, 0 < t \le \lambda t < T.
$$
and
$$
\lambda^\delta\lesssim\frac{W^\circ\big({\kappa^{\circ}}^\ast(\lambda t)\big)}{W^\circ\big({\kappa^{\circ}}^\ast(t)\big)}\lesssim \lambda^{p-\delta} \quad \forall \, 0<t<\lambda t<\frac{1}{\kappa(T)}
$$
for some $\varepsilon$, $\delta > 0$, then there are proper extensions $w \in A_p(\bR)$ and $\kappa:\bR_+\to\bR_+$ of $w^{\circ} \in A_p\left( \left( 0, T \right) \right)$ and of $\kappa^{\circ}$, repectively.
In Proposition \ref{230925614}, we provide detailed descriptions of such extensions.
For simplicity, in the remaining of this section, we use $w \in A_p(\bR)$ and $\kappa:\bR_+\to\bR_+$ instead of $w^{\circ} \in A_p\left( \left( 0, T \right) \right)$ and $\kappa^{\circ}$, respectively.

We give a lemma required for the proof of Theorem~\ref{2307111235}.

\begin{lem}
	\label{lem2307170000}
Let $X_0$ be continuously embedded into $X_1$ and let $p\in(1,\infty)$, $w\in A_p$, $\kappa \in \cI_o(-1, 0)$, and $W \circ \kappa^{*} \in \cI_o(0, p)$.
Suppose $u\in L_p((0,T),w\,\mathrm{d}t;X_0)$, $f\in L_p((0,T),w\,\mathrm{d}t;X_1)$ and for some $u_0 \in X_1 \,(= X_0 + X_1)$, we have
    \begin{equation}
    \label{eq230720_02}
    \int_0^t \kappa(t-s) \left( u(s) - u_0 \right)\,\mathrm{d}s = \int_0^t f(s)\,\mathrm{d}s,
    \end{equation}
    for $t \in (0, T)$.
Then
\begin{align}\label{eq230720_021111}
\| u_0 \|_{X_1} \lesssim_{N} \|u\|_{L_p((0, T),w\,\mathrm{d}t;X_0)}+\|f\|_{L_p((0, T),w\,\mathrm{d}t;X_1)},
\end{align}
    where $N=N(p, [w]_{A_p}, \kappa,N_{X_0,X_1},T)$.
\end{lem}

\begin{proof}
By \eqref{eq230720_02}, 
\[
\left( \int_0^t \kappa(s) \,\mathrm{d}s \right) \| u_0 \|_{X_1}   \le \int_0^t \kappa(t-s) \|u(s)\|_{X_1} \,\mathrm{d}s + \int_0^t \| f(s) \|_{X_1} \,\mathrm{d}s
\]
for $t \in (0, T)$.
Then,
\begin{align}\label{2308171038}
\begin{split}
     &\| u_0 \|_{X_1} \int_0^{T} \mathrm{e}^{-t} \left( \int_0^t \kappa(s)\,\mathrm{d}s \right)\mathrm{d}t  \\
     \leq  &\int_0^{\infty} \mathrm{e}^{-t} \left( \kappa *U \right) \left( t \right) \,\mathrm{d}t +  \int_0^{\infty} \mathrm{e}^{-t} \left( \int_0^tF(s)\,\mathrm{d}s \right)\mathrm{d}t\\
     = &\cL[\kappa](1) \int_0^{\infty} \mathrm{e}^{-s} U(s) \,\mathrm{d}s + \int_0^{\infty} \mathrm{e}^{-s} F(s) \,\mathrm{d}s,
    \end{split}
\end{align}
where $U(s) =  \|u(s)\|_{X_1}1_{(0, T)}(s)$ and $F(s) =  \| f(s) \|_{X_1}1_{(0, T)}(s)$ for $s \in \bR_+$.
For
$$
w^* := w^{-1/(p-1)} \in A_{p/(p-1)},
$$
we have
\begin{align}\label{2308171039}
\left| \int_0^{\infty} \mathrm{e}^{-s} U(s) \,\mathrm{d}s \right| \le \left( \int_0^{\infty} \mathrm{e}^{-sp/(p-1)} w^*(s) \,\mathrm{d}s \right)^{(p-1)/p} \left(  \int_0^{\infty} \left| U(s) \right|^p w(s)\,\mathrm{d}s\right)^{1/p}.
\end{align}
Note that
\begin{align*}
    \int_0^{\infty} \mathrm{e}^{-sp/(p-1)} w^*(s) \,\mathrm{d}s &\simeq_{p} \int_0^{\infty} \int_s^{\infty} \mathrm{e}^{-rp/(p-1)} w^*(s) \,\mathrm{d}r\,\mathrm{d}s\\
&= \int_0^{\infty} \left( \int_0^{r}  w^*(s)\,\mathrm{d}s \right) \mathrm{e}^{-rp/(p-1)}\,\mathrm{d}r\\
&= \int_0^{\infty} W^*(r) \mathrm{e}^{-rp/(p-1)}\,\mathrm{d}r \lesssim_{[w]_{A_p}} 1,
\end{align*}
where 
$$
W^*(r) := \int_0^{r}  w^*(s)\,\mathrm{d}s.
$$
The last inequality is due to Lemma \ref{22.10.28.18.21}, in particular, $W^* \in \cI_o(0, p/(p-1))$.
Due to \eqref{2308171037}, we have
$$
 \int_0^{\infty} \mathrm{e}^{-s} U(s) \,\mathrm{d}s  \lesssim_{N} \|u\|_{L_p((0, T),w\,\mathrm{d}t;X_0)}\,,
$$
where $N=N(p, [w]_{A_p}, N_{X_0,X_1})$.
Similarly, 
\begin{align}\label{23081710310}
 \int_0^{\infty} \mathrm{e}^{-s} F(s) \,\mathrm{d}s  \lesssim_{p, [w]_{A_p}} \|f\|_{L_p((0, T),w\,\mathrm{d}t;X_1)}.
\end{align}
Finally, by the assumption $\kappa \in \cI_o(-1, 0)$ and \eqref{w scaling},
\begin{align}\label{23081710311}
\| u_0 \|_{X_1} \lesssim_{\kappa,T}  \int_0^{T} \mathrm{e}^{-t} \left( \int_0^t \kappa(s)\,\mathrm{d}s \right)\mathrm{d}t \, \| u_0 \|_{X_1}
\end{align}
and then, by combining \eqref{2308171038} - \eqref{23081710311}, we obtain \eqref{eq230720_021111}.
The lemma is proved.
\end{proof}

\begin{thm}[Trace theorem with non-local derivative; finite interval]
\label{2307111235}
Let $X_0$ be continuously embedded into $X_1$ and let $p\in(1,\infty)$, $w\in A_p$, $\kappa \in \cI_o(-1, 0)$, and $W \circ \kappa^{*} \in \cI_o(0, p)$.
Suppose $u\in L_p((0,T),w\,\mathrm{d}t;X_0)$, $f\in L_p((0,T),w\,\mathrm{d}t;X_1)$ and for some $u_0 \in X_1\,(= X_0 + X_1)$, we have
    \begin{equation}
    \label{eq230710_01}
    \int_0^t \kappa(t-s) \left( u(s) - u_0 \right)\,\mathrm{d}s = \int_0^t f(s)\,\mathrm{d}s,
    \end{equation}
    for $t \in (0,T)$.
Then we have $u_0 \in (X_0,X_1)_{(W\circ\kappa^{*})^{1/p},p}$ and
    \begin{equation}
        \label{trace 3}
    \begin{aligned}
        \|u_0\|_{(X_0,X_1)_{(W\circ\kappa^{*})^{1/p},p}} \lesssim_N\|u\|_{L_p((0, T),w\,\mathrm{d}t;X_0)}+\|f\|_{L_p((0, T),w\,\mathrm{d}t;X_1)},
    \end{aligned}
    \end{equation}
    where $N=N(p, [w]_{A_p}, \kappa,N_{X_0,X_1},T)$.
\end{thm}
\begin{proof}
Thanks to Proposition \ref{23.08.11.15.13}, we have an extension $\kappa \in \cI_o(-1, 0)$ of $\kappa^{\circ}$.
Through zero-extension, we consider $u$ and $f$ as functions defined on $\bR_+$, so that 
$$
u(t)=f(t)=0\quad \forall \,\,t\geq T.
$$
From \eqref{eq230710_01}, we have
\[
\left( \int_0^t \kappa(s) \,\mathrm{d}s \right) u_0 = \int_0^t \kappa(t-s) u(s) \,\mathrm{d}s - \int_0^t f(s)\,\mathrm{d}s
\]
for $t \in (0, T)$.
Then by integration by parts with the fact that $\mathcal{L}[\kappa](\lambda) = \phi(\lambda)/\lambda$,
\begin{align*}
    u_0 \frac{\phi(\lambda)}{\lambda^2}  &=  \int_0^{\infty} \mathrm{e}^{-\lambda t} \left( \int_0^t \kappa (s)\,\mathrm{d}s \right)u_0 \, \mathrm{d}t\\
    &= \int_0^T \mathrm{e}^{-\lambda t} \left( \int_0^t \kappa(t-s) u(s) \,\mathrm{d}s\right) \,\mathrm{d}t\\
    &\quad- \int_0^T \mathrm{e}^{-\lambda t} \left( \int_0^t  f(s) \,\mathrm{d}s\right) \,\mathrm{d}t+ \left( \int_T^{\infty} \mathrm{e}^{-\lambda t} \left( \int_0^t \kappa (s)\,\mathrm{d}s \right) \, \mathrm{d}t \right)u_0 \\
    &:= I_1(\lambda) + I_2(\lambda) + I_3(\lambda).
\end{align*}
By the definition of $K$-functional,
$$
K(\tau, u_0; X_0, X_1) \le \frac{\lambda^2}{\phi(\lambda)} \left( \left\| I_1(\lambda) \right\|_{X_0} + \tau  \left\| I_2(\lambda) \right\|_{X_1} + \tau \left\| I_3(\lambda) \right\|_{X_1} \right)\quad \forall \tau,\lambda\in(0,\infty).
$$
If we take $0 < \tau := \phi(\lambda)$, \textit{i.e.},
\begin{equation}
	\label{eq2307141147}
\lambda = \phi^{-1}(\tau) = \frac{1}{\psi(\tau^{-1})}.
\end{equation}
Then
\begin{equation}\label{2307122141}
\begin{aligned}
	&\|u_0\|_{(X_0,X_1)_{(W\circ\psi)^{1/p},p}}^p
	\\
 = &\int_0^\infty (W\circ\psi)(\tau^{-1})  \left|K(\tau, u_0 ; X_0, X_1)\right|^p~\frac{\mathrm{d}\tau}{\tau}\\
	\lesssim &\int_0^\infty (W\circ\psi)(\tau^{-1}) \left|  \frac{\tau^{-1}}{\psi(\tau^{-1})^2} \left( \left\| I_1(1/\psi(\tau^{-1})) \right\|_{X_0} + \tau  \left\| I_2(1/\psi(\tau^{-1})) \right\|_{X_1} \right)  \right|^p~\frac{\mathrm{d}\tau}{\tau}\\
	& + \int_0^\infty (W\circ\psi)(\tau^{-1}) \left|  \frac{1}{\psi(\tau^{-1})^2} \left\| I_3(1/\psi(\tau^{-1})) \right\|_{X_1}  \right|^p~\frac{\mathrm{d}\tau}{\tau} \\
 := &A + B.
\end{aligned}
\end{equation}

\textbf{Estimate of $A$}: Note that
\begin{align*}
\| I_1(\lambda) \|_{X_0} &\leq  \int_0^{\infty} \mathrm{e}^{-\lambda t} \left( \int_0^t \kappa (t-s)\|u(s)\|_{X_0}\,\mathrm{d}s \right) \,\mathrm{d}t \\
&= \cL[\kappa](\lambda) \cL[\| u \|_{X_0}](\lambda) = \frac{\phi(\lambda)}{\lambda} \cL[\| u \|_{X_0}](\lambda)
\end{align*}
and
$$
\| I_2(\lambda) \|_{X_1} \le  \int_0^{\infty} \mathrm{e}^{-\lambda t} \left( \int_0^t \| f(s) \|_{X_1}\,\mathrm{d}s \right) \,\mathrm{d}t = \lambda^{-1} \cL[\| f\|_{X_1}](\lambda).
$$
Therefore,
\begin{align*}
    A\leq \int_0^{\infty}(W\circ\psi)(\tau^{-1})\left|\frac{\cL[\|u\|_{X_0}+\|f\|_{X_1}](1/\psi(\tau^{-1}))}{\psi(\tau^{-1})}\right|^p\frac{\mathrm{d}\tau}{\tau}
\end{align*}
Then by \eqref{22.11.08.14.24}, we obtain
$$
A \lesssim_{p, [w]_{A_p}, \kappa} \|u\|_{L_p((0, T),w\,\mathrm{d}t;X_0)}+\|f\|_{L_p((0, T),w\,\mathrm{d}t;X_1)}.
$$

\textbf{Estimate of $B$}: We prove only
\begin{align}
\label{23.08.11.20.50}
J^p := \int_0^\infty (W\circ\psi)(\tau^{-1}) \left|  \frac{1}{\psi(\tau^{-1})^2}  \left( \int_T^{\infty} \mathrm{e}^{-\frac{t}{\psi(\tau^{-1})}} \left( \int_0^t \kappa (s)\,\mathrm{d}s \right) \, \mathrm{d}t \right)  \right|^p~\frac{\mathrm{d}\tau}{\tau} \leq N,
\end{align}
where $N=N(p,[w]_{A_p},\kappa)$.
If we have \eqref{23.08.11.20.50}, then by Lemma \ref{lem2307170000},
$$
B \lesssim_{p, [w]_{A_p}, \kappa} \|u_0\|_{X_1} \lesssim_{N} \|u\|_{L_p((0, T),w\,\mathrm{d}t;X_0)}+\|f\|_{L_p((0, T),w\,\mathrm{d}t;X_1)}\,,
$$
where $N=N(p, [w]_{A_p}, \kappa, N_{X_0,X_1},T)$.
Note that 
\[
\int_0^t \kappa(s) \,\mathrm{d}s \leq \mathrm{e}\int_0^{\infty} \mathrm{e}^{-s/t} \kappa(s) \,\mathrm{d}s = \mathrm{e}\,t \phi(t^{-1}).
\]
Then by Minkowski's integral inequality, we have
\begin{align*}
    J &\lesssim  \left( \int_0^\infty (W\circ\psi)(\tau^{-1}) \left| \frac{1}{\psi(\tau^{-1})^2}  \left( \int_T^{\infty} \mathrm{e}^{-\frac{t}{\psi(\tau^{-1})}}t \phi(t^{-1}) \, \mathrm{d}t \right)  \right|^p~\frac{\mathrm{d}\tau}{\tau} \right)^{1/p}\\
&\leq \int_T^\infty \left( \int_0^{\infty}  \frac{  \mathrm{e}^{-\frac{t p}{\psi(\tau^{-1})}} (W\circ\psi)(\tau^{-1})}{\psi(\tau^{-1})^{2p}}\, \frac{\mathrm{d}\tau}{\tau} \right)^{1/p}  t \phi(t^{-1}) \, \mathrm{d}t.
\end{align*}
By changing variable $\tau \to \phi(\lambda)$ with the help of \eqref{2302231059},
$$
\int_0^{\infty}  \frac{  \mathrm{e}^{-\frac{t p}{\psi(\tau^{-1})}} (W\circ\psi)(\tau^{-1})}{\psi(\tau^{-1})^{2p}}\, \frac{\mathrm{d}\tau}{\tau}\lesssim\int_0^{\infty}  \lambda^{2p-1}  \mathrm{e}^{-\lambda t p} W(\lambda^{-1})\, \mathrm{d}\lambda \lesssim_{p, [w]_{A_p}} t^{-2p}W(t),
$$
where the last inequality is due to Lemma \ref{22.10.28.18.21}.
Thus, by changing variable $\lambda \to \phi^{-1}(\tau)$ again,
\begin{align*}
J &\lesssim \int_T^{\infty} \phi(t^{-1}) \left( W \left( t \right) \right)^{1/p}\, \frac{\mathrm{d}t}{t} \simeq \int_0^{\phi(T^{-1})}  \left( (W \circ \psi) \left( \tau^{-1} \right) \right)^{1/p} \,\mathrm{d}\tau \leq N(p, [w]_{A_p}, \kappa,T),
\end{align*}
where the last inequality is due to $\left( (W \circ \psi) \left( \tau^{-1} \right) \right)^{1/p} \in \cI_o(-1, 0)$.
By combining the above estimates for $A$ and $B$, we obtain \eqref{trace 3}.
The theorem is proved.
\end{proof}

\begin{rem}
Note that Theorem~\ref{2307111235} is independent of the choice of the extension $\kappa$.
More precisely, for two extensions $\kappa$ and $\tilde{\kappa}$ of $\kappa^{\circ}$ satisfying the assumptions in Theorem \ref{2307111235},
$$
(X_0,X_1)_{(W\circ\kappa^{*})^{1/p},p}=(X_0,X_1)_{(W\circ\tilde{\kappa}^{*})^{1/p},p}.
$$
This is an easy consequence of Proposition \ref{prop_230223}-($iv$).
\end{rem}

\vspace{1cm}

\textbf{Acknowledgements.}
J.-H. Choi has been supported by the National Research Foundation of Korea(NRF) grant funded by the Korea government(ME) (No.RS-2023-00237315).
J. B. Lee has been supported by the NRF grants funded by the Korea government(MSIT) (No.2021R1C1C2008252 and No. 2022R1A4A1018904).
J. Seo has been supported by a KIAS Individual Grant (MG095801) at Korea Institute for Advanced Study.
K. Woo has been supported by the NRF grant funded by the Korea government(MSIT) (No.2019R1A2C1084683 and No.2022R1A4A1032094).

\appendix
\mysection{Examples of generalized real interpolation spaces}
\label{23.08.09.13.53}
In this appendix, we present concrete examples of spaces $(X_0,X_1)_{\phi,p}$ when $X_0$ and $X_1$ are Sobolev or Besov spaces.
First, we recall the definitions of Sobolev and Besov spaces.
\begin{defn}
    Let $p\in(1,\infty)$, $q\in[1,\infty]$, $\gamma\in\bR$ and $w\in A_p$.
    The sets $H_p^{\gamma}(w)$ and $B_{p,q}^{\gamma}(w)$ are Banach spaces equipped with norms given by
    \begin{align*}
        &\|f\|_{H_p^{\gamma}(w)}:=\|(1-\Delta)^{\gamma/2}f\|_{L_p(w)},\\
        &\|f\|_{B_{p,q}^{\gamma}(w)}:=\|S_0f\|_{L_p(w)}+\left\|\left\{(2^{j\gamma} \|\Delta_jf\|_{L_p(w)}\right\}_{j\in\bN}\right\|_{\ell_q(\bN)},
    \end{align*}
    where
    \begin{align*}
        (1-\Delta)^{\gamma/2}f:=\cF^{-1}[(1+|\cdot|^2)^{\gamma/2}\cF[f]],\quad \Delta_jf:=\cF^{-1}[\Psi(2^{-j}\cdot)\cF[f]],\quad S_0:=1-\sum_{j\in\bN}\Delta_j.
    \end{align*}
    Here $\cF$ and $\cF^{-1}$ are Fourier and inverse Fourier transforms and $\Psi$ is a nonnegaive infinitely smooth function satisfying $supp(\Psi)\subseteq \{\xi\in\bR^d:1/2\leq|\xi|\leq2\}$ and
    $$
    \sum_{k\in\bZ}\Psi(2^{-j}\xi)=1\quad \forall \xi\in\bR^d\setminus\{0\}.
    $$
    For $\gamma\in\bR_+$, we also denote the homogeneous spaces $\mathring{H}_p^{\gamma}(w)$ and $\mathring{B}_{p,q}^{\gamma}(w)$
    \begin{align*}
        \|f\|_{\mathring{H}_p^{\gamma}(w)}:=\|(-\Delta)^{\gamma/2}f\|_{L_p(w)},\quad \|f\|_{\mathring{B}_{p,q}^{\gamma}(w)}:=\left\|\left\{(2^{j\gamma} \|\Delta_jf\|_{L_p(w)}\right\}_{j\in\bZ}\right\|_{\ell_q(\bZ)}.
    \end{align*}
\end{defn}
We also introduce a result of the generalized interpolation on $L_p$ spaces which yields Orlicz-type spaces $L_{\phi, p}$.
To describe the Orlicz-type spaces, we need a definition of decreasing rearrangements of $f$.
\begin{defn}
    Let $X$ be a measure space with a $\sigma$-finite positive measure $\mu$.
    \begin{enumerate}[(i)]
        \item For a measurable function $f$ on $X$, $d_f(\lambda) := \mu\big(\{ x\in X : |f(x)| > \lambda \}\big)$ is called the distribution function of $f$.

        \item For a measurable function $f$ on $X$, the function
        $$
        f^*:[0,\infty)\rightarrow [0,\infty),\quad f^\ast(t) := \inf\{s>0 : d_f(s) \leq t\}, \quad t\in [0,\infty).
        $$
    is called the decreasing rearrangement of $f$.

    \item For $p\in[1,\infty]$ and $\phi\in\cI_o(0,1)$, 
    $L_{\phi,p}$ denotes the set of all measurable functions satisfying
    $$
    \|f\|_{L_{\phi, p}} := \|\phi f^*\|_{L_p\left(\bR_+,\frac{\mathrm{d}t}{t}\right)}<\infty.
    $$
    \end{enumerate}
\end{defn}
We state the main result of this appendix.
\begin{prop}
\label{exs}
    Let $p,p_0,p_1\in[1,\infty]$, $q_0,q_1,q\in[1,\infty]$, $s,s_0,s_1\in\bR$, $w\in A_p$ and $\phi\in\cI_o(0,1)$.

    \begin{enumerate}[(i)]
    \item If $s_0\neq s_1$, then
        \begin{align*}
        (B_{p, q_0}^{s_0}(w), B_{p, q_1}^{s_1}(w))_{\phi, q} &= B_{p,q}^{\phi(s_0, s_1)}(w),\\
        (\mathring{B}_{p, q_0}^{s_0}(w), \mathring{B}_{p, q_1}^{s_1}(w))_{\phi, q} &= \mathring{B}_{p,q}^{\phi(s_0, s_1)}(w),
    \end{align*}
    where $B_{p,q}^{\phi(s_0, s_1)}(w)$ and $\mathring{B}_{p,q}^{\phi(s_0, s_1)}(w)$ are spaces equipped with norms given by
    \begin{align*}
        \| f\|_{B_{p,q}^{\phi(s_0, s_1)}(w)} &= \|S_0f\|_{L_p(\bR^d,w\,\mathrm{d}x)}+ \left\|\left\{ 2^{js_0  } \phi(2^{-j(s_0 - s_1)}) \|\Delta_j f\|_{L_p(w)}\right\}_{j\in\bN}\right\|_{\ell_q(\bN)}\,,\\
        \| f\|_{\dot{B}_{p,q}^{\phi(s_0, s_1)}(w)} &=  \left\|\left\{  2^{js_0 } \phi(2^{-j(s_0 - s_1)}) \|\Delta_j f\|_{L_p(w)}\right\}_{j\in\bZ}\right\|_{\ell_q(\bZ)}.
    \end{align*}
        \item If $s_0\neq s_1$, then
        \begin{align*}
    (H_p^{s_0}(w), H_p^{s_1}(w))_{\phi, q} &= B_{p,q}^{\phi(s_0, s_1)}(w),\\
    (\mathring{H}_p^{s_0}(w), \mathring{H}_p^{s_1}(w))_{\phi, q} &= \mathring{B}_{p,q}^{\phi(s_0, s_1)}(w).
\end{align*}
\item If we put $\widetilde{\phi}(t) := t \phi(t^{-1})$, then
    \begin{align}\label{lorentz 1 infty}
        (L_1, L_\infty)_{\phi, q} = L_{\widetilde{\phi}, q},\quad \widetilde{\phi}(t) := t \phi(t^{-1}).
    \end{align}
    Moreover for any $p_0,\,p_1\in[1,\infty]$ with $p_1\neq p_1$,
    \begin{align}\label{lorentz pq}
        (L_{p_0}, L_{p_1})_{\phi, q} = (L_1, L_\infty)_{\psi, q},\quad \psi(s) := s^{\theta_0} \phi(s^{\theta_1-\theta_0}),\quad \theta_i:=1-\frac{1}{p_i},\,i=0,1.
    \end{align}
    \end{enumerate}
\end{prop}
For verifying $(i)$, we need the following proposition:
\begin{prop}\label{thm interpolation}
Let $1\leq q, q_0, q_1 \leq\infty$ and $\sigma_0, \sigma_1 \in \mathbb{R}$ with $\sigma_0\neq \sigma_1$.
If $A$ is a Banach space, then we have
	\begin{align*}
		(\ell_{q_0}^{\sigma_0}(A), \ell_{q_1}^{\sigma_1}(A) )_{\phi, q} = \ell_q^{\phi(\sigma_0, \sigma_1)}(A),
	\end{align*}
where 
\begin{align*}
		\| a\|_{\ell_q^{\phi(\sigma_0, \sigma_1)}(A)}: = 
  \begin{cases}
      \bigg(\sum_{j\in \mathbb{Z}} \left( 2^{j\sigma_0} \phi \left( 2^{-j \left( \sigma_0 - \sigma_1 \right)} \right) \|a_j\|_A\right)^q\bigg)^{1/q}\quad&\textrm{if}\,\,\,\, q\in[1,\infty)\\[4mm]
      \sup_{j\in\bZ} \left(2^{j\sigma_0} \phi \left( 2^{-j \left( \sigma_0 - \sigma_1\right)} \right) \|a_j\|_A \right)\quad&\textrm{if}\,\,\,\,q=\infty
  \end{cases}
\end{align*}
\end{prop}

\begin{proof}[Proof of Proposition \ref{exs}]

    ($i$) This is a direct consequence of Proposition \ref{thm interpolation} with $A=L_p(w)$.
    For detailed arguments, see the proof of Theorem~\cite[Theorem 2.4.2/(a)]{triebel} with \cite[Theorem 1]{kurtz1979results}.

    ($ii$) First, we prove the homogeneous case,
$$
(\mathring{H}_p^{s_0}(w), \mathring{H}_p^{s_1}(w))_{\phi, q} = \mathring{B}_{p,q}^{\phi(s_0, s_1)}(w).
$$
For $f = f_0 + f_1$, $f_j \in \mathring{H}_p^{s_j}(w)$ with $j=0,1$, one can observe that 
	\begin{equation}\label{230706550}
\begin{aligned}
     \| \Delta_k f \|_{L_p(w)}
	    &\leq \|\Delta_k f_0\|_{L_p(w)}+\|\Delta_k f_1\|_{L_p(w)}\\
	    &=2^{-s_0k}\|M_{k,s_0} ((-\Delta)^{s_0/2}f_0)\|_{L_p(w)} +2^{-s_1k}\|M_{k,s_1} ((-\Delta)^{s_1/2}f_1)\|_{L_p(w)},
	\end{aligned}
 \end{equation}
	where $k\in\bZ$ and $M_{k,s}$ is defined by
	$$
	\cF[M_{k,s}g](\xi):=m_{s}(2^{-k}\xi)\cF[g](\xi):=\frac{\cF[\varphi](2^{-k}\xi)}{|2^{-k}\xi|^{s}}\cF[g](\xi).
	$$
	Since for all multi-index $\alpha$,
	$$
	|D^{\alpha}m_{s}(\xi)|\lesssim_{\alpha,s}\mathbbm{1}_{B_{2}\setminus B_{1/2}}(\xi),
	$$
	we have
	$$
	|D^{\alpha}(m_{s}(2^{-k}\cdot))(\xi)|\lesssim_{\alpha,s}2^{-k|\alpha|}\mathbbm{1}_{B_{2^{k+1}}\setminus B_{2^{k-1}}}(\xi)\lesssim |\xi|^{-|\alpha|}.
	$$
	Therefore by the weighted Mikhlin multiplier theorem (\textit{e.g.} \cite[Theorem 1]{kurtz1979results}), \eqref{230706550} implies that
	$$
	\| \Delta_k f \|_{L_p(w)}\lesssim 2^{-s_0k} \| f_0\|_{\mathring{H}_p^{s_0}(w)} + 2^{-s_1 k} \|f_1\|_{\mathring{H}_p^{s_1}(w)}
	$$
	which yields
	\begin{equation}\label{ineq 221109 2120}
    \begin{aligned}
		\| \Delta_k f \|_{L_p(w)} &\lesssim 2^{-s_0k} K(2^{k(s_0 - s_1)}, f ; \mathring{H}_p^{s_0}(w), \mathring{H}_p^{s_1}(w)).
	\end{aligned}
    \end{equation}
	With the help of \eqref{ineq 221109 2120} and Proposition \ref{22.10.27.16.40}-$(i)$, we have
\begin{equation}\label{230706605}
 \begin{aligned}
		\| f\|_{\mathring{B}_{p,q}^{\phi(s_0, s_1)}(w)} &:=
        \left\| \left\{ 2^{s_0k} \phi(2^{-(s_0 - s_1)k}) \| \Delta_k f \|_{L_p(w)}\right\}_{k\in\bZ}\right\|_{\ell_q(\bZ)}  \\
		&\lesssim  \left\| \left\{  \phi\big(2^{-(s_0 - s_1)k}\big) K\left(2^{\left( s_0 - s_1 \right)k}, f ; \mathring{H}_p^{s_0}\left( w \right), \mathring{H}_p^{s_1} \left( w \right)\right)  \right\}_{k\in\bZ}\right\|_{\ell_q(\bZ)}  \\
		&\simeq \|f\|_{ \left( \mathring{H}_p^{s_0} \left( w \right), \mathring{H}_p^{s_1}\left( w \right) \right)_{\phi, q}}.
	\end{aligned}
 \end{equation}
For the converse, we make use of $J$-method. Since
$$
\|\Delta_k f\|_{\mathring{H}_{p}^{s}(w)}=\|(2^{s(k-1)}M_{k-1,-s_i}+2^{sk}M_{k,-s_i}+2^{s(k+1)}M_{k+1,-s})(\Delta_kf)\|_{L_p(w)}\,,
$$
we have
    \begin{equation}\label{ineq 221109 2123}
	\begin{aligned}
		 J\left(2^{(s_0-s_1)k}, \Delta_k f; \mathring{H}_p^{s_0}(w), \mathring{H}_p^{s_1}(w)\right) \lesssim 2^{s_0k} \| \Delta_k f\|_{L_p(w)}.
	\end{aligned}
    \end{equation}
	By \eqref{ineq 221109 2123} and Proposition \ref{22.10.27.16.40}-$(ii)$, we have
\begin{equation}
\label{230706604}
\begin{aligned}
		&\|f\|_{(\mathring{H}_p^{s_0}(w), \mathring{H}_p^{s_1}(w))_{\phi, q}}\\
		&\lesssim \left\| \left\{ \phi(2^{-(s_0-s_1)k}) J\left(2^{(s_0 - s_1)k}, \Delta_k f; \mathring{H}_p^{s_0}(w), \mathring{H}_p^{s_1}(w)\right) \right\}_{k\in\bZ}\right\|_{\ell_q(\bZ)}\\
		&\lesssim    \left\| \left\{ 2^{s_0k} \phi(2^{-(s_0 - s_1)k}) \| \Delta_k f \|_{L_p(w)} \right\}_{k\in\bZ}\right\|_{\ell_q(\bZ)} \simeq \| f\|_{B_{p,q}^{\phi(s_0, s_1)}(w)}.
	\end{aligned}	       
\end{equation}
The homogeneous case is proved by \eqref{230706605} and \eqref{230706604}.
For the inhomogeneous case, based on the proof for the homogeneous case, it suffices to observe that additionally
\begin{equation*}
\begin{aligned}
    J(1, S_0 f; H_p^{s_0}(w), H_p^{s_1}(w)) &\lesssim\| S_0 f\|_{L_p(\bR^d,w\,\mathrm{d}x)} \\
    &\lesssim  K(1, f ; H_p^{s_0}(w), H_p^{s_1}(w) )\\
    &\lesssim \|f\|_{(H_p^{s_0}(w), H_p^{s_1}(w))_{\phi, q}}.
\end{aligned}
\end{equation*}
This follows from a modification of the proof for the homogeneous case.

    ($iii$) It is clear that \eqref{lorentz pq} follows from \eqref{lorentz 1 infty} and Theorem \ref{thm stability}; note that for $A_0=L_1$ and $A_1=L_\infty$, $L_{p_i}\in K(\theta_i)\cap J(\theta_i)$ (see \cite[Theorem 1.10.3/1]{triebel} for the case $0<\theta_i<1$).
    Therefore, we only prove \eqref{lorentz 1 infty}.
    We borrow the argument in \cite[Theorem 1.18.6]{triebel}.
    Let $f\in L_1 + L_\infty$. 
    By \cite[1.18.6/(9)]{triebel}, we have
    $$
        K(t, f; L_1, L_\infty) = \int_0^t f^*(\tau) \mathrm{d}\tau\geq tf^\ast (t),
    $$
    which yields
    \begin{align*}
        \|f\|_{(L_1, L_\infty)_{\phi, p}} 
        &:= \left\| \phi(\cdot^{-1}) K(\cdot, f; L_1, L_\infty)\right\|_{L_p\left(\bR_+,\frac{\mathrm{d}t}{t}\right)}
        \geq \left\| \widetilde{\phi} f^\ast\right\|_{L_p\left(\bR_+,\frac{\mathrm{d}t}{t}\right)}
        = \|f\|_{L_{\widetilde{\phi},p}}.
    \end{align*}
    For the converse inequality, 
    \begin{align*}
        \left\| \phi(\cdot^{-1}) K(\cdot, f;L_1,L_{\infty})\right\|_{L_p\left(\bR_+,\frac{\mathrm{d}t}{t}\right)}
        &= \left\| \tilde{\phi}\cdot \left(\int_0^1f^\ast(\cdot\tau)\dd \tau\right)\right\|_{L_p\left(\bR_+,\frac{\mathrm{d}t}{t}\right)}\\
        &\leq \int_0^1         \left\| \tilde{\phi}f^\ast(\cdot\tau)\right\|_{L_p\left(\bR_+,\frac{\mathrm{d}t}{t}\right)} \mathrm{d}\tau\\
        &=\int_0^1 \left\| \tilde{\phi}(\cdot\tau^{-1})f^\ast\right\|_{L_p\left(\bR_+,\frac{\mathrm{d}t}{t}\right)} \mathrm{d}\tau\\
        &\lesssim \int_0^1         \tau^{-1+\varepsilon}\left\| \tilde{\phi} f^\ast\right\|_{L_p\left(\bR_+,\frac{\mathrm{d}t}{t}\right)} \mathrm{d}\tau\simeq  \|f\|_{L_{\widetilde{\phi},p}}.
    \end{align*}
    Note that the last inequalities hold due to Proposition \ref{22.10.26.15.52}.-($vi$).
    The proposition is proved.
\end{proof}

We end this appendix by proving Proposition \ref{thm interpolation}.
\begin{proof}[Proof of Proposition \ref{thm interpolation}]
Without loss of generality, we assume that $\sigma_0>\sigma_1$ (see Proposition~\ref{prop_230223}-($iii$)).

\textbf{Step 1)} We first prove that 	\begin{align}\label{right ineq}
		(\ell_\infty^{\sigma_0}(A), \ell_\infty^{\sigma_1}(A) )_{\phi, q} \subset \ell_q^{\phi(\sigma_0, \sigma_1)}(A)\quad \forall q\,\,\in[1,\infty].
	\end{align}
For the notational convenience, put $\ell_\infty^{\sigma_0}(A) = X_0$ and $\ell_\infty^{\sigma_1}(A) = X_1$.
Note that 
	\begin{align}
		K(t, a; X_0, X_1) \simeq \sup_{k\in \mathbb{Z}} \min(2^{k\sigma_0}, t2^{k\sigma_1}) \|a_k\|_{A}
	\end{align}
 (see, \textit{e.g.}, \cite[1.18.2/(3)]{triebel}).
Then for $q\in[1,\infty)$, it follows that 
	\begin{equation}\label{ineq-221012 1340}
	\begin{aligned}
		\| a \|_{(X_0,\, X_1)_{\phi, q}}^q
		&= \int_0^\infty \left( \phi(t^{-1}) K(t, a;  X_0, X_1) \right)^q \frac{\mathrm{d}t}{t}\\
		&\simeq \int_0^\infty \left( \phi(t^{-1}) \sup_{k\in\mathbb{Z}} \left(\min(2^{k\sigma_0}, t2^{k\sigma_1}) \|a_k\|_{A} \right)\right)^q \frac{\mathrm{d}t}{t}\\
		&= \sum_{j\in\mathbb{Z}} \int_{2^{j(\sigma_0 - \sigma_1)}}^{2^{(j+1)(\sigma_0 - \sigma_1)}} \left( \phi(t^{-1}) \sup_{k\in\mathbb{Z}} \left(\min(2^{k\sigma_0}, t2^{k\sigma_1}) \|a_k\|_{A} \right)\right)^q
		\frac{\mathrm{d}t}{t}\\
		&\simeq \sum_{j\in\mathbb{Z}} \phi(2^{-j(\sigma_0 - \sigma_1)})^q   \sup_{k\in\mathbb{Z}} \left(\min(2^{k\sigma_0}, 2^{j(\sigma_0 - \sigma_1)}2^{k\sigma_1}) \|a_k\|_{A}\right)^q\\
		&\geq \sum_{j\in\mathbb{Z}} \phi(2^{-j(\sigma_0 - \sigma_1)})^q 2^{j\sigma_0 q} \|a_j\|_{A}^q.
	\end{aligned}
	\end{equation}
By \eqref{ineq-221012 1340} we have shown \eqref{right ineq} for $q\in[1,\infty)$.
A modification of \eqref{ineq-221012 1340} shows that \eqref{right ineq} holds for $q=\infty$.

\textbf{Step 2)} In this step, we prove that
	\begin{align}\label{left_ineq}
		\ell_q^{\phi(\sigma_0, \sigma_1)}(A) \subset (\ell_1^{\sigma_0}(A), \ell_1^{\sigma_0}(A))_{\phi, q}, \quad\forall q\in [1,\infty].
	\end{align}

We first consider the case $q\in[1,\infty)$.
For simplicity we put $\ell_1^{\sigma_0} (A) = Y_0$ and $\ell_1^{\sigma_1}(A) = Y_1$. 
Note that 
	\begin{align}
		K(t, a; Y_0, Y_1) \sim \sum_{k\in\mathbb{Z}} \min(2^{k\sigma_0}, t2^{k\sigma_1}) \|a_k\|_{A}
	\end{align}
 (see, \textit{e.g.}, \cite[Step 2 of proof of Theorem 1.18.2]{triebel}).
Then for $q\in[1,\infty)$, we have
	\begin{equation}\label{230706542}
\begin{aligned}
  &\int_0^\infty \left( \phi(t^{-1}) K(t, a;  Y_0, Y_1) \right)^q \frac{\mathrm{d}t}{t}\quad\left(=\| a \|_{(X_0,\, X_1)_{\phi, q}}^q\right)\\
		&\simeq \sum_{j\in\mathbb{Z}} \int_{2^{j(\sigma_0 - \sigma_1)}}^{2^{(j+1)(\sigma_0 - \sigma_1)}} \left( \phi(t^{-1}) \sum_{k\in\mathbb{Z}} \min(2^{k\sigma_0}, t2^{k\sigma_1}) \|a_k\|_{A}\right)^q \frac{\mathrm{d}t}{t}\\
		&\simeq \sum_{j\in\mathbb{Z}} \phi(2^{-j(\sigma_0 - \sigma_1)})^q 2^{j\sigma_0 p}  \left( \sum_{k\in\mathbb{Z}} \min(2^{(k-j)\sigma_0}, 2^{(k-j)\sigma_1}) \|a_k\|_{A}\right)^q\\
		&\lesssim  \sum_{j\in\mathbb{Z}} \phi(2^{-j(\sigma_0 - \sigma_1)})^q 2^{j\sigma_0 q}  \left( \bigg( \sum_{k \leq j}2^{(k-j)\sigma_0} \|a_k\|_{A} \bigg)^q+  \bigg(\sum_{k>j} 2^{(k-j)\sigma_1} \|a_k\|_{A}\bigg)^q \right)\\
		&=: I + II.
	\end{aligned}
 \end{equation}
We handle terms $I$ and $II$ separately.
Note that since $\phi\in\cI_o(0,1)$, there exists $\varepsilon\in(0,1/2)$ such that $\phi\in\cI(\varepsilon,1-\varepsilon)$, and therefore
\begin{align}\label{230706455}
\sup_{t>0}\left( \phi \left( \lambda t \right) / \phi \left( t \right) \right)\lesssim \lambda^\varepsilon+\lambda^{1-\varepsilon}.
\end{align}

$\bullet$ Case 1. Estimation of $I$.

Recall that we assume that $\sigma_0>\sigma_1$.
Take $\chi_0<\sigma_0$ sufficiently close to $\sigma_0$ such that
\begin{align}\label{230706500}
(\sigma_0 - \sigma_1)\varepsilon - (\sigma_0 - \chi_0) >0\,,
\end{align}
and observe that
\begin{equation}\label{ineq-221012 1317}
			\begin{aligned}
				I\,&=\sum_{j\in\mathbb{Z}} \phi(2^{-j(\sigma_0 - \sigma_1)})^q 2^{j\sigma_0 q}   \left( \sum_{k:k \leq j}2^{(k-j)\sigma_0} \|a_k\|_{A} \right)^q\\
				&=\sum_{j\in\mathbb{Z}} \phi(2^{-j(\sigma_0 - \sigma_1)})^q\left( \sum_{k:k \leq j}2^{k\sigma_0} \|a_k\|_{A} \right)^q\\
				&\leq \sum_{j\in\mathbb{Z}} \phi(2^{-j(\sigma_0 - \sigma_1)})^q   \left( \sum_{k:k \leq j}2^{k(\sigma_0 - \chi_0)q'} \right)^{q/q'} \left( \sum_{k:k \leq j} 2^{k\chi_0 q} \|a_k\|_{A}^q \right)\\
				& = \sum_{j\in\mathbb{Z}} \phi(2^{-j(\sigma_0 - \sigma_1)})^q 2^{j(\sigma_0 -\chi_0)q}  \sum_{k:k \leq j} 2^{k\chi_0 q} \|a_k\|_{A}^q \\
				&= \sum_{k\in\mathbb{Z}} 2^{k\chi_0 q} \|a_k\|_A^q \sum_{j:k\leq j} \phi(2^{-j(\sigma_0 - \sigma_1)})^q 2^{j(\sigma_0 -\chi_0)q},
			\end{aligned}
			\end{equation}
   where $q'=q/(q-1)$.
   Observe that
  \begin{equation}\label{ineq-221012 1305}
  \begin{aligned}
   &\sum_{j:k\leq j} \phi(2^{-j(\sigma_0 - \sigma_1)})^q 2^{j(\sigma_0 -\chi_0)q}\\
   &\leq 2^{k(\sigma_0 -\chi_0)q}\phi(2^{-k(\sigma_0 -\sigma_1)})^q\sum_{j:k\leq j} \bigg(\frac{2^{j(\sigma_0 -\chi_0)}\phi\big(2^{-j(\sigma_0 - \sigma_1)}\big)}{2^{k(\sigma_0 -\chi_0)}\phi\big(2^{-k(\sigma_0 -\sigma_1)}\big)}\bigg)^q.
   \end{aligned}
   \end{equation}
Due to \eqref{230706455} and that $k\leq j$,
			\begin{align}\label{ineq-221012 1303}
				\frac{2^{j(\sigma_0 -\chi_0)}\phi\big(2^{-j(\sigma_0 - \sigma_1)}\big)}{2^{k(\sigma_0 -\chi_0)}\phi\big(2^{-k(\sigma_0 -\sigma_1)}\big)} \lesssim  2^{(j-k)\left[(\sigma_0-\chi_0)-(\sigma_0-\sigma_1)\varepsilon\right]}
			\end{align}
		Together with \eqref{230706500} - \eqref{ineq-221012 1303}, we have
\begin{align}\label{left ineq 1}
				I \lesssim \sum_{k\in\mathbb{Z}} 2^{k\sigma_0 q} \phi(2^{-k(\sigma_0 -\sigma_1)})^q \|a_k\|_A^q.
			\end{align}
$\bullet$ Case 2. Estimation of $II$.
		
		In the same manner of \eqref{230706500} and \eqref{ineq-221012 1317},
  Take sufficiently small $\xi_1>\sigma_1$ such that
  $$
  (\sigma_0-\chi_1)-(\sigma_0-\sigma_1)(1-\varepsilon)=\varepsilon(\sigma_0-\sigma_1)-(\chi_1-\sigma_1)>0\,,
  $$
  and observe that
  \begin{equation}\label{230706520}
\begin{aligned}				II\,&=\sum_{j\in\mathbb{Z}} \phi(2^{-j(\sigma_0 - \sigma_1)})^q 2^{j\sigma_0 q}   \left( \sum_{k:k>j }2^{(k-j)\sigma_1} \|a_k\|_{A} \right)^q\\
&\simeq\sum_{k\in\mathbb{Z}} 2^{k\chi_1 q} \|a_k\|_A^q \sum_{j:j<k} \phi(2^{-j(\sigma_0 - \sigma_1)})^q 2^{j(\sigma_0 - \chi_1)q}\\
				&=\sum_{k\in\mathbb{Z}} 2^{k\sigma_0 q} \phi(2^{-k(\sigma_0 -\sigma_1)})^q \|a_k\|_A^q \sum_{j:j<k}\bigg( \frac{2^{j(\sigma_0 -\chi_1)}\phi(2^{-j(\sigma_0 - \sigma_1)})}{2^{k(\sigma_0 -\chi_1)}\phi(2^{-k(\sigma_0 -\sigma_1)})}\bigg)^q.
			\end{aligned}
\end{equation}
  Due to \eqref{230706455} and that $\sigma_0>\sigma_1$,
			\begin{align*}
				\frac{2^{j(\sigma_0 -\chi_1)}\phi(2^{-j(\sigma_0 - \sigma_1)})}{2^{k(\sigma_0 -\chi_1)}\phi(2^{-k(\sigma_0 -\sigma_1)})}
    \lesssim 
   2^{-(j-k)[(\sigma_0 - \sigma_1)(1-\varepsilon)-(\sigma_0-\chi_1)]}.
			\end{align*}
		Therefore the summation over $j:j<k$ in \eqref{230706520} is finite, hence it follows that
			\begin{align}\label{left ineq 2}
				II \lesssim \sum_{k\in\mathbb{Z}} 2^{k\sigma_0 q} \phi(2^{-k(\sigma_0 -\sigma_1)})^q \|a_k\|_A^q.
			\end{align}	
\eqref{left_ineq} follows from \eqref{left ineq 1} and \eqref{left ineq 2} for $q\in [1,\infty)$.
Modifications of \eqref{230706542}, \eqref{ineq-221012 1317}, and \eqref{230706520} shows that \eqref{left_ineq} holds for $q=\infty$.

\textbf{Step 3)} 
Note that $\ell_1^{\sigma_0}\subset \ell_{p_0}^{\sigma_0}(A)\subset \ell_\infty^{\sigma_0}(A)$ and $\ell_1^{\sigma_1}(A)\subset \ell_{p_1}^{\sigma_1}(A)\subset \ell_\infty^{\sigma_1}(A)$.
Together with \eqref{right ineq}, \eqref{left_ineq}, and the definition of generalized interpolation, we have
	\begin{equation}
	\begin{aligned}
		\ell_q^{\phi(\sigma_0, \sigma_1)}(A) 
		&\subseteq (\ell_1^{\sigma_0}(A), \ell_1^{\sigma_0}(A))_{\phi, q}\\
		&\subset (\ell_{p_0}^{\sigma_0}(A), \ell_{p_1}^{\sigma_0}(A))_{\phi, q}\\
		&\subseteq (\ell_\infty^{\sigma_0}(A), \ell_\infty^{\sigma_0}(A))_{\phi, q}
		\subseteq \ell_q^{\phi(\sigma_0, \sigma_1)}(A).
	\end{aligned}
	\end{equation}
The proposition is proved.
\end{proof}

\mysection{Extenstions from $(0, T)$ to $\bR_+$}
\label{appendix_finite}
In this section, we provide suitable extensions of $\kappa^{\circ}: (0, T) \to \bR_+$ and $w^\circ \in A_p \left( \left(0, T \right) \right)$ (\textit{i.e.} $w^{\circ}$ satisfies \eqref{eq230925_01}) to functions defined on $\bR_+$.
To prove the existence of such extensions (Proposition \ref{230925614}), we need the following lemmas.

\begin{lem}\label{2307061156}
	Let $\kappa^\circ:(0,T)\rightarrow \bR_+$ be a function satisfying
	\begin{align}\label{2307061155}
	s^{\circ}_{\kappa^\circ}(\lambda):=\sup_{\substack{0<t< (1\wedge \lambda^{-1})T}}\frac{\kappa^\circ(\lambda t)}{\kappa^\circ(t)}<\infty \quad\forall\lambda\in(0,\infty).
 \end{align}
If we put
	\begin{align}\label{2307061157}
	\kappa(t):=
	\begin{cases}
		\,\,\kappa^\circ(t)\quad &\textrm{for} \quad 0<t<T,\\[2mm]
		\,\,\inf\limits_{r<T}\kappa^\circ(r)s^{\circ}_{\kappa^\circ}(t/r)\quad &\textrm{for} \quad t\geq T,
	\end{cases}
	\end{align}
then
\begin{align}\label{230608938}
s^{\circ}_{\kappa^\circ}(\lambda)=s_{\kappa}(\lambda)\left(:=\sup_{t>0}\frac{\kappa(\lambda t)}{\kappa(t)}\right)\quad \forall \lambda\in(0,\infty).
\end{align}
In addition, if $\kappa^\circ$ is decreasing on $(0,T)$, then $\kappa$ is also decreasing on $\bR_+$.
\end{lem}

\begin{proof}
Clearly, $s^{\circ}_{\kappa^\circ}(\lambda)\leq s_{\kappa}(\lambda)$.
Thus, to prove \eqref{230608938}, we only need to show that for any $t,\,\lambda>0$, 
\begin{align}\label{230608943}
\kappa(\lambda t)\leq \kappa(t)s^{\circ}_{\kappa^\circ}(\lambda).
\end{align}
Note that if $s<T$, then for any $\lambda>0$, 
\begin{align}\label{2307041106}
\kappa(\lambda s)\leq \kappa^\circ(s)s_{\kappa^\circ}^\circ(\lambda).
\end{align}

$\bullet$ Case 1. $t<T$.

In this case, \eqref{230608943} follows immediately from \eqref{2307041106}.

$\bullet$ Case 2. $t\geq T$.

We first claim that for any $a,\,b>0$,
\begin{align}\label{230608957}
	s^{\circ}_{\kappa^\circ}(ab)\leq s^{\circ}_{\kappa^\circ}(a)s^{\circ}_{\kappa^\circ}(b).
\end{align}
Without loss of generality, we assume that $a\geq 1$ or $0<b\leq 1$.
Let $0<s<T$ and $0<abs<T$.
Then the assumption for $a$ and $b$ implies that $0<bs<T$.
Therefore we have 
\begin{align*}
\frac{\kappa^\circ(abs)}{\kappa^\circ(s)}=\frac{\kappa^\circ(abs)}{\kappa^\circ(bs)}\cdot \frac{\kappa^\circ(bs)}{\kappa^\circ(s)}\leq s^{\circ}_{\kappa^\circ}(a)s^{\circ}_{\kappa^\circ}(b)\,,
\end{align*}
which implies \eqref{230608957}.
Due to \eqref{2307041106} and \eqref{230608957}, we obtain that for any $0<r<T$,
$$
\kappa(\lambda t)=\kappa(r\cdot(\lambda t/r))\leq \kappa^\circ(r)s^{\circ}_{\kappa^\circ}\left(\frac{\lambda t}{r}\right)\leq \kappa^\circ(r)s^{\circ}_{\kappa^\circ}\left(\frac{t}{r}\right)\cdot s^{\circ}_{\kappa^\circ}(\lambda).
$$
By taking the infimum for $0<r<T$, we have
$$
\kappa(\lambda t)\leq \kappa(t)s^{\circ}_{\kappa^\circ}(\lambda).
$$
Therefore \eqref{230608943} is proved, thus \eqref{230608938} is also.

Next, we prove the second assertion.
To observe that $s^\circ_{\kappa^\circ}$ is decreasing, let $0<\lambda_1<\lambda_2<\infty$.
Then for any $t\in(0,T)$ satisfying $0<\lambda_2t<T$, by \eqref{2307041106},
$$
\kappa^\circ(\lambda_2t)\leq \kappa^\circ(\lambda_1t)\leq s^\circ_{\kappa^\circ}(\lambda_1)\kappa^\circ(t).
$$
This implies that $s^\circ_{\kappa^\circ}$ is decreasing on $\bR_+$.
Since $s_{\kappa^\circ}^\circ$ is decreasing on $\bR_+$,
$$
\kappa(t)\leq \kappa^\circ(s)s_{\kappa^\circ}^\circ(t/s)\leq\kappa^\circ(s)s_{\kappa^\circ}^\circ(1)=\kappa(s)\,,\quad 0<s\leq t<\infty. 
$$
Therefore, $\kappa$ is decreasing on $\bR_+$.
The lemma is proved.
\end{proof}

With the help of Lemma \ref{2307061156}, the following proposition allows us to find an extension $\kappa: \bR_+ \to \bR_+$ of $\kappa^{\circ}: (0, T) \to \bR_+$ such that $\kappa \in \cI_o(-1, 0)$.

\begin{lem}
\label{23.08.11.15.13}
Let $\kappa^\circ:(0,T)\to \bR_+$ be a right-continuous decreasing function.
The following are equivalent.
\begin{enumerate}[(i)]
\item There exists a right-continuous extension $\kappa:\bR_+\to\bR_+$ of $\kappa^{\circ}:(0,T)\to\bR_+$ such that $\kappa\in\cI_o(-1,0)$,
    $$
    \kappa(t)=\kappa^{\circ}(t)\quad \forall t\in(0,T)\quad \textrm{and}\quad s_\kappa(\lambda) \simeq s^\circ_{\kappa^\circ}(\lambda) \quad \forall \lambda \in (0, \infty).
    $$

\item There exists $\varepsilon\in(0,1)$ such that
    \begin{equation}
    \label{23.08.11.15.33}
    \lambda^{-1+\varepsilon}\lesssim \frac{\kappa^{\circ}(\lambda t)}{\kappa^{\circ}(t)}\lesssim \lambda^{-\varepsilon}\quad \forall\lambda\in[1,\infty),\, 0 < t \le \lambda t < T.
        \end{equation}
\end{enumerate}
\end{lem}

\begin{proof}
Suppose $(\textit{i})$ holds.
Then due to Lemma \ref{22.10.26.15.52}-$(vi)$, \eqref{23.08.11.15.33} holds.
Conversely, if \eqref{23.08.11.15.33} holds, then since $\kappa^{\circ}$ satisfies \eqref{2307061155}, by Lemma \ref{2307061156}, the extension $\kappa$ (defined in \eqref{2307061157}) exists satisfying \eqref{230608938}.
    There is no guarantee that the extension $\kappa$ of $\kappa^{\circ}$ in \eqref{2307061157} is right-continuous on $[T,\infty)$.
For the right continuity, we put
$$
\kappa_{\mathrm{\,r.c.}}(t):=\lim_{s\searrow t}\kappa(s)=\lim_{\lambda\nearrow 1}\kappa(t/\lambda).
$$
Then $\kappa_{\mathrm{\,r.c.}}$ is right-continuous, decreasing and an extension of $\kappa^{\circ}$.
Moreover, we have
$$
\left(\lim_{\lambda\nearrow 1}s_{\kappa}(\lambda)\right)^{-1}\kappa(t)\leq \kappa_{\mathrm{\,r.c.}}(t)\leq \kappa(t),
$$
(note that $\lim_{\lambda\nearrow 1}s_{\kappa}(\lambda)<\infty$) which implies that
$$
s_{\kappa_{\mathrm{r.c.}}}\simeq s_{\kappa}=s^\circ_{\kappa^\circ}.
$$
This implies that for $0<\lambda<1$,
    $$
    \frac{s_{\kappa_{\mathrm{r.c.}}}(\lambda)}{\lambda^{-1}}=\frac{s_{\kappa^{\circ}}^{\circ}(\lambda)}{\lambda^{-1}}\lesssim\lambda^{\varepsilon},
    $$
    thus $s_{\kappa_{\mathrm{r.c.}}}(\lambda)=o(\lambda^{-1})$ as $\lambda\to0$.
    Similarly, $s_{\kappa_{\mathrm{r.c.}}}(\lambda)=o(1)$ as $\lambda\to\infty$.
    Therefore, $\kappa_{\mathrm{r.c.}}\in\cI_o(-1,0)$.
    The lemma is proved.
\end{proof}

We recall the definition and properties of $A_p\left( \left( 0, T \right) \right)$.
For $p\in (1,\infty)$, by $A_p\left( \left(0, T \right) \right)$ we denote the set of all locally integrable function $w^\circ:(0,T)\rightarrow (0,\infty]$ satisfying
\begin{align}
\label{eq230925_01}
\sup_{0\leq a<b\leq T}\left(\frac{1}{b-a}\int_a^bw^\circ(t)\dd t\right)\left(\frac{1}{b-a}\int_a^bw^\circ(t)^{-\frac{1}{p-1}}\dd t\right)^{p-1}<\infty.
\end{align}
By $A_1\left( \left( 0,T \right) \right)$ we denote the set of all locally integrable function $w^\circ:(0,T)\rightarrow (0,\infty]$ satisfying
\begin{align*}
M(w^\circ\mathbbm{1}_{(0,T)})\lesssim w\quad \textrm{on} \quad (0,T),
\end{align*}
where $\cM$ is the maximal operator defined in Definition~\ref{23.08.08.16.02}.
The following properties are introduced in \cite[Propositions 2.9, 2.10 and Lemma 2.12]{KurkiMudarra2022}.
\begin{enumerate}[(1)]
	\item For $p\geq 1$, if $w^\circ\in A_p\left( \left(0, T \right) \right)$, then there exists $\gamma_0>0$ such that for any $\gamma\in (0,\gamma_0)$, $(w^\circ)^{1+\gamma}\in A_p\left( \left(0, T \right) \right)$.
	
	\item For $p\in (1,\infty)$, $w^\circ\in A_p\left( \left(0, T \right) \right)$ if and only if there exists $v_1,\,v_2\in A_1\left( \left(0, T \right) \right)$ such that $w^\circ=v_1\,v_2^{1-p}$.
	
	\item For $f\in L_{1,\mathrm{loc}}\left( \left(0, T \right) \right)$, if $0<\cM(f\mathbbm{1}_{(0,T)})<\infty$ a.e. on $\bR$, then for any $\varepsilon\in(0,1)$, $\left( \cM \left( f\mathbbm{1}_{\left( 0,T \right)} \right)\right)^\varepsilon\in A_1\left( \left(0, T \right) \right)$.
\end{enumerate}
Indeed, the 'if' part in the second property is not provided in \cite[Propositions 2.9]{KurkiMudarra2022}, but it follows from direct calculations.
Obviously, the above properties also hold even for $\bR$ instead of $(0,T)$ (see \cite[Chapter 9]{grafakos2014modern}).

\begin{lem}\label{230921600}
    Let $p \in (1, \infty)$ and $w^\circ\in A_p\left( \left(0, T \right) \right)$.
    Then for any small enough $\varepsilon>0$, there exists $w_\varepsilon\in A_p = A_p(\bR)$ such that
    $$
w_\varepsilon=w^\circ \quad \textrm{on} \quad (0,T)\quad \textrm{and}\quad w_\varepsilon(t)\simeq |t|^{-1+\varepsilon} \quad \textrm{if}\quad |t|\geq T+1.
    $$
\end{lem}

\begin{proof}
We use the aforementioned properties of $A_p\left( \left( 0, T \right) \right)$ (properties $(1)$--$(3)$ right above this lemma). 

Since $w^\circ\in A_p\left( \left( 0,T \right)\right)$, there exists $\gamma_0$ such that $(w^\circ)^{1+\gamma}\in A_p\left( \left(0, T \right) \right)$ for any $\gamma \in (0, \gamma_0)$.
Then for any $\gamma\in (0,\gamma_0)$, there exists $v_1,\,v_2\in A_1\left( \left(0, T \right) \right)$ such that
\begin{align}\label{230921558}
(w^\circ)^{1+\gamma}=v_1\,v_2^{1-p}.
\end{align}
Direct calculation implies that
\begin{align}\label{230921557}
\cM(v_1)\simeq v_1\quad \textrm{on}\quad (0,T)\quad \textrm{and}\quad \cM(v_1)(t)\lesssim |t|^{-1}\quad \textrm{if}\quad |t|\geq T+1\,,
\end{align}
In addition, we obtain that for $\overline{v}_2:=v_2\mathbbm{1}_{(0,T)}+\mathbbm{1}_{\bR\setminus (0,T)}$,
\begin{align}\label{230921556}
\cM(\overline{v}_2)\simeq v_2\quad \textrm{on}\quad (0,T)\quad \textrm{and}\quad \cM(v_2)\simeq 1\quad \textrm{on}\quad (T+1,\infty).
\end{align}
Here, the first inequality is due to the fact that if $t\in (0, T)$, then
\begin{align*}
\cM\big(\overline{v}_2\big)(t)\leq \cM\big(v_2\mathbbm{1}_{(0,T)}\big)(t)+1 \lesssim \cM\big(v_2\mathbbm{1}_{(0,T)}\big)(t)+\frac{1}{T}\int_0^Tv_2(s)\dd s\lesssim \cM\big(v_2\big)(t).
\end{align*}
The second inequality is due to that if $|t|\geq T+1$ and $a\leq t\leq b$, then
\begin{align*}
\frac{1}{b-a}\int_a^b\overline{v}_2(s)\dd s&\leq \frac{1}{b-a}\int_a^bv_2(s)\mathbbm{1}_{(0,T)}\dd s+1\lesssim \int_0^Tv_2(s)\dd s+1\lesssim 1.
\end{align*}
In addition, one can observe that $\cM(v_1)$ and $\cM(v_2)$ are finite almost everywhere, so that $\cM(v_1)^{1/(1+\gamma)}$ and  $\cM(v_2)^{1/(1+\gamma)}$ are $A_1(\bR)$-weights (see \cite[Theorem 9.2.7]{grafakos2014modern}).

Let $\varepsilon<\frac{\gamma_0}{1+\gamma_0}$ and put $\gamma:=\frac{\varepsilon}{1+\varepsilon}<\gamma_0$, so that $1/(1+\gamma)=1-\varepsilon$.
For this $\gamma$ and $v_1$, $v_2$ in the above, we set
\begin{align*}
W_\varepsilon:=\left[\cM \left(v_1 \right) \left(\cM \left(v_2 \right) \right)^{1-p}\right]^{1-\varepsilon}= \cM \left(v_1 \right)^{1-\varepsilon} \left[ \cM \left(v_2 \right)^{1-\varepsilon}\right]^{1-p}.
\end{align*}
Then $W_\varepsilon\in A_p(\bR)$, since $\cM(v_1)^{1/(1+\gamma)},\,\cM(v_2)^{1/(1+\gamma)}\in A_1(\bR)$.
Moreover, 
$$
W_\varepsilon\simeq w^\circ \quad \textrm{on}\quad (0,T)\quad \textrm{and}\quad W_\varepsilon(t)\simeq t^{-1+\varepsilon}\quad \textrm{if}\quad |t|\geq T+1\,,
$$
due to \eqref{230921558}, \eqref{230921557}, and \eqref{230921556}.
By taking
$$
w_\varepsilon=w^\circ\mathbbm{1}_{(0,T)}+W_\varepsilon\mathbbm{1}_{\bR\setminus (0,T)},
$$
the lemma is proved.

\end{proof}

Since we assume that $X_0 \subset X_1$ in Section \ref{23.08.09.13.27}, the explicit form of the extended function $\left( W\circ \kappa^\ast \right) \left( t \right)$ in the following proposition for large $t$ is not used in proving Theorem \ref{2307111235}, and $W\circ \kappa^\ast \in \cI_o(0, p)$ is sufficient.

\begin{prop}\label{230925614}
Let $p\in(1,\infty)$, $w^{\circ} \in A_p\left( \left( 0, T \right) \right)$ and  $\kappa^{\circ}:(0,T)\rightarrow \bR_+$ be a right-continuous decreasing function such that
\begin{equation}
	\label{eq230926_01}
\lambda^{-1+\varepsilon}\lesssim \frac{\kappa^{\circ}(\lambda t)}{\kappa^{\circ}(t)}\lesssim \lambda^{-\varepsilon},\quad \forall \,\, 0 < t \le \lambda t < T
\end{equation}
and
\begin{align}\label{230925628}
\lambda^\delta\lesssim\frac{W^\circ\big({\kappa^{\circ}}^\ast(\lambda t)\big)}{W^\circ\big({\kappa^{\circ}}^\ast(t)\big)}\lesssim \lambda^{p-\delta}, \quad \forall  \, 0<t<\lambda t<\frac{1}{\kappa(T)}
\end{align}
for some $\varepsilon$, $\delta > 0$, where 
$$
W^\circ(t):=\int_0^t w(s)\dd s\quad \textrm{and}\quad  {\kappa^{\circ}}^\ast(t):=\big(\kappa^\circ\mathbbm{1}_{(0,T)}\big)^\ast(t):=\big(\kappa^\circ\mathbbm{1}_{(0,T)}\big)(1/t).
$$
Then there are extensions $w \in A_p(\bR)$ and $\kappa\in \cI_o(-1,0)$ of $w^{\circ}$ and of $\kappa^{\circ}$ such that $\kappa$ is right-continuous and decreasing, and $W\circ \kappa^\ast\in \cI_o(0,p)$, where $W(t):=\int_0^tw(s)\dd s$.
\end{prop}

\begin{proof}
Due to Lemma \ref{23.08.11.15.13}, there exists an extension of $\kappa^\circ$, denoted by $\kappa$.
Note that $\kappa \in \cI(-1+\varepsilon,-\varepsilon)$ by definition of the class $\cI(-1+\varepsilon,-\varepsilon)$ and the assumption \eqref{eq230926_01} (and \eqref{230608938}).
For this $\kappa$, let $\phi(\lambda):=\lambda\int_0^\infty\mathrm{e}^{-\lambda t}\kappa (t) \dd t$ which is well-defined.
Due to Proposition~\ref{prop0129_01} and Remark~\ref{23.03.02.15.47},
\begin{align}\label{230925629}
\phi(\lambda)\simeq \kappa(\lambda^{-1})\quad \textrm{and}\quad \psi(\lambda):=\frac{1}{\phi^{-1}(1/\lambda)}\simeq \kappa^{\ast}(\lambda)\quad \textrm{for all}\quad \lambda>0\,,
\end{align}
which implies that $\phi\in \cI(\varepsilon,1-\varepsilon)$, and therefore $\psi\in \cI(\frac{1}{1-\varepsilon},\frac{1}{\varepsilon})$.
In addition, from \eqref{230925628} and \eqref{230925629}, one can observe that there exists small enough $c_0>0$ such that 
\begin{align}\label{230925637}
\lambda^\delta\lesssim \frac{W^\circ\left( \psi \left( \lambda t \right) \right)}{W^\circ\left( \psi \left( t \right) \right)}\lesssim \lambda^{p-\delta}\quad \textrm{for all}\quad 0<t\leq \lambda t<c_0.
\end{align}
On the other hand, since $w^{\circ} \in A_p\left( \left( 0,T \right) \right)$, there exists $\varepsilon'>0$ such that 
\begin{align*}
	\lambda^{\varepsilon'}\lesssim\frac{W^\circ(\lambda t)}{W^\circ(t)}\lesssim \lambda^{p-\varepsilon'},\quad \forall \, 0<t<\lambda t<T. 
\end{align*}
We can take $\varepsilon'$ to be small enough such that $\varepsilon'<\varepsilon \delta/2$ and then by applying Proposition~\ref{230921600} for this $\varepsilon'$ (instead of $\varepsilon$), we have an extension $w = w_{\varepsilon'} \in A_p(\bR)$ of $w^{\circ}$ such that
$$
W(t)\simeq |t|^{\varepsilon'}\quad \textrm{if} \quad t\geq T+1.
$$
Since $\psi\simeq \kappa^\ast$ and $W\in\cI_o(0,p)$ and due to Lemma \ref{22.10.26.15.52}-$(vi)$, the proof is completed if we prove that for some $\overline{\varepsilon}>0$,
$$
\lambda^{\overline{\varepsilon}}\lesssim \frac{W\left( \psi \left( \lambda t \right) \right)}{W\left( \psi \left( \lambda t \right) \right)}\lesssim \lambda^{p-\overline{\varepsilon}}\quad \textrm{for all} \quad 0<t\leq \lambda t<\infty.
$$
If $0<t\leq \lambda t\leq c_0$, there is nothing to prove because of \eqref{230925637}.
If $0<t\leq c_0\leq \lambda t$, then
$$
\frac{W\left( \psi \left( \lambda t \right) \right)}{W\left( \psi \left( t \right) \right)}\simeq \frac{\left( \psi \left( \lambda t \right) \right)^{\varepsilon'}}{W^\circ\left( {\psi} \left( t \right) \right)}
$$
so that
\begin{align}
\label{230921758}
\lambda^{\varepsilon'/(1-\varepsilon)}\lesssim \frac{\left( \psi \left( \lambda t \right) \right)^{\varepsilon'}}{\left( \psi \left( t \right) \right)^{ \varepsilon' }}\lesssim \frac{W\left( \psi \left( \lambda t \right) \right)}{W\left( \psi \left( t \right) \right)}\lesssim \frac{(\lambda t)^{\varepsilon'/\varepsilon}}{t^{p-\delta}}\lesssim \lambda^{p-\delta/2}\,,
\end{align}
where the last inequality is due to $1/\lambda\lesssim t\lesssim 1$ and $\varepsilon' < \varepsilon \delta /2$.
If $c_0 \leq t\leq \lambda t$, then
\begin{align}\label{230921759}
\lambda^{\varepsilon'/(1-\varepsilon)}\lesssim \frac{W\left( \psi \left( \lambda t \right) \right)}{W\left( \psi \left( t \right) \right)}\simeq \frac{\left( \psi \left( \lambda t \right) \right)^{\varepsilon'}}{\left( \psi \left( t \right)\right)^{\varepsilon'}}\lesssim \lambda^{\delta/2},
\end{align}
where the last inequality is due to $\varepsilon' < \varepsilon \delta /2$ and $\lambda \ge 1$.
Consequently, due to \eqref{230925637}--\eqref{230921759}, we have $W\circ \psi \in\cI_o(0,p)$, and the proof is completed.
\end{proof}

\bibliographystyle{acm}
\bibliography{TRACEreference}

\end{document}